\documentclass[a4paper,reqno]{amsart}
\usepackage{amssymb}

\usepackage{graphicx}
\usepackage{tikz}
\usepackage{braket}
\usepackage{mathrsfs}
\usepackage[all]{xy}
\usepackage{color}

\usepackage{enumerate, pifont, geometry}

\usepackage{here}

\usepackage{stmaryrd}
\usepackage{xcolor}
\usetikzlibrary{arrows,calc}
\usepackage{etex}

\newtheorem{theorem}{Theorem}[section]
\newtheorem{corollary}[theorem]{Corollary}
\newtheorem{proposition}[theorem]{Proposition}
\newtheorem{lemma}[theorem]{Lemma}
\newtheorem{definition-proposition}[theorem]{Definition-Proposition}
\theoremstyle{definition}
\newtheorem{definition}[theorem]{Definition}
\newtheorem{example}[theorem]{Example}
\newtheorem{claim}[theorem]{Claim}
\newtheorem{fact}[theorem]{Fact}

\newtheorem{problem}[theorem]{Problem}

\theoremstyle{remark}
\newtheorem{remark}[theorem]{Remark}

\usepackage{amscd}
\usepackage{mathtools}
\usepackage{pstricks}
\usepackage{lscape}
\usepackage{comment}

\newcommand{\Dd}{\mathbb{D}}         

\DeclareMathOperator{\proj}{\mathsf{proj}}
\DeclareMathOperator{\inj}{\mathsf{inj}}
\DeclareMathOperator{\ind}{\mathsf{ind}}
\DeclareMathOperator{\moduleCategory}{\mathsf{mod}} \renewcommand{\mod}{\moduleCategory}
\DeclareMathOperator{\Mod}{\mathsf{Mod}}

\newcommand{\kbproj}{\operatorname{K^b}(\proj A)}
\newcommand{\blossom}{^\text{\ding{96}}} 
\newcommand{\bsm}{\begin{smallmatrix}}
\newcommand{\esm}{\end{smallmatrix}}

\numberwithin{equation}{section}

\begin{document}

\title[Auslander--Reiten theory in extriangulated categories]{Auslander--Reiten theory in extriangulated categories}

\author{Osamu Iyama}
\address{Graduate School of Mathematical Sciences, University of Tokyo, 3-8-1 Komaba Meguro-ku Tokyo 153-8914, Japan}
\email{iyama@ms.u-tokyo.ac.jp}
\urladdr{https://www.ms.u-tokyo.ac.jp/~iyama/index.html}
\author{Hiroyuki Nakaoka}
\address{Graduate School of Mathematics, Nagoya University, Furocho, Chikusaku, Nagoya 464-8602, Japan}
\email{nakaoka.hiroyuki@math.nagoya-u.ac.jp}
\author{Yann Palu}
\address{LAMFA, Universit\'e de Picardie Jules Verne, 33 rue Saint-Leu, Amiens, France}
\email{yann.palu@u-picardie.fr}
\urladdr{http://www.lamfa.u-picardie.fr/palu/}

\thanks{
The authors wish to thank Ivo Dell'Ambrogio for interesting discussions.
We acknowledge some anonymous referee for his/her many valuable comments that helped improve readability of the article.
The first author is supported by JSPS Grant-in-Aid for Scientific Research (B) 16H03923, (B) 22H01113 and (S) 15H05738.
The second author is supported by JSPS KAKENHI Grant Numbers JP17K18727, JP20K03532. The third author is supported by the French ANR grant SC$^3$A (15 CE40 0004 01).}

\begin{abstract}
The notion of an extriangulated category gives a unification of existing theories in exact or abelian categories and in triangulated categories. In this article, we develop Auslander--Reiten theory for extriangulated categories. This unifies Auslander--Reiten theories developed in exact categories and triangulated categories independently.
We give two different sets of sufficient conditions on the extriangulated category so that existence of almost split extensions becomes equivalent to that of an Auslander--Reiten--Serre duality. We also show that existence of almost split extensions is preserved under taking relative extriangulated categories, ideal quotients, and extension-closed subcategories.
Moreover, we prove that the stable category $\underline{\mathscr{C}}$ of an extriangulated category $\mathscr{C}$ is a $\tau$-category \cite{I1} if $\mathscr{C}$ has enough projectives, almost split extensions and source morphisms. This gives various consequences on $\underline{\mathscr{C}}$, including Igusa--Todorov's Radical Layers Theorem \cite{IT}, Auslander--Reiten Combinatorics on dimensions of Hom-spaces, and Reconstruction Theorem of the associated completely graded category of $\underline{\mathscr{C}}$ via the complete mesh category of the Auslander--Reiten species of $\underline{\mathscr{C}}$.
Finally we prove that any locally finite symmetrizable $\tau$-quiver (=valued translation quiver) is an Auslander--Reiten quiver of some extriangulated category with sink morphisms and source morphisms.
\end{abstract}

\maketitle

\tableofcontents

\section*{Introduction}
Auslander--Reiten theory, initiated in \cite{AR1,AR2}, is a key tool to study the local structure of additive categories. Its generalizations have been studied by many authors, and many of them can be divided into two classes of additive categories.
The first one is the class of Quillen's exact categories \cite{GR}, including finitely generated modules over finite dimensional algebras \cite{ARS,ASS}, their subcategories \cite{AS,K,R} and Cohen--Macaulay modules over orders and commutative rings \cite{A3,RS,Y,LW}. The second one is the class of Grothendieck--Verdier's triangulated categories \cite{Ha1,RV2,Kr}, including the derived categories of finite dimensional algebras \cite{Ha2}, differential graded categories (e.g.\ \cite{Jo1,Sc}), and commutative and non-commutative schemes (e.g. \cite{AR4,GL,RV2}).
Also there are many references studying Auslander--Reiten theory in more general additive categories (e.g.\ \cite{I1,Liu,Ru2,Sh}).

Recently, the class of extriangulated categories was introduced in \cite{NP} as a simultaneous generalization of exact categories and triangulated categories. The aim of this paper is to develop a fundamental part of Auslander--Reiten theory for extriangulated categories. More explicitly, we introduce the notion of almost split extensions (Definition \ref{DefARExt}), the (co)stable categories $\underline{\mathscr{C}}$ and $\overline{\mathscr{C}}$ of extriangulated categories (Definition \ref{DefCoStable}) and Auslander--Reiten--Serre duality for extriangulated categories (Definition \ref{DefARDual}), and give explicit connections between these notions and also with the classical notion of dualizing $k$-varieties. Our main results can be summarized as follows.

\begin{theorem}[Theorems~\ref{main theorem} and~\ref{main theorem 2}]\label{theorem in intro}
Let $\mathscr{C}$ be a $k$-linear, Ext-finite, extriangulated category.
\begin{enumerate}[\rm(1)]
\item Assume that $\mathscr{C}$ is Krull--Schmidt. Then $\mathscr{C}$ has almost split extensions if and only if it has an Auslander--Reiten--Serre duality.
\item Assume that $\mathscr{C}$ has enough projectives and enough injectives. Then $\mathscr{C}$ has an Auslander--Reiten--Serre duality if and only if $\underline{\mathscr{C}}$ (resp.\ $\overline{\mathscr{C}}$) is a dualizing $k$-variety.
\end{enumerate}
\end{theorem}

This generalizes and strengthens the corresponding results mentioned above as well as recent results in \cite{LNP,Ji,E} (exact categories), \cite{J} (subcategories of triangulated categories) and \cite{ZZ} (extriangulated categories). We also refer to \cite{Niu,LNi} for recent results on subcategories of triangulated categories that do not assume Ext-finiteness.

We also study the stability of the existence of an Auslander--Reiten theory under various constructions: relative extriangulated categories (Proposition~\ref{PropoRelARExt}), ideal quotients (Proposition~\ref{PropQuotARExt}) and extension-closed subcategories (Theorem~\ref{CorSubARExt}). They can be regarded as generalizations of previous works by Auslander--Solberg \cite{ASo}, Auslander--Smal\o {} \cite{AS}, and so on.

\medskip
Auslander--Reiten theory clarifies a common categorical feature of various categories $\mathscr{C}$ in terms of the functor category of $\mathscr{C}$. We denote by $S_X={\rm top}\,\mathscr{C}(-,X)$ and $S^X={\rm top}\,\mathscr{C}(X,-)$ the simple $\mathscr{C}$- and $\mathscr{C}^\mathrm{op}$-modules corresponding to an indecomposable object $X\in\mathscr{C}$ respectively.
If $\mathscr{C}$ has almost split extensions, then each indecomposable non-projective object $C\in\mathscr{C}$ (resp.\ non-injective object $A\in\mathscr{C}$) has an almost split sequence $A\overset{x}{\longrightarrow}B\overset{y}{\longrightarrow}C$, which is a conflation such that
\begin{eqnarray*}
\mathscr{C}(-,A)\xrightarrow{x\circ-}\mathscr{C}(-,B)\xrightarrow{y\circ-}\mathscr{C}(-,C)\to S_C\to0\\
\mathscr{C}(C,-)\xrightarrow{-\circ y}\mathscr{C}(B,-)\xrightarrow{-\circ x}\mathscr{C}(A,-)\to S^A\to0
\end{eqnarray*}
are exact. 
In particular, if $X$ is non-projective (resp.\ non-injective), the first three terms of the minimal projective resolution of $S_X$ (resp.\ $S^X$) has a remarkable symmetry.

On the other hand, the characteristics of the category $\mathscr{C}$ appear more strongly in the minimal projective resolutions of $S_X$ (resp.\ $S^X$) for indecomposable projective (resp.\ injective) objects $X\in\mathscr{C}$.
In Section~\ref{section:sink and source}, we study them in terms of sink sequences (resp.\ source sequences) (see Definition~\ref{define sink sequence}), and we prove the following general result.

\begin{theorem}[Theorem~\ref{weak kernel of right almost split}]\label{theorem2 in intro}
Let $\mathscr{C}$ be a Krull--Schmidt extriangulated category with enough projectives and injectives. If $B\to A\to P$ is a sink sequence of an indecomposable projective object $P$, then $B$ is injective. Dually, if $I\to A\to B$ is a source sequence of an indecomposable injective object $I$, then $B$ is projective.
\end{theorem}

The most basic example of exact categories which has almost split extensions is given by the category of finitely generated modules over a finite dimensional algebra over a field. This example has the following two natural generalizations, where it is well-known that both classes of exact categories also have almost split extensions.
\begin{enumerate}
\item[(A$_d$)] the category ${}^\perp U$ for a cotilting $\Lambda$-module $U$ of injective dimension $d$ over a finite dimensional algebra $\Lambda$ over a field,
\item[(B$_d$)] the category $\mathsf{CM}\Lambda$ for an $R$-order $\Lambda$ which is an isolated singularity over a complete local Cohen--Macaulay ring $R$ of dimension $d$.
\end{enumerate}
In these exact categories, Theorem~\ref{theorem2 in intro} can be improved by using Auslander--Buchweitz approximation theory \cite{AB,AR5} (see Propositions~\ref{sink for cotilting} and \ref{sink for CM}).
This is also closely related to the Auslander correspondence given in Theorem~4.2.4 in \cite{I6} for the case $n=1$.

One of important consequences of classical Auslander--Reiten theory is that the \emph{Auslander--Reiten quiver} contains a lot of important information about the category $\mathscr{C}$. 
For example, the following result was first proved for the category of finitely generated modules over a finite dimensional algebra over a field \cite{Rie,BG,IT}.

\begin{enumerate}[$\bullet$]
\item (\emph{Reconstruction Theorem}) \cite[Theorem~9.2]{I1} Let $\mathscr{C}$ be a category in (A$_d$) with $d\le 1$ or (B$_d$) with $d\le 2$. Then the associated completely graded category of $\mathscr{C}$ with respect to the radical filtration is equivalent to the complete mesh category of the Auslander--Reiten species of $\mathscr{C}$.
\end{enumerate}

The validity of this result is closely related to the basic fact that, in such a category $\mathscr{C}$, the object $B$ in Theorem~\ref{theorem2 in intro} is always zero.
In \cite{I1,I2,I3}, an additive category enjoying this property is called a \emph{$\tau$-category} (see Definition~\ref{define tau category}) and studied in depth to give a characterization of the Auslander--Reiten quiver of the category in (B$_1$). The notion of $\tau$-categories can be regarded as a categorical counterpart of the notion of \emph{$\tau$-quivers} (=valued translation quivers, Definition \ref{define valued translation quiver}), see Example \ref{example of tau-category}(3).
Another source of $\tau$-categories is given by a triangulated category with almost split triangles, see Example \ref{example of tau-category}(2).
On the other hand, the category in (A$_d$) with $d\ge2$ or (B$_d$) with $d\ge 3$ are not $\tau$-categories since the sink (resp.\ source) sequences of these categories are not as nice as almost split sequences, see Propositions~\ref{sink for cotilting} and \ref{sink for CM}. This observation is one of the motivations to study cluster tilting subcategories in higher dimensional Auslander--Reiten theory \cite{I5}. 


The aim of Sections \ref{section_tau-categories} and \ref{section_inverse} is to apply the theory of $\tau$-categories to study extriangulated categories.
In Section~\ref{section_tau-categories}, we prove the following result, which was surprising to us since it asserts that the additive structure of the stable category $\underline{\mathscr{C}}$ and the costable category $\overline{\mathscr{C}}$ of an extriangulated category $\mathscr{C}$ is much nicer than that of $\mathscr{C}$.

\begin{theorem}[Theorem~\ref{stable categories are tau}]\label{Th0.2}
Let $\mathscr{C}$ be a Krull--Schmidt extriangulated category with enough projectives and injectives, sink morphisms and source morphisms.
Then $\overline{\mathscr{C}}$ and $\underline{\mathscr{C}}$ are $\tau$-categories.
\end{theorem}

As an application of Theorem~\ref{Th0.2}, we show that some important results in representation theory still hold in $\underline{\mathscr{C}}$ and $\overline{\mathscr{C}}$ for a large class of extriangulated categories $\mathscr{C}$. In fact, we prove
\begin{enumerate}[$\bullet$]
\item \emph{Radical Layers Theorem} (Corollary~\ref{radical layer}) originally due to Igusa-Todorov \cite{IT}, which gives exact sequences associated with almost split sequences,
\item \emph{Auslander--Reiten Combinatorics} (Corollary~\ref{calculate dimension}) originally due to Gabriel \cite{G}, which gives dimensions of Hom-spaces of $\underline{\mathscr{C}}$ (resp.\ $\overline{\mathscr{C}}$),
\item \emph{Reconstruction Theorem} (Corollary~\ref{gr=mesh}) originally due to Bongartz--Gabriel \cite{BG}, which gives an equivalence between the associated completely graded category of $\underline{\mathscr{C}}$ (resp.\ $\overline{\mathscr{C}}$) with respect to the radical filtration and the complete mesh category of its Auslander--Reiten species.
\end{enumerate}

One of the important applications of Auslander--Reiten Combinatorics is beautiful characterizations of $\tau$-quivers which are realized as the Auslander--Reiten quivers of a relevant class of categories. It was given for the class of categories in (B$_0$) in \cite{I4,Ru1,IT2,B,Rie2}, in (B$_1$) in \cite{I3,Ru3,W,Luo}, in (B$_2$) in \cite{RV1}, and for the category of Cohen--Macaulay dg modules over certain differential graded algebras in \cite{Jin}.

It is natural to ask the following Inverse Problem (see Problem \ref{inverse problem} for details): Can arbitrary $\tau$-quiver be realized as the Auslander--Reiten quiver of a Krull--Schmidt extriangulated category? In Section \ref{section_inverse}, we give the following positive answer.



\begin{theorem}[Theorem~\ref{answer for inverse}]
Let $Q$ be a locally finite symmetrizable $\tau$-quiver. Then there exists a Krull--Schmidt extriangulated category with sink morphisms and source morphisms whose Auslander--Reiten quiver is $Q$.
\end{theorem}

A key ingredient of our proof is Enomoto's classification of exact structure on additive categories \cite[Theorem 2.7]{E}.

\medskip
There are many important extriangulated categories which are neither exact nor triangulated.
One of such class is given by $\tau$-tilting theory, introduced by Adachi--Iyama--Reiten in~\cite{AIR}. This is an increasingly studied part of representation theory (see the introductory~\cite{IR} and references therein).
Indeed, it is at the same time combinatorially better-behaved than classical tilting theory and much more general than cluster-tilting theory.
When $A$ is a $2$-Calabi--Yau tilted algebra (resp.\ cluster-tilted algebra), then the isoclasses of support $\tau$-tilting modules over $A$ are in bijection with the isoclasses of cluster-tilting objects in a corresponding triangulated category (resp.\ cluster category).
For more general finite dimensional algebras $A$, there is a bijection between the isoclasses of support $\tau$-tilting modules over $A$ and $2$-term silting complexes in $\kbproj$.
A naive approach consists in thinking of the full subcategory $A\ast A[1]$ of $\kbproj$ as a replacement for a non-existing cluster category over $A$.
Since $A\ast A[1]$ is extension-closed in $\kbproj$, it inherits some structure from the triangulated structure of $\kbproj$: It is extriangulated. It is thus possible to apply the results in this paper in order to show that $A\ast A[1]$ has almost split extensions (Example~\ref{example: n-term}).

We propose a slightly different approach to $\tau$-tilting theory for some specific algebra $A$, by giving another construction of some replacement for a cluster category.
We construct an exact category $\mathscr{E}$ whose quotient $\mathscr{E}/\mathscr{B}$ by the ideal generated by projective-injective objects  might serve as a ``categorification'' of support $\tau$-tilting modules over $A$ (see Section~\ref{subsection: example3} for more details).
The isoclasses of indecomposable objects in $\mathscr{E}/\mathscr{B}$ are in bijection with the isoclasses of indecomposable objects in $A\ast A[1]$. Since $\mathscr{E}$ is not Frobenius, $\mathscr{E}/\mathscr{B}$ is not triangulated and we check that it is not exact either.
However, by~\cite[Proposition 3.30]{NP}, it is extriangulated so that it comes equipped with enough structure for applying our main results. In particular, it has Auslander--Reiten extriangles and an Auslander--Reiten--Serre duality. Notably, $\tau$-tilting mutation is given by approximation extriangles in the category $\mathscr{E}/\mathscr{B}$.

\medskip
Section~\ref{section_ARext} deals with the definitions and properties of almost split extensions. In Sections~\ref{section_ARtr} and~\ref{section_ARst}, assuming Ext-finiteness, we show that the existence of almost split extensions can be given by Auslander--Reiten--Serre duality. In Section~\ref{section_ARsubcat}, we study the stability of the existence of an Auslander--Reiten theory under various constructions. In Section~\ref{section:sink and source}, we study sink sequences of projective objects and source sequences of injective objects. Section~\ref{section_tau-categories} is devoted to the proof of Theorem~\ref{Th0.2}, and to drawing some of its consequences.
Section~\ref{section_inverse} is devoted to the proof of Theorem 0.4.
Finally, we give an example in Section~\ref{subsection: example3}.

\section{Preliminaries}

All categories are assumed (locally small and) essentially small with respect to a fixed Grothendieck universe.

\subsection{Extriangulated categories}

Let us briefly recall the definition and basic properties of extriangulated categories from \cite{NP}. Throughout this paper, let $\mathscr{C}$ be an additive category, and $\mathit{Ab}$ denotes the category of abelian groups.

\begin{definition}\label{DefExtension}
Suppose $\mathscr{C}$ is equipped with a biadditive functor $\mathbb{E}\colon\mathscr{C}^\mathrm{op}\times\mathscr{C}\to\mathit{Ab}$. For any pair of objects $A,C\in\mathscr{C}$, an element $\delta\in\mathbb{E}(C,A)$ is called an {\it $\mathbb{E}$-extension}. 
\end{definition}

The following notions will be used in the proceeding sections.
\begin{definition}\label{DefProjInj}
Let $\mathscr{C}$ be an additive category, and let $\mathbb{E}\colon\mathscr{C}^\mathrm{op}\times\mathscr{C}\to\mathit{Ab}$ be a biadditive functor.
\begin{enumerate}[\rm(1)]
\item $\mathrm{Proj}_{\mathbb{E}}\mathscr{C}$ denotes the full subcategory of $\mathscr{C}$ consisting of objects $X$ satisfying $\mathbb{E}(X,\mathscr{C})=0$.
\item $\mathrm{Inj}_{\mathbb{E}}\mathscr{C}$ denotes the full subcategory of $\mathscr{C}$ consisting of objects $X$ satisfying $\mathbb{E}(\mathscr{C},X)=0$.
\end{enumerate}
We call an object in $\mathrm{Proj}_{\mathbb{E}}\mathscr{C}$ an {\it $\mathbb{E}$-projective} object, or just a {\it projective} object if no confusion may arise. Similarly an object in $\mathrm{Inj}_{\mathbb{E}}\mathscr{C}$ is called an {\it $\mathbb{E}$-injective} object, or just an {\it injective} object.
\end{definition}

\begin{definition}
Let $\mathscr{C}$ be a category, and let $\mathbb{E}\colon\mathscr{C}^\mathrm{op}\times\mathscr{C}\to\mathit{Ab}$ be a biadditive functor.
\begin{enumerate}[\rm(1)]
\item A functor $\mathbb{F}\colon\mathscr{C}^\mathrm{op}\times\mathscr{C}\to\mathit{Set}$ is called a {\it subfunctor} of $\mathbb{E}$ if it satisfies the following conditions.
\begin{itemize}
\item $\mathbb{F}(C,A)$ is a subset of $\mathbb{E}(C,A)$, for any $A,C\in\mathscr{C}$.
\item $\mathbb{F}(c,a)=\mathbb{E}(c,a)|_{\mathbb{F}(C,A)}$ holds, for any $a\in\mathscr{C}(A,A^{\prime})$ and $c\in\mathscr{C}(C^{\prime},C)$.
\end{itemize}
In this case, we write $\mathbb{F}\subseteq\mathbb{E}$.
\item A subfunctor $\mathbb{F}\subseteq\mathbb{E}$ is said to be {\it additive} if $\mathbb{F}(C,A)\subseteq\mathbb{E}(C,A)$ is a subgroup for any $A,C\in\mathscr{C}$. In this case, $\mathbb{F}\colon\mathscr{C}^\mathrm{op}\times\mathscr{C}\to\mathit{Ab}$ itself becomes a biadditive functor.
\end{enumerate}
\end{definition}

\begin{remark}\label{RemProjInj}
For any additive subfunctor $\mathbb{F}\subseteq\mathbb{E}$, we have $\mathrm{Proj}_{\mathbb{F}}\mathscr{C}\supseteq\mathrm{Proj}_{\mathbb{E}}\mathscr{C}$ and $\mathrm{Inj}_{\mathbb{F}}\mathscr{C}\supseteq\mathrm{Inj}_{\mathbb{E}}\mathscr{C}$.
\end{remark}

\begin{remark}
Let $\delta\in\mathbb{E}(C,A)$ be any $\mathbb{E}$-extension. By the functoriality of $\mathbb{E}$, for any $a\in\mathscr{C}(A,A^{\prime})$ and $c\in\mathscr{C}(C^{\prime},C)$, we have $\mathbb{E}$-extensions
\[ \mathbb{E}(C,a)(\delta)\in\mathbb{E}(C,A^{\prime})\ \ \text{and}\ \ \mathbb{E}(c,A)(\delta)\in\mathbb{E}(C^{\prime},A). \]
We abbreviately denote them by $a\circ\delta$ and $\delta\circ c$, or just by $a\delta$ and $\delta c$.
In this terminology, we have
\[ \mathbb{E}(c,a)(\delta)=(a\delta)c=a(\delta c) \]
which we simply denote by $a\delta c$, in $\mathbb{E}(C^{\prime},A^{\prime})$.
\end{remark}

\begin{definition}\label{DefMorphExt}
Let $\delta\in\mathbb{E}(C,A),\delta^{\prime}\in\mathbb{E}(C^{\prime},A^{\prime})$ be any pair of $\mathbb{E}$-extensions. A {\it morphism} $(a,c)\colon\delta\to\delta^{\prime}$ of $\mathbb{E}$-extensions is a pair of morphisms $a\in\mathscr{C}(A,A^{\prime})$ and $c\in\mathscr{C}(C,C^{\prime})$ in $\mathscr{C}$, satisfying the equality
$a\delta=\delta^{\prime}c$.
\end{definition}

\begin{definition}\label{DefSplitExtension}
For any $A,C\in\mathscr{C}$, the zero element $0\in\mathbb{E}(C,A)$ is called the {\it split $\mathbb{E}$-extension}.
\end{definition}

\begin{definition}\label{DefSumExtension}
Let $\delta=(A,\delta,C),\delta^{\prime}=(A^{\prime},\delta^{\prime},C^{\prime})$ be any pair of $\mathbb{E}$-extensions. Let
\[ C\overset{\iota_C}{\longrightarrow}C\oplus C^{\prime}\overset{\iota_{C^{\prime}}}{\longleftarrow}C^{\prime}\ \mbox{ and }\ 
A\overset{p_A}{\longleftarrow}A\oplus A^{\prime}\overset{p_{A^{\prime}}}{\longrightarrow}A^{\prime} \]
be coproduct and product in $\mathscr{C}$, respectively. We remark that, by the additivity of $\mathbb{E}$, we have a natural isomorphism
\[ \mathbb{E}(C\oplus C^{\prime},A\oplus A^{\prime})\simeq \mathbb{E}(C,A)\oplus\mathbb{E}(C,A^{\prime})\oplus\mathbb{E}(C^{\prime},A)\oplus\mathbb{E}(C^{\prime},A^{\prime}). \]

Let $\delta\oplus\delta^{\prime}\in\mathbb{E}(C\oplus C^{\prime},A\oplus A^{\prime})$ be the element corresponding to $(\delta,0,0,\delta^{\prime})$ through this isomorphism. In other words, $\delta\oplus\delta^{\prime}$ is the unique element which satisfies
\begin{eqnarray*}
\mathbb{E}(\iota_C,p_A)(\delta\oplus\delta^{\prime})=\delta&,&\mathbb{E}(\iota_C,p_{A^{\prime}})(\delta\oplus\delta^{\prime})=0,\\
\mathbb{E}(\iota_{C^{\prime}},p_A)(\delta\oplus\delta^{\prime})=0&,&\mathbb{E}(\iota_{C^{\prime}},p_{A^{\prime}})(\delta\oplus\delta^{\prime})=\delta^{\prime}.
\end{eqnarray*}
\end{definition}

\begin{definition}\label{DefSqEquiv}
Let $A,C\in\mathscr{C}$ be any pair of objects. Two sequences of morphisms in $\mathscr{C}$
\[ A\overset{x}{\longrightarrow}B\overset{y}{\longrightarrow}C\ \ \text{and}\ \ A\overset{x^{\prime}}{\longrightarrow}B^{\prime}\overset{y^{\prime}}{\longrightarrow}C \]
are said to be {\it equivalent} if there exists an isomorphism $b\in\mathscr{C}(B,B^{\prime})$ which makes the following diagram commutative.
\[
\xy
(-16,0)*+{A}="0";
(3,0)*+{}="1";
(0,8)*+{B}="2";
(0,-8)*+{B^{\prime}}="4";
(-3,0)*+{}="5";
(16,0)*+{C}="6";
{\ar^{x} "0";"2"};
{\ar^{y} "2";"6"};
{\ar_{x^{\prime}} "0";"4"};
{\ar_{y^{\prime}} "4";"6"};
{\ar^{b}_{\simeq} "2";"4"};
{\ar@{}|{} "0";"1"};
{\ar@{}|{} "5";"6"};
\endxy
\]

We denote the equivalence class of $A\overset{x}{\longrightarrow}B\overset{y}{\longrightarrow}C$ by $[A\overset{x}{\longrightarrow}B\overset{y}{\longrightarrow}C]$.
\end{definition}

\begin{definition}\label{DefAddSeq}
$\ \ $
\begin{enumerate}[\rm(1)]
\item For any $A,C\in\mathscr{C}$, we denote as
\[ 0=[A\overset{\Big[\raise1ex\hbox{\leavevmode\vtop{\baselineskip-8ex \lineskip1ex \ialign{#\crcr{$\scriptstyle{1}$}\crcr{$\scriptstyle{0}$}\crcr}}}\Big]}{\longrightarrow}A\oplus C\overset{[0\ 1]}{\longrightarrow}C]. \]

\item For any $[A\overset{x}{\longrightarrow}B\overset{y}{\longrightarrow}C]$ and $[A^{\prime}\overset{x^{\prime}}{\longrightarrow}B^{\prime}\overset{y^{\prime}}{\longrightarrow}C^{\prime}]$, we denote as
\[ [A\overset{x}{\longrightarrow}B\overset{y}{\longrightarrow}C]\oplus [A^{\prime}\overset{x^{\prime}}{\longrightarrow}B^{\prime}\overset{y^{\prime}}{\longrightarrow}C^{\prime}]=[A\oplus A^{\prime}\overset{x\oplus x^{\prime}}{\longrightarrow}B\oplus B^{\prime}\overset{y\oplus y^{\prime}}{\longrightarrow}C\oplus C^{\prime}]. \]
\end{enumerate}
\end{definition}

\begin{definition}\label{DefRealization}
Let $\mathfrak{s}$ be a correspondence which associates an equivalence class $\mathfrak{s}(\delta)=[A\overset{x}{\longrightarrow}B\overset{y}{\longrightarrow}C]$ to any $\mathbb{E}$-extension $\delta\in\mathbb{E}(C,A)$. This $\mathfrak{s}$ is called a {\it realization} of $\mathbb{E}$, if it satisfies the following condition $(\ast)$. \begin{itemize}
\item[$(\ast)$] Let $\delta\in\mathbb{E}(C,A)$ and $\delta^{\prime}\in\mathbb{E}(C^{\prime},A^{\prime})$ be any pair of $\mathbb{E}$-extensions, with
\[\mathfrak{s}(\delta)=[A\overset{x}{\longrightarrow}B\overset{y}{\longrightarrow}C]\text{ and } \mathfrak{s}(\delta^{\prime})=[A^{\prime}\overset{x^{\prime}}{\longrightarrow}B^{\prime}\overset{y^{\prime}}{\longrightarrow}C^{\prime}].\]
Then, for any morphism $(a,c)\colon\delta\to\delta^{\prime}$, there exists $b\in\mathscr{C}(B,B^{\prime})$ which makes the following diagram commutative.
\begin{equation}\label{MorphRealize}
\xy
(-12,6)*+{A}="0";
(0,6)*+{B}="2";
(12,6)*+{C}="4";
(-12,-6)*+{A^{\prime}}="10";
(0,-6)*+{B^{\prime}}="12";
(12,-6)*+{C^{\prime}}="14";
{\ar^{x} "0";"2"};
{\ar^{y} "2";"4"};
{\ar_{a} "0";"10"};
{\ar^{b} "2";"12"};
{\ar^{c} "4";"14"};
{\ar_{x^{\prime}} "10";"12"};
{\ar_{y^{\prime}} "12";"14"};
{\ar@{}|{} "0";"12"};
{\ar@{}|{} "2";"14"};
\endxy
\end{equation}
\end{itemize}
We say that the sequence $A\overset{x}{\longrightarrow}B\overset{y}{\longrightarrow}C$ {\it realizes} $\delta$, whenever it satisfies $\mathfrak{s}(\delta)=[A\overset{x}{\longrightarrow}B\overset{y}{\longrightarrow}C]$.
Also in $(\ast)$, we say that the triplet $(a,b,c)$ {\it realizes} $(a,c)$.
\end{definition}

\begin{definition}\label{DefAdditiveRealization}
Let $\mathscr{C},\mathbb{E}$ be as above. A realization $\mathfrak{s}$ of $\mathbb{E}$ is said to be {\it additive}, if it satisfies the following conditions.
\begin{enumerate}[\rm(i)]
\item For any $A,C\in\mathscr{C}$, the split $\mathbb{E}$-extension $0\in\mathbb{E}(C,A)$ satisfies
\[ \mathfrak{s}(0)=0. \]
\item For any pair of $\mathbb{E}$-extensions $\delta\in\mathbb{E}(C,A)$ and $\delta^{\prime}\in\mathbb{E}(C^{\prime},A^{\prime})$, we have:
\[ \mathfrak{s}(\delta\oplus\delta^{\prime})=\mathfrak{s}(\delta)\oplus\mathfrak{s}(\delta^{\prime}). \]
\end{enumerate}
\end{definition}

\begin{definition}\label{DefExtCat}(\cite[Definition 2.12]{NP})
A triplet $(\mathscr{C},\mathbb{E},\mathfrak{s})$ is called an {\it extriangulated category} if it satisfies the following conditions.
\begin{itemize}
\item[{\rm (ET1)}] $\mathbb{E}\colon\mathscr{C}^{\mathrm{op}}\times\mathscr{C}\to\mathit{Ab}$ is a biadditive functor.
\item[{\rm (ET2)}] $\mathfrak{s}$ is an additive realization of $\mathbb{E}$.
\item[{\rm (ET3)}] Let $\delta\in\mathbb{E}(C,A)$ and $\delta^{\prime}\in\mathbb{E}(C^{\prime},A^{\prime})$ be any pair of $\mathbb{E}$-extensions, realized as
\[ \mathfrak{s}(\delta)=[A\overset{x}{\longrightarrow}B\overset{y}{\longrightarrow}C],\ \ \mathfrak{s}(\delta^{\prime})=[A^{\prime}\overset{x^{\prime}}{\longrightarrow}B^{\prime}\overset{y^{\prime}}{\longrightarrow}C^{\prime}]. \]
For any commutative square
\begin{equation}\label{SquareForET3}
\xy
(-12,6)*+{A}="0";
(0,6)*+{B}="2";
(12,6)*+{C}="4";
(-12,-6)*+{A^{\prime}}="10";
(0,-6)*+{B^{\prime}}="12";
(12,-6)*+{C^{\prime}}="14";
{\ar^{x} "0";"2"};
{\ar^{y} "2";"4"};
{\ar_{a} "0";"10"};
{\ar^{b} "2";"12"};
{\ar_{x^{\prime}} "10";"12"};
{\ar_{y^{\prime}} "12";"14"};
{\ar@{}|{} "0";"12"};
\endxy
\end{equation}
in $\mathscr{C}$, there exists a morphism $(a,c)\colon\delta\to\delta^{\prime}$ satisfying $cy=y^{\prime}b$.
\item[{\rm (ET3)$^{\mathrm{op}}$}] Dual of {\rm (ET3)}.
\item[{\rm (ET4)}] Let $\delta\in\mathbb{E}(D,A)$ and $\delta^{\prime}\in\mathbb{E}(F,B)$ be $\mathbb{E}$-extensions realized by
\[ A\overset{f}{\longrightarrow}B\overset{f^{\prime}}{\longrightarrow}D\ \ \text{and}\ \ B\overset{g}{\longrightarrow}C\overset{g^{\prime}}{\longrightarrow}F \]
respectively. Then there exist an object $E\in\mathscr{C}$, a commutative diagram
\begin{equation}\label{DiagET4}
\xy
(-21,7)*+{A}="0";
(-7,7)*+{B}="2";
(7,7)*+{D}="4";
(-21,-7)*+{A}="10";
(-7,-7)*+{C}="12";
(7,-7)*+{E}="14";
(-7,-21)*+{F}="22";
(7,-21)*+{F}="24";
{\ar^{f} "0";"2"};
{\ar^{f^{\prime}} "2";"4"};
{\ar@{=} "0";"10"};
{\ar_{g} "2";"12"};
{\ar^{d} "4";"14"};
{\ar_{h} "10";"12"};
{\ar_{h^{\prime}} "12";"14"};
{\ar_{g^{\prime}} "12";"22"};
{\ar^{e} "14";"24"};
{\ar@{=} "22";"24"};
{\ar@{}|{} "0";"12"};
{\ar@{}|{} "2";"14"};
{\ar@{}|{} "12";"24"};
\endxy
\end{equation}
in $\mathscr{C}$, and an $\mathbb{E}$-extension $\delta^{\prime\prime}\in\mathbb{E}(E,A)$ realized by $A\overset{h}{\longrightarrow}C\overset{h^{\prime}}{\longrightarrow}E$, which satisfy the following compatibilities.
\begin{itemize}
\item[{\rm (i)}] $D\overset{d}{\longrightarrow}E\overset{e}{\longrightarrow}F$ realizes $f^{\prime}\delta^{\prime}$,
\item[{\rm (ii)}] $\delta^{\prime\prime}d=\delta$,

\item[{\rm (iii)}] $f\delta^{\prime\prime}=\delta^{\prime}e$. 
\end{itemize}

\item[{\rm (ET4)$^{\mathrm{op}}$}] Dual of {\rm (ET4)}.
\end{itemize}
\end{definition}

\begin{example}\label{Example1}
Exact categories and triangulated categories are extriangulated categories. See \cite[Example 2.13]{NP} for more detail.
\end{example}

We use the following terminology.
\begin{definition}\label{DefTermExact1}
Let $(\mathscr{C},\mathbb{E},\mathfrak{s})$ be a triplet satisfying {\rm (ET1)} and {\rm (ET2)}.
\begin{enumerate}[\rm(1)]
\item A sequence $A\overset{x}{\longrightarrow}B\overset{y}{\longrightarrow}C$ is called an {\it $\mathfrak{s}$-conflation} if it realizes some $\mathbb{E}$-extension in $\mathbb{E}(C,A)$.
\item A morphism $f\in\mathscr{C}(A,B)$ is called an {\it $\mathfrak{s}$-inflation} if there is some $\mathfrak{s}$-conflation $A\overset{f}{\longrightarrow}B\to C$.
\item A morphism $f\in\mathscr{C}(A,B)$ is called an {\it $\mathfrak{s}$-deflation} if there is some $\mathfrak{s}$-conflation $K\to A\overset{f}{\longrightarrow}B$.
\end{enumerate}
\end{definition}

\begin{definition}\label{DefTermExact2}
Let $(\mathscr{C},\mathbb{E},\mathfrak{s})$ be a triplet satisfying {\rm (ET1)} and {\rm (ET2)}.
\begin{enumerate}[\rm(1)]
\item If an $\mathfrak{s}$-conflation $A\overset{x}{\longrightarrow}B\overset{y}{\longrightarrow}C$ realizes $\delta\in\mathbb{E}(C,A)$, we call the pair $(A\overset{x}{\longrightarrow}B\overset{y}{\longrightarrow}C,\delta)$ an {\it $\mathfrak{s}$-triangle}, and write it in the following way.
\begin{equation}\label{Etriangle}
A\overset{x}{\longrightarrow}B\overset{y}{\longrightarrow}C\overset{\delta}{\dashrightarrow}
\end{equation}
\item Let $A\overset{x}{\longrightarrow}B\overset{y}{\longrightarrow}C\overset{\delta}{\dashrightarrow}$ and $A^{\prime}\overset{x^{\prime}}{\longrightarrow}B^{\prime}\overset{y^{\prime}}{\longrightarrow}C^{\prime}\overset{\delta^{\prime}}{\dashrightarrow}$ be any pair of $\mathfrak{s}$-triangles. If a triplet $(a,b,c)$ realizes $(a,c)\colon\delta\to\delta^{\prime}$ as in $(\ref{MorphRealize})$, then we write it as
\[
\xy
(-12,6)*+{A}="0";
(0,6)*+{B}="2";
(12,6)*+{C}="4";
(24,6)*+{}="6";
(-12,-6)*+{A^{\prime}}="10";
(0,-6)*+{B^{\prime}}="12";
(12,-6)*+{C^{\prime}}="14";
(24,-6)*+{}="16";
{\ar^{x} "0";"2"};
{\ar^{y} "2";"4"};
{\ar@{-->}^{\delta} "4";"6"};
{\ar_{a} "0";"10"};
{\ar^{b} "2";"12"};
{\ar^{c} "4";"14"};
{\ar_{x^{\prime}} "10";"12"};
{\ar_{y^{\prime}} "12";"14"};
{\ar@{-->}_{\delta^{\prime}} "14";"16"};
{\ar@{}|{} "0";"12"};
{\ar@{}|{} "2";"14"};
\endxy
\]
and call $(a,b,c)$ a {\it morphism of $\mathfrak{s}$-triangles}.
\end{enumerate}
\end{definition}

\begin{definition}\label{DefYoneda}
Assume $\mathscr{C}$ and $\mathbb{E}$ satisfy {\rm (ET1)}.
By Yoneda's lemma, any $\mathbb{E}$-extension $\delta\in\mathbb{E}(C,A)$ induces natural transformations
\[ \delta\circ-\colon\mathscr{C}(-,C)\to\mathbb{E}(-,A)\ \ \text{and}\ \ -\circ\delta\colon\mathscr{C}(A,-)\to\mathbb{E}(C,-). \]
For any $X\in\mathscr{C}$, they are given as follows.
\begin{enumerate}[\rm(1)]
\item $\delta\circ-\colon\mathscr{C}(X,C)\to\mathbb{E}(X,A)\ ;\ f\mapsto \delta f$.
\item $-\circ\delta\colon\mathscr{C}(A,X)\to\mathbb{E}(C,X)\ ;\ g\mapsto g\delta$.
\end{enumerate}
\end{definition}

\begin{remark}\label{Remdelautom}
By \cite[Corollary 3.8]{NP}, for any $\mathfrak{s}$-triangle $A\overset{x}{\longrightarrow}B\overset{y}{\longrightarrow}C\overset{\delta}{\dashrightarrow}$ and any $\delta^{\prime}\in\mathbb{E}(C,A)$, the following are equivalent.
\begin{enumerate}[\rm(1)]
\item $\mathfrak{s}(\delta)=\mathfrak{s}(\delta^{\prime})$.
\item There are automorphisms $a\in\mathscr{C}(A,A), c\in\mathscr{C}(C,C)$ satisfying $x a=x$, $c y=y$ and $\delta^{\prime}=a\delta c$.
\end{enumerate}
\end{remark}

\begin{definition}\label{DefExtClosed}
Let $(\mathscr{C},\mathbb{E},\mathfrak{s})$ be a triplet satisfying {\rm (ET1)} and {\rm (ET2)}.
Let $\mathscr{D}\subseteq\mathscr{C}$ be an additive full subcategory.
We say that $\mathscr{D}$ is {\it extension-closed} if it satisfies the following condition.
\begin{itemize}
\item If an $\mathfrak{s}$-conflation $A\to B\to C$ satisfies $A,C\in\mathscr{D}$, then $B\in\mathscr{D}$.
\end{itemize}
\end{definition}

The following has been shown in \cite[Propositions 3.3, 3.11]{NP}.
\begin{fact}\label{FactExact}
Assume that $(\mathscr{C},\mathbb{E},\mathfrak{s})$ satisfies {\rm (ET1),(ET2),(ET3),(ET3)$^{\mathrm{op}}$}, and let $A\overset{x}{\longrightarrow}B\overset{y}{\longrightarrow}C\overset{\delta}{\dashrightarrow}$ be any $\mathfrak{s}$-triangle.
\begin{enumerate}[\rm(1)]
\item The following sequences of natural transformations are exact.
\[ \mathscr{C}(C,-)\overset{-\circ y}{\longrightarrow}\mathscr{C}(B,-)\overset{-\circ x}{\longrightarrow}\mathscr{C}(A,-)\overset{-\circ\delta}{\longrightarrow}\mathbb{E}(C,-)\overset{-\circ y}{\longrightarrow}\mathbb{E}(B,-), \]
\[ \mathscr{C}(-,A)\overset{x\circ-}{\longrightarrow}\mathscr{C}(-,B)\overset{y\circ-}{\longrightarrow}\mathscr{C}(-,C)\overset{\delta\circ-}{\longrightarrow}\mathbb{E}(-,A)\overset{x\circ -}{\longrightarrow}\mathbb{E}(-,B). \]
In particular, $x$ is a section if and only if $\delta=0$ if and only if $y$ is a retraction.
\item If $(\mathscr{C},\mathbb{E},\mathfrak{s})$ moreover satisfies {\rm (ET4)}, then the following sequence is exact.
\[ \mathbb{E}(C,-)\overset{-\circ y}{\longrightarrow}\mathbb{E}(B,-)\overset{-\circ x}{\longrightarrow}\mathbb{E}(A,-). \]
\item Dually, if $(\mathscr{C},\mathbb{E},\mathfrak{s})$ satisfies {\rm (ET4)$^\mathrm{op}$}, then the following sequence is exact.
\[ \mathbb{E}(-,A)\overset{x\circ -}{\longrightarrow}\mathbb{E}(-,B)\overset{y\circ -}{\longrightarrow}\mathbb{E}(-,C). \]
\end{enumerate}
\end{fact}

\subsection{Stable categories of extriangulated categories}\label{section: preliminary stable}

We introduce the stable (resp.\ costable) categories of arbitrary extriangulated categories.

\begin{definition}\label{DefCoStable}
Let $(\mathscr{C},\mathbb{E},\mathfrak{s})$ be an extriangulated category.
We denote by $\mathcal{P}$ (resp.\ $\mathcal{I}$) the ideal of $\mathscr{C}$ consisting of all morphisms $f$ satisfying $\mathbb{E}(f,-)=0$ (resp.\ $\mathbb{E}(-,f)=0$).
The \emph{stable category} (resp.\ \emph{costable category}) of $\mathscr{C}$ is defined as the ideal quotient
\[\underline{\mathscr{C}}:=\mathscr{C}/\mathcal{P}\ (\mbox{resp.}\ \overline{\mathscr{C}}:=\mathscr{C}/\mathcal{I}).\]
\end{definition}

The next remark shows that our stable category coincides with the classical one if $\mathscr{C}$ has enough projectives.
Recall that, for an additive category $\mathscr{C}$ and an additive full subcategory $\mathscr{D}$ of $\mathscr{C}$, let $[\mathscr{D}]$ be the ideal of $\mathscr{C}$ defined by
\[ [\mathscr{D}](X,Y)=\{ f\in\mathscr{C}(X,Y)\mid f\ \text{factors through some}\ D\in\mathscr{D} \}.\]
Then we have an additive category $\mathscr{C}/\mathscr{D}:=\mathscr{C}/[\mathscr{D}]$ called the \emph{ideal quotient}.

\begin{remark}\label{C/P and C/P}
\begin{enumerate}[\rm(1)]
\item If $\mathscr{C}$ \emph{has enough projectives} $\mathscr{P}:=\mathrm{Proj}_{\mathbb{E}}\mathscr{C}$ (namely, any object $X\in\mathscr{C}$ admits an $\mathfrak{s}$-deflation $P\to X$ from some $P\in\mathscr{P}$ \cite[Definition~1.13]{LNa}\cite[Definition~3.25]{NP}), then $\underline{\mathscr{C}}$ coincides with the ideal quotient $\mathscr{C}/\mathscr{P}$. Otherwise they are different in general.
\item If $\mathscr{C}$ \emph{has enough injectives} $\mathscr{I}:=\mathrm{Inj}_{\mathbb{E}}\mathscr{C}$ (namely, any object $X\in\mathscr{C}$ admits an $\mathfrak{s}$-inflation $X\to I$ to some $I\in\mathscr{I}$), then $\overline{\mathscr{C}}$ coincides with the ideal quotient $\mathscr{C}/\mathscr{I}$. Otherwise they are different in general.
\end{enumerate}
\end{remark}

The following properties are immediate from definition.

\begin{proposition}\label{first properties}
\begin{enumerate}[\rm(1)]
\item The functor $\mathbb{E}\colon\mathscr{C}^\mathrm{op}\times\mathscr{C}\to\mathit{Ab}$ induces a functor $\mathbb{E}\colon\underline{\mathscr{C}}^\mathrm{op}\times\overline{\mathscr{C}}\to\mathit{Ab}$.
\item An object $A\in\mathscr{C}$ is projective if and only if $\phantom{x}\underline{\mathscr{C}}(A,-)=0$, if and only if $A\simeq0$ in $\underline{\mathscr{C}}$.
\item An object $B\in\mathscr{C}$ is injective if and only if $\overline{\mathscr{C}}(-,B)=0$, if and only if $B\simeq0$ in $\overline{\mathscr{C}}$.
\end{enumerate}
\end{proposition}

Now we introduce the (co)syzygy functors and the long exact sequences associated with $\mathfrak{s}$-triangles.
The following has been shown in \cite{HLN2} and \cite{LNa}.

\begin{definition-proposition}\label{LemOmegaExact}
Let $(\mathscr{C},\mathbb{E},\mathfrak{s})$ be an extriangulated category with enough projectives.
For each $C\in\mathscr{C}$, choose an object $\Omega C\in\mathscr{C}$ and an extension $\omega_C\in\mathbb{E}(C,\Omega C)$ such that $\mathfrak{s}(\omega_C)=[\Omega C\to P\to C]$ satisfies that $P$ is projective.
\begin{enumerate}[\rm(1)]
\item The following gives an additive endofunctor $\Omega\colon\underline{\mathscr{C}}\to\underline{\mathscr{C}}$.
\begin{itemize}
\item[-] To each $C\in\underline{\mathscr{C}}$, associate $\Omega C\in\underline{\mathscr{C}}$ chosen as above.
\item[-] For any $\underline{c}\in\underline{\mathscr{C}}(C,C^{\prime})$, put $\Omega\underline{c}=\underline{d}$, where $d\in\mathscr{C}(\Omega C,\Omega C^{\prime})$ is a morphism satisfying $d\omega_C=\omega_{C^{\prime}}c$. Such $\underline{d}$ is uniquely given.
\end{itemize}
The functor $\Omega$ is uniquely determined up to natural isomorphism, independently of the choices of those $\Omega C$ and $\omega_C$.

\item For any pair of objects $A,C\in\mathscr{C}$, the homomorphism $\mathbf{w}=\mathbf{w}_{C,A}\colon\mathbb{E}(C,A)\to\underline{\mathscr{C}}(\Omega C,A)$ which sends each $\delta\in\mathbb{E}(C,A)$ to $\mathbf{w}(\delta)=\underline{w}$ where $w\in\mathscr{C}(\Omega C,A)$ is any morphism satisfying $\delta=w\omega_C$, is well-defined. Moreover, this $\mathbf{w}$ is natural in $A,C\in\mathscr{C}$.
By definition, for any morphism $\underline{c}\in\underline{\mathscr{C}}(C,C^{\prime})$ we have $\Omega\underline{c}=\mathbf{w}(\omega_{C^{\prime}}c)$.

\item Put $\mathfrak{s}(\omega_C)=[\Omega C\overset{p}{\longrightarrow}P\overset{q}{\longrightarrow}C]$. Let $A\overset{x}{\longrightarrow}B\overset{y}{\longrightarrow}C\overset{\delta}{\dashrightarrow}$ be any $\mathfrak{s}$-triangle. If $\mathbf{w}(\delta)=\underline{w}$ for $w\in\mathscr{C}(\Omega C,A)$, then there is $b\in\mathscr{C}(P,B)$ such that
\[ \Omega C\xrightarrow{\left[\bsm -w\\p \esm\right]}A\oplus P\xrightarrow{[x\ b]}B\overset{\omega_Cy}{\dashrightarrow} \]
becomes an $\mathfrak{s}$-triangle.
\end{enumerate}

Similarly, if $(\mathscr{C},\mathbb{E},\mathfrak{s})$ has enough injectives, we obtain an endofunctor $\Sigma\colon\overline{C}\to\overline{C}$ and homomorphisms $\mathbb{E}(C,A)\to\overline{\mathscr{C}}(C,\Sigma A)$ with the dual properties.
\end{definition-proposition}

\begin{proof}
{\rm (1)} has been shown in \cite[Proposition~4.3]{LNa} and \cite[Proposition~3.4]{HLN2}. {\rm (2)} is the dual of \cite[Remark~3.5, Proposition~3.9]{HLN2}. {\rm (3)} follows from \cite[Proposition~1.20]{LNa}.
\end{proof}

The following important observation is an analogue of a result in \cite[p.\,346]{AR3}.

\begin{theorem}\label{long exact sequence}
Let $\mathscr{C}$ be an extriangulated category, and let $A\overset{x}{\longrightarrow}B\overset{y}{\longrightarrow}C\overset{\delta}{\dashrightarrow}$ be any $\mathfrak{s}$-triangle.
\begin{enumerate}[\rm(1)]
\item If $\mathscr{C}$ has enough projectives, then
\begin{eqnarray*}
&\cdots\to\underline{\mathscr{C}}(-,\Omega^{i+1}C)\xrightarrow{(\Omega^i\underline{w})\circ-}\underline{\mathscr{C}}(-,\Omega^iA)\xrightarrow{(\Omega^i\underline{x})\circ-}\underline{\mathscr{C}}(-,\Omega^iB)\xrightarrow{(\Omega^i\underline{y})\circ-}\underline{\mathscr{C}}(-,\Omega^iC)\to\cdots&\\
&\cdots\xrightarrow{\underline{w}\circ-}\underline{\mathscr{C}}(-,A)\xrightarrow{\underline{x}\circ-}\underline{\mathscr{C}}(-,B)\xrightarrow{\underline{y}\circ-}\underline{\mathscr{C}}(-,C)\xrightarrow{\delta\circ-}\mathbb{E}(-,A)\xrightarrow{\underline{x}\circ-}\mathbb{E}(-,B)\xrightarrow{\underline{y}\circ-}\mathbb{E}(-,C)&
\end{eqnarray*}
is exact for $\underline{w}=\mathbf{w}(\delta)$, where $\Omega$ and $\mathbf{w}$ are those obtained in Definition-Proposition~\ref{LemOmegaExact}.

\item Dually, if $\mathscr{C}$ has enough injectives, then we obtain an exact sequence
\begin{eqnarray*}
&\cdots\to\overline{\mathscr{C}}(\Sigma^{i+1}A,-)\longrightarrow\overline{\mathscr{C}}(\Sigma^iC,-)\xrightarrow{-\circ(\Sigma^i\overline{y})}\overline{\mathscr{C}}(\Sigma^iB,-)\xrightarrow{-\circ(\Sigma^i\overline{x})}\overline{\mathscr{C}}(\Sigma^iA,-)\to\cdots&\\
&\cdots\to\overline{\mathscr{C}}(C,-)\xrightarrow{-\circ\overline{y}}\overline{\mathscr{C}}(B,-)\xrightarrow{-\circ\overline{x}}\overline{\mathscr{C}}(A,-)\xrightarrow{-\circ\delta}\mathbb{E}(C,-)\xrightarrow{-\circ\overline{y}}\mathbb{E}(B,-)\xrightarrow{-\circ\overline{x}}\mathbb{E}(A,-).&
\end{eqnarray*}
\end{enumerate}
\end{theorem}

To prove this, we need a preparation.

\begin{lemma}[{e.g.\ \cite[1.3(4)]{I2}}]\label{tensor}
Let $\mathscr{C}$ be an additive category and $\mathscr{D}$ an additive full subcategory.
For a complex $A\xrightarrow{x}B\xrightarrow{y}C$ in $\mathscr{C}$, we assume that
\[\mathscr{C}(-,A)\xrightarrow{x\circ-}\mathscr{C}(-,B)\xrightarrow{y\circ-}\mathscr{C}(-,C)\to F\to0\]
is an exact sequence such that $F(\mathscr{D})=0$. Then the following sequence is also exact.
\[(\mathscr{C}/\mathscr{D})(-,A)\xrightarrow{\overline{x}\circ-}(\mathscr{C}/\mathscr{D})(-,B)\xrightarrow{\overline{y}\circ-}(\mathscr{C}/\mathscr{D})(-,C)\to F\to0.\]
\end{lemma}

\begin{proof}[Proof of Theorem \ref{long exact sequence}]
We only show {\rm (1)}.
By \cite[Corollary 3.12]{NP}, we have an exact sequence
\[\mathscr{C}(-,A)\to\mathscr{C}(-,B)\to\mathscr{C}(-,C)\to\mathbb{E}(-,A)\to\mathbb{E}(-,B)\to\mathbb{E}(-,C).\]
Thus the exactness of 
\begin{equation}\label{ExLeftHalf}
\underline{\mathscr{C}}(-,A)\xrightarrow{\underline{x}\circ-}\underline{\mathscr{C}}(-,B)\xrightarrow{\underline{y}\circ-}\underline{\mathscr{C}}(-,C)\xrightarrow{\delta\circ-}\mathbb{E}(-,A)\xrightarrow{\underline{x}\circ-}\mathbb{E}(-,B)\xrightarrow{\underline{y}\circ-}\mathbb{E}(-,C)
\end{equation}
follows from Lemma~\ref{tensor}.

Applying Definition-Proposition~\ref{LemOmegaExact} iteratively, we obtain $\mathfrak{s}$-triangles of the following form,
\begin{eqnarray*}
&&\Omega C\xrightarrow{\left[\bsm -w\\\ast \esm\right]}A\oplus P\xrightarrow{[x\ b]}B\overset{\omega_Cy}{\dashrightarrow},\\
&&\Omega B\xrightarrow{\left[\bsm -w^{\prime}\\\ast \esm\right]}(\Omega C)\oplus P^{\prime}\xrightarrow{\left[\bsm -w& \ast\\ \ast&\ast\esm\right]}A\oplus P\overset{\omega_B[x\ b]}{\dashrightarrow},\\
&&\Omega (A\oplus P)\xrightarrow{\left[\bsm -w^{\prime\prime}\\\ast \esm\right]}(\Omega B)\oplus P^{\prime\prime}\xrightarrow{\left[\bsm -w^{\prime}& \ast\\ \ast&\ast\esm\right]}(\Omega C)\oplus P^{\prime}\dashrightarrow,
\end{eqnarray*}
where $P,P^{\prime},P^{\prime\prime}$ are projective, and
\[ \underline{w}=\mathbf{w}(\delta),\ \ \underline{w}^{\prime}=\mathbf{w}(\omega_Cy)=\Omega\underline{y},\ \ \underline{w}^{\prime\prime}=\mathbf{w}(\omega_B[x\ b])=\Omega\underline{[x\ b]}. \]
By the left half of the exact sequences $(\ref{ExLeftHalf})$ obtained from these $\mathfrak{s}$-triangles,
\[ \underline{\mathscr{C}}(-,\Omega A)\xrightarrow{\Omega\underline{x}\circ-}\underline{\mathscr{C}}(-,\Omega B)\xrightarrow{\Omega\underline{y}\circ-}\underline{\mathscr{C}}(-,\Omega C)\xrightarrow{\underline{w}\circ-}\underline{\mathscr{C}}(-,A)\xrightarrow{\underline{x}\circ-}\underline{\mathscr{C}}(-,B) \]
becomes exact, since
\[
\xy
(-12,7)*+{\Omega A}="0";
(12,7)*+{\Omega B}="2";
(-12,-7)*+{\Omega(A\oplus P)}="4";
(12,-7)*+{\Omega B}="6";
{\ar^{\Omega\underline{x}} "0";"2"};
{\ar_{\Omega\underline{\left[\bsm 1\\0\esm\right]}}^{\cong} "0";"4"};
{\ar@{=} "2";"6"};
{\ar_(0.6){\Omega\underline{[x\ b]}} "4";"6"};
{\ar@{}|\circlearrowright "0";"6"};
\endxy
\]
is commutative. Repeating this, we obtain the desired exact sequence.
\end{proof}

\section{Almost split extensions}\label{section_ARext}

\subsection{Almost split extension}

In this subsection, let $\mathscr{C}$ be an additive category.
In the rest of this subsection, we fix a biadditive functor $\mathbb{E}\colon\mathscr{C}^\mathrm{op}\times\mathscr{C}\to\mathit{Ab}$.

\begin{definition}\label{DefARExt}
A non-split (i.e. non-zero) $\mathbb{E}$-extension $\delta\in\mathbb{E}(C,A)$ is said to be {\it almost split} if it satisfies the following conditions.
\begin{itemize}
\item[{\rm (AS1)}] $a\delta=0$ for any non-section $a\in\mathscr{C}(A,A^{\prime})$.
\item[{\rm (AS2)}] $\delta c=0$ for any non-retraction $c\in\mathscr{C}(C^{\prime},C)$.
\end{itemize}
\end{definition}

Thus $\delta$ is a `universally minimal element' of $\mathbb{E}(-,A)$ and $\mathbb{E}(C,-)$ in the sense of \cite[p.\ 292]{A2}.

\begin{remark}
If $a\delta=0$ holds for a {\it section} $a\in\mathscr{C}(A,A^{\prime})$, then we have $\delta=0$. Thus {\rm (AS1)} is the best possible vanishing condition with respect to $a\circ -$ where $a\in\mathscr{C}(A,X)$, for a non-split $\delta$. Similarly for {\rm (AS2)}.
\end{remark}

\begin{remark}\label{RemRelAR2}
For any $\delta\in\mathbb{E}(C,A)$, the property of being {\it almost split} does not depend on a biadditive subfunctor $\mathbb{F}\subseteq\mathbb{E}$ containing $\delta$.
Indeed, if $\delta\in\mathbb{F}(C,A)$ holds for some biadditive subfunctor $\mathbb{F}\subseteq\mathbb{E}$, then $\delta$ is almost split as an $\mathbb{E}$-extension if and only if it is almost split as an $\mathbb{F}$-extension.
\end{remark}

In the following, we call an almost split $\mathbb{E}$-extension simply an {\it almost split extension} if there is no confusion.
\begin{definition}\label{DefStrInd}
A non-zero object $A\in\mathscr{C}$ is said to be {\it endo-local} if $\mathrm{End}_{\mathscr{C}}(A)$ is local.
\end{definition}

We refer to \cite[Proposition 15.15]{AF} for characterizations of local rings.

\begin{proposition}\label{PropARExt}
For any non-split $\mathbb{E}$-extension $\delta\in\mathbb{E}(C,A)$, the following holds.
\begin{enumerate}[\rm(1)]
\item If $\delta$ satisfies {\rm (AS1)}, then $A$ is endo-local.
\item If $\delta$ satisfies {\rm (AS2)}, then $C$ is endo-local.
\end{enumerate}
\end{proposition}
\begin{proof}
{\rm (1)} It suffices to show that $I:=\{ a\in\mathrm{End}_{\mathscr{C}}(A)\mid a\ \text{is not a section}\}$ is closed under addition. The condition {\rm (AS1)} implies that $I$ is the kernel of the homomorphism $\mathscr{C}(A,A)\to\mathbb{E}(C,A)$, $a\mapsto a\delta$. Thus the assertion follows.
{\rm (2)} is dual to {\rm (1)}.
\end{proof}

Recall that an additive category $\mathscr{C}$ is called \emph{Krull--Schmidt} if any object is isomorphic to a finite direct sum of endo-local objects. We denote by $\ind\mathscr{C}$ the set of isoclasses of indecomposable objects in $\mathscr{C}$. For example, for an arbitrary Krull--Schmidt extraingulated category $\mathscr{C}$, the stable category $\underline{\mathscr{C}}$ and the costable category $\overline{\mathscr{C}}$ are Krull--Schmidt with $\ind\underline{\mathscr{C}}=\{X\in\ind\mathscr{C}\mid\mbox{$X$ is non-projective}\}$ and $\ind\overline{\mathscr{C}}=\{X\in\ind\mathscr{C}\mid\mbox{$X$ is non-injective}\}$.

\subsection{Realization by almost split sequences}

In the rest, let $(\mathscr{C},\mathbb{E},\mathfrak{s})$ be an extriangulated category\footnote{In fact, it only needs to satisfy {\rm (ET1),(ET2),(ET3),(ET3)$^\mathrm{op}$}, until the end of this section.}.
Then we have the following uniqueness of almost split extensions.

\begin{proposition}\label{PropARSeq3}
Let $A,A',C,C'\in\mathscr{C}$.
\begin{enumerate}[\rm(1)]
\item If $\rho\in\mathbb{E}(C,A)$ and $\rho^{\prime}\in\mathbb{E}(C,A^{\prime})$ satisfy {\rm (AS1)}, then there is an isomorphism $a\in\mathscr{C}(A,A^{\prime})$ such that $a\rho=\rho^{\prime}$.
\item If $\delta\in\mathbb{E}(C,A)$ and $\delta^{\prime}\in\mathbb{E}(C^{\prime},A)$ satisfy {\rm (AS2)}, then there is an isomorphism $c\in\mathscr{C}(C^{\prime},C)$ such that $\delta c=\delta^{\prime}$.
\item Let $\delta\in\mathbb{E}(C,A)$ be an almost split extension. Then $\mathrm{End}_{\mathscr{A}}(A)\delta=\delta\mathrm{End}_{\mathscr{C}}(C)$ is an $(\mathrm{End}_{\mathscr{C}}(A),\mathrm{End}_{\mathscr{C}}(C))$-bimodule which is simple on both sides, and consists of 0 and all almost split extensions in $\mathbb{E}(C,A)$. Thus we have an isomorphism $f:\mathrm{End}_{\mathscr{C}}(C)/\mathrm{rad}\mathrm{End}_{\mathscr{C}}(C)\simeq \mathrm{End}_{\mathscr{C}}(A)/\mathrm{rad}\mathrm{End}_{\mathscr{C}}(A)$ of rings satisfying $\delta c=f(c)\delta$ for each $c\in\mathrm{End}_{\mathscr{C}}(C)$.
\end{enumerate}
\end{proposition}

\begin{proof}
(2) Let $\mathfrak{s}(\delta)=[A\overset{x}{\longrightarrow}B\overset{y}{\longrightarrow}C]$ and $\mathfrak{s}(\delta')=[A\overset{x'}{\longrightarrow}B'\overset{y'}{\longrightarrow}C']$. Then $x$ and $x'$ are non-sections. By 
(AS1), we have $x\delta'=0$ and $x'\delta=0$. By Fact \ref{FactExact}(1),
there are $c\in\mathscr{C}(C^{\prime},C)$ and $c^{\prime}\in\mathscr{C}(C,C^{\prime})$ satisfying $\delta'=\delta c$ and $\delta=\delta' c'$.
Thus $\delta(1-cc')=0$ and $\delta'(1-c'c)=0$ hold, and hence $1-cc'$ and $1-c'c$ are not isomorphisms. Since $C$ and $C'$ are endo-local by Proposition~\ref{PropARExt}{\rm (2)},
$c'c$ and $cc'$ are isomorphisms. Thus $c$ is an isomorphism. (1) is dual to (2).

(3) By {\rm (AS2)}, we have $\delta\mathrm{End}_{\mathscr{C}}(C)\simeq\mathrm{End}_{\mathscr{C}}(C)/\mathrm{rad}\mathrm{End}_{\mathscr{C}}(C)$, which is simple over $\mathrm{End}_{\mathscr{C}}(C)$.
If $c\in\mathrm{End}_{\mathscr{C}}(C)$ is an automorphism, then $\delta c$ is an almost split extension. Thus (2) implies that $\delta\mathrm{End}_{\mathscr{C}}(C)$ consists of $0$ and all almost split extensions in $\mathbb{E}(C,A)$. The dual argument shows that $\mathrm{End}_{\mathscr{C}}(A)\delta$ is simple over $\mathrm{End}_{\mathscr{C}}(A)$, and consists of $0$ and all almost split extensions in $\mathbb{E}(C,A)$. Thus $\mathrm{End}_{\mathscr{C}}(A)\delta=\delta\mathrm{End}_{\mathscr{C}}(C)$ holds. The last assertion follows from bijections $\mathrm{End}_{\mathscr{C}}(C)/\mathrm{rad}\mathrm{End}_{\mathscr{C}}(C)\simeq\delta\mathrm{End}_{\mathscr{C}}(C)=\mathrm{End}_{\mathscr{C}}(A)\delta\simeq\mathrm{End}_{\mathscr{C}}(A)/\mathrm{rad}\mathrm{End}_{\mathscr{C}}(A)$.
\end{proof}

Now we introduce the following central notion.

\begin{definition}\label{DefARSEQ}
A sequence of morphisms $A\overset{x}{\longrightarrow}B\overset{y}{\longrightarrow}C$ in $\mathscr{C}$ is called an {\it almost split sequence} if it realizes some almost split extension $\delta\in\mathbb{E}(C,A)$.
\end{definition}

The following class of morphisms is basic to study the structure of additive categories.
\begin{definition}
Let $\mathscr{C}$ be an additive category and $A$ an object in $\mathscr{C}$. A morphism $a\colon A\to B$ which is not a section is called \emph{left almost split} if any morphism $A\to B^{\prime}$ which is not a section factors through $a$.
Dually, a morphism $a\colon B\to A$ which is not a retraction is called \emph{right almost split} if any morphism $B^{\prime}\to A$ which is not a retraction factors through $a$.

A morphism $a:A\to B$ is called \emph{left minimal} if each morphism $b:B\to B$ satisfying $ba=a$ is an automorphism. Dually, a morphism $a:B\to A$ is called \emph{right minimal} if each morphism $b:B\to B$ satisfying $ab=a$ is an automorphism.

A right minimal right almost split morphism is called a \emph{sink morphism}. Dually, a left minimal left almost split morphism is called a \emph{source morphism}.
\end{definition}


We prove the following characterizations of almost split extensions, which is analogue of a standard result in classical Auslander--Reiten theory.

\begin{theorem}\label{PropASEquiv}
Let $\delta\in\mathbb{E}(C,A)$ be a non-zero element with $\mathfrak{s}(\delta)=[A\overset{x}{\longrightarrow}B\overset{y}{\longrightarrow}C]$. Then the following conditions are equivalent.
\begin{enumerate}[\rm(1)]
\item $\delta$ is an almost split extension.
\item {\rm (AS1)} holds and $C$ is endo-local.
\item $x$ is left almost split and $C$ is endo-local.
\item $x$ is a source morphism.
\item {\rm (AS2)} holds and $A$ is endo-local.
\item $y$ is right almost split and $A$ is endo-local.
\item $y$ is a sink morphism.
\end{enumerate}
\end{theorem}

In particular, our almost split sequence is nothing but an {\it Auslander--Reiten $\mathbb{E}$-triangle} in the sense of \cite[Definition 4.1]{ZZ} defined for a Krull--Schmidt extriangulated category.

We prepare the following.

\begin{proposition}\label{PropARSeq1}
For any $0\neq\delta\in\mathbb{E}(C,A)$ with $\mathfrak{s}(\delta)=[A\overset{x}{\longrightarrow}B\overset{y}{\longrightarrow}C]$, the following holds.
\begin{enumerate}[\rm(1)]
\item $\delta$ satisfies {\rm (AS1)} if and only if $x$ is a left almost split morphism.
\item $\delta$ satisfies {\rm (AS2)} if and only if $y$ is a right almost split morphism.
\end{enumerate}
\end{proposition}
\begin{proof}
{\rm (1)} follows from the exactness of $\mathscr{C}(B,-)\xrightarrow{x\circ-}\mathscr{C}(A,-)\xrightarrow{\delta\circ-}\mathbb{E}(C,-)$. Dually, {\rm (2)} follows from the exactness of $\mathscr{C}(-,B)\xrightarrow{-\circ y}\mathscr{C}(-,C)\xrightarrow{-\circ\delta}\mathbb{E}(-,A)$.
\end{proof}

\begin{proposition}\label{PropARSeq2}
For any $0\neq\delta\in\mathbb{E}(C,A)$ with $\mathfrak{s}(\delta)=[A\overset{x}{\longrightarrow}B\overset{y}{\longrightarrow}C]$, the following holds.
\begin{enumerate}[\rm(1)]
\item If $A$ is endo-local, 
then $y$ is right minimal.
\item If $C$ is endo-local, 
then $x$ is left minimal.
\end{enumerate}
\end{proposition}

\begin{proof}
{\rm (1)} Let $b\in\mathscr{C}(B,B)$ be any morphism satisfying $y b=y$. By {\rm (ET3)$^\mathrm{op}$}, there is some $a\in\mathscr{C}(A,A)$ which gives the following morphism of $\mathfrak{s}$-triangles.
\[
\xy
(-12,6)*+{A}="0";
(0,6)*+{B}="2";
(12,6)*+{C}="4";
(24,6)*+{}="6";
(-12,-6)*+{A}="10";
(0,-6)*+{B}="12";
(12,-6)*+{C}="14";
(24,-6)*+{}="16";
{\ar^{x} "0";"2"};
{\ar^{y} "2";"4"};
{\ar@{-->}^{\delta} "4";"6"};
{\ar_{a} "0";"10"};
{\ar^{b} "2";"12"};
{\ar@{=} "4";"14"};
{\ar_{x} "10";"12"};
{\ar_{y} "12";"14"};
{\ar@{-->}_{\delta} "14";"16"};
{\ar@{}|{} "0";"12"};
{\ar@{}|{} "2";"14"};
\endxy
\]
Since $a\delta=\delta$, then $1-a$ is not an isomorphism. Since $A$ is endo-local, $a$ is an isomorphism. Thus $b$ becomes an isomorphism as in \cite[Corollary 3.6]{NP}.
{\rm (2)} is dual to {\rm (1)}.
\end{proof}

\begin{proposition}\label{PropARSeq2.5}
For any $0\neq\delta\in\mathbb{E}(C,A)$ with $\mathfrak{s}(\delta)=[A\overset{x}{\longrightarrow}B\overset{y}{\longrightarrow}C]$, the following holds.
\begin{enumerate}[\rm(1)]
\item If $x$ is a source morphism, then {\rm (AS2)} holds.
\item If $y$ is a sink morphism, then {\rm (AS1)} holds.
\end{enumerate}
\end{proposition}

\begin{proof}
(2) Let $a\in\mathscr{C}(A,A')$ be a non-section.
As in the diagram \eqref{MorphRealize} in Definition \ref{DefRealization}, we have the following morphism of $\mathfrak{s}$-triangles.
\[
\xy
(-12,6)*+{A}="0";
(0,6)*+{B}="2";
(12,6)*+{C}="4";
(24,6)*+{}="6";
(-12,-6)*+{A'}="10";
(0,-6)*+{B'}="12";
(12,-6)*+{C}="14";
(24,-6)*+{}="16";
{\ar^{x} "0";"2"};
{\ar^{y} "2";"4"};
{\ar@{-->}^{\delta} "4";"6"};
{\ar_{a} "0";"10"};
{\ar^{b} "2";"12"};
{\ar@{=} "4";"14"};
{\ar_{x'} "10";"12"};
{\ar_{y'} "12";"14"};
{\ar@{-->}_{a\delta} "14";"16"};
{\ar@{}|{} "0";"12"};
{\ar@{}|{} "2";"14"};
\endxy
\]
Assume $a\delta\neq0$. Then $y'$ is a non-retraction. Since $y$ is right almost split, there exists $b'\in\mathscr{C}(B',B)$ such that $y'=yb'$. By {\rm (ET3)$^\mathrm{op}$}, there is some $a'\in\mathscr{C}(A',A)$ which gives the following morphism of $\mathfrak{s}$-triangles.
\[
\xy
(-12,6)*+{A'}="0";
(0,6)*+{B'}="2";
(12,6)*+{C}="4";
(24,6)*+{}="6";
(-12,-6)*+{A}="10";
(0,-6)*+{B}="12";
(12,-6)*+{C}="14";
(24,-6)*+{}="16";
{\ar^{x'} "0";"2"};
{\ar^{y'} "2";"4"};
{\ar@{-->}^{a\delta} "4";"6"};
{\ar_{a'} "0";"10"};
{\ar^{b'} "2";"12"};
{\ar@{=} "4";"14"};
{\ar_{x} "10";"12"};
{\ar_{y} "12";"14"};
{\ar@{-->}_{\delta} "14";"16"};
{\ar@{}|{} "0";"12"};
{\ar@{}|{} "2";"14"};
\endxy
\]
Since $y=yb'b$ and $y$ is right minimal, $b'b$ is an isomorphism.  Thus $a'a$ is an isomorphism by \cite[Corollary 3.6]{NP}, a contradiction to our choice of $a$. Thus $a\delta=0$ holds. (1) is dual to (2).
\end{proof}

\begin{proof}[Proof of Theorem \ref{PropASEquiv}]
(2)$\Rightarrow$(3) follows from Proposition \ref{PropARSeq1}, and (3)$\Rightarrow$(4) follows from Proposition \ref{PropARSeq2}. (4)$\Rightarrow$(5) holds since (AS2) holds by Proposition \ref{PropARSeq2.5}, and $A$ is endo-local by Propositions \ref{PropARSeq1} and \ref{PropARExt}.
Dually, (5)$\Rightarrow$(6)$\Rightarrow$(7)$\Rightarrow$(2) holds. Thus all conditions (2)--(7) are equivalent. Therefore (1) is also equivalent.
\end{proof}


The following observation will play a crucial role.

\begin{lemma}\label{LemARSeq3}
Let $0\ne\delta\in\mathbb{E}(C,A)$ be any $\mathbb{E}$-extension.
\begin{enumerate}[\rm(1)]
\item If $\delta$ satisfies {\rm (AS1)}, then the following holds for any $X\in\mathscr{C}$.
\begin{enumerate}[\rm(a)]
\item For any $0\ne\alpha\in\mathbb{E}(X,A)$, there exists $c\in\mathscr{C}(C,X)$ such that $\delta=\alpha c$.
\item For any $0\ne\overline{a}\in\overline{\mathscr{C}}(X,A)$, there exists $\gamma\in\mathbb{E}(C,X)$ such that $\delta=a\gamma$.
\end{enumerate}
\item If $\delta$ satisfies {\rm (AS2)}, then the following holds for any $X\in\mathscr{C}$.
\begin{enumerate}[\rm(c)]
\item For any $0\ne\gamma\in\mathbb{E}(C,X)$, there exists $a\in\mathscr{C}(X,A)$ such that $\delta=a\gamma$.
\item[{\rm (d)}] For any $0\ne\underline{c}\in\underline{\mathscr{C}}(C,X)$, there exists $\alpha\in\mathbb{E}(X,A)$ such that $\delta=\alpha c$.
\end{enumerate}
\end{enumerate}
\end{lemma}
\begin{proof}
{\rm (1)(a)} Realize $\delta$ and $\alpha$ as
\[ \mathfrak{s}(\delta)=[A\overset{x}{\longrightarrow}B\overset{y}{\longrightarrow}C],\quad\mathfrak{s}(\alpha)=[A\overset{f}{\longrightarrow}Y\overset{g}{\longrightarrow}X]. \]
Since $\alpha\ne0$, the morphism $f$ is not a section. Thus there is some $b\in\mathscr{C}(B,Y)$ satisfying $bx=f$ by Proposition~\ref{PropARSeq1}{\rm (1)}. By {\rm (ET3)}, there exists $c\in\mathscr{C}(C,X)$ which gives the following morphism of $\mathfrak{s}$-triangles.
\[
\xy
(-12,6)*+{A}="0";
(0,6)*+{B}="2";
(12,6)*+{C}="4";
(24,6)*+{}="6";
(-12,-6)*+{A}="10";
(0,-6)*+{Y}="12";
(12,-6)*+{X}="14";
(24,-6)*+{}="16";
{\ar^{x} "0";"2"};
{\ar^{y} "2";"4"};
{\ar@{-->}^{\delta} "4";"6"};
{\ar@{=} "0";"10"};
{\ar^{b} "2";"12"};
{\ar^{c} "4";"14"};
{\ar_{f} "10";"12"};
{\ar_{g} "12";"14"};
{\ar@{-->}_{\alpha} "14";"16"};
{\ar@{}|{} "0";"12"};
{\ar@{}|{} "2";"14"};
\endxy
\]
In particular it satisfies $\delta=\alpha c$.

{\rm (b)} Suppose that $a\in\mathscr{C}(X,A)$ does not belong to $\mathcal{I}$. By definition of $\mathcal{I}$, there exists $Y\in\mathscr{C}$ such that the map $a\circ -\colon\mathbb{E}(Y,X)\to\mathbb{E}(Y,A)$ is non-zero. Take $\zeta\in\mathbb{E}(Y,X)$ such that $a\zeta\ne0$. By {\rm (a)}, there exists $c\in\mathscr{C}(C,Y)$ such that $\delta=(a\zeta)c$. Thus $\gamma=\zeta c\in\mathbb{E}(C,X)$ satisfies the desired condition.

(2) is dual to (1).
\end{proof}

We give another characterization of almost split extensions.

\begin{proposition}\label{LemEquivARExt}
For a non-zero element $\delta\in\mathbb{E}(C,A)$, the following are equivalent.
\begin{enumerate}[\rm(1)]
\item $\delta$ is an almost split extension.
\item $\delta$ satisfies the following conditions.
\begin{itemize}
\item[{\rm (a)}] $A,C$ are endo-local.
\item[{\rm (b)}] For any $X\in\mathscr{C}$ and any non-split $\theta\in\mathbb{E}(X,A)$, there exists $c\in\mathscr{C}(C,X)$ such that $\theta c=\delta$.
Dually, For any $Y\in\mathscr{C}$ and any non-split $\mu\in\mathbb{E}(C,Y)$, there exists $a\in\mathscr{C}(Y,A)$ such that $a\mu=\delta$.
\end{itemize}
\end{enumerate}
\end{proposition}

\begin{proof}
(1)$\Rightarrow$(2) Immediate from Proposition~\ref{PropARExt} and Lemma~\ref{LemARSeq3}.

(2)$\Rightarrow$(1) By Theorem~\ref{PropASEquiv} it is enough to show {\rm (AS1)}. Assume that $a\in\mathscr{C}(A,A^{\prime})$ satisfies $a\delta\neq0$. Then by {\rm (b)}, there is some $f\in\mathscr{C}(A^{\prime},A)$ satisfying $\delta=fa\delta$. Since $(1-fa)\delta=0$ and $A$ is endo-local, $fa$ is an isomorphism. Thus $a$ is a section.
\end{proof}

Later we need the following easy observation.
\begin{lemma}\label{Rem2.5}
Suppose that $\mathscr{C}$ is Krull--Schmidt, and let $\delta\in\mathbb{E}(C,A)$ be a non-split $\mathbb{E}$-extension satisfying {\rm (AS2)}.
For a direct sum decomposition $A=A_1\oplus\cdots\oplus A_n$ into indecomposables $A_i\in\mathscr{C}$ $(1\le i\le n)$, we take $i$ such that $p_i\delta\ne0$, where $p_i\colon A\to A_i$ denotes the projection. Then $p_i\delta\in\mathbb{E}(C,A_i)$ is an almost split extension.
\end{lemma}
\begin{proof}
Clearly $p_i\delta$ satisfies (AS2). Thus the assertion is immediate from Theorem~\ref{PropASEquiv}(5)$\Rightarrow$(1).
\end{proof}

\section{Auslander--Reiten--Serre duality}\label{section_ARtr}

\subsection{Definitions and results}
As usual, for an additive category $\mathscr{C}$, we denote by ${\rm rad}\mathscr{C}$ its Jacobson radical. Throughout this section, let $(\mathscr{C},\mathbb{E},\mathfrak{s})$ be an extriangulated category.
Our aim in this section is to characterize when $\mathscr{C}$ has almost split extensions in the following sense.

\begin{definition}\label{have ASE}
We say that \emph{$\mathscr{C}$ has right almost split extensions} if for any endo-local non-projective object $A\in\mathscr{C}$, there exists an almost split extension $\delta\in\mathbb{E}(A,B)$ for some $B\in\mathscr{C}$.
Dually, we say that \emph{$\mathscr{C}$ has left almost split extensions} if for any endo-local non-injective object $B\in\mathscr{C}$, there exists an almost split extension $\delta\in\mathbb{E}(A,B)$ for some $A\in\mathscr{C}$.
We say that \emph{$\mathscr{C}$ has almost split extensions} if it has right and left almost split extensions.

We say that \emph{$\mathscr{C}$ has sink morphisms} if any endo-local object $A$ has a sink morphism $x\in\mathscr{C}(B,A)$. Dually, we define the condition that \emph{$\mathscr{C}$ has source morphisms}.
\end{definition}

There is the following obvious implication between these notions.

\begin{lemma}\label{sink imply right ASE}
Assume that $\mathscr{C}$ is Krull--Schmidt. If $\mathscr{C}$ has sink (resp.\ source) morphisms, then it has right (resp.\ left) almost split extensions.
\end{lemma}
\begin{proof}

Let $C\in\mathscr{C}$ be any endo-local (or equivalently, indecomposable) non-projective object, and $u\in\mathscr{C}(U,C)$ be a sink morphism.
By Proposition~\ref{PropARSeq1}(2) and Lemma~\ref{Rem2.5}, it suffices to show the existence of a deflation to $C$ which is right almost split.

Since $C$ is non-projective, there exists a non-split $\mathfrak{s}$-conflation $A\overset{x}{\longrightarrow}B\overset{y}{\longrightarrow}C$. By the dual of \cite[Corollary 3.16]{NP}, we can find an $\mathfrak{s}$-triangle
\[ D\to B\oplus U\overset{[y\ u]}{\longrightarrow}C\overset{\theta}{\dashrightarrow} \]
for some $D\in\mathscr{C}$ and $\theta\in\mathbb{E}(C,D)$. We remark that $y,u\in\mathrm{rad}\mathscr{C}$ implies $[y\ u]\in\mathrm{rad}\mathscr{C}$, which means $\theta\ne0$. Then $[y\ u]$ is right almost split, since $u$ factors through it.
\end{proof}

In this section, we fix a base field $k$, and we denote by $\Dd$ the $k$-dual.
We say that an extriangulated category $(\mathscr{C},\mathbb{E},\mathfrak{s})$ is \emph{$k$-linear} if $\mathscr{C}(A,B)$ and $\mathbb{E}(A,B)$ are $k$-vector spaces such that the following compositions are $k$-linear for any $A,B,C,D\in\mathscr{C}$.
\begin{eqnarray*}
&\mathscr{C}(A,B)\times\mathscr{C}(B,C)\to\mathscr{C}(A,C),&\\
&\mathscr{C}(A,B)\times\mathbb{E}(B,C)\times\mathscr{C}(C,D)\to\mathbb{E}(A,D).&
\end{eqnarray*}
Moreover we call $(\mathscr{C},\mathbb{E},\mathfrak{s})$ \emph{Ext-finite} if $\dim_k\mathbb{E}(A,B)<\infty$ holds for any $A,B\in\mathscr{C}$.

We start with the following observation.

\begin{proposition}\label{local theorem}
Let $\mathscr{C}$ be a $k$-linear Ext-finite extriangulated category.
For a non-projective endo-local object $A$ and a non-injective endo-local object $B$, the following conditions are equivalent.
\begin{enumerate}[\rm(1)]
\item There exists an almost split extension in $\mathbb{E}(A,B)$.
\item There exists an isomorphism $\underline{\mathscr{C}}(A,-)\simeq\Dd\mathbb{E}(-,B)$ of functors on $\mathscr{C}$.
\item There exists an isomorphism $\mathbb{E}(A,-)\simeq\Dd\overline{\mathscr{C}}(-,B)$ of functors on $\mathscr{C}$.
\end{enumerate}
\end{proposition}

\begin{proof}
(1)$\Rightarrow$(2)
Let $\delta\in\mathbb{E}(A,B)$ be an almost split extension.
Take any linear form $\eta\colon\mathbb{E}(A,B)\to k$ satisfying $\eta(\delta)\neq0$. It follows from Lemma~\ref{LemARSeq3} that for each $X\in\mathscr{C}$, the composition
\begin{eqnarray*}
\underline{\mathscr{C}}(A,X)\times\mathbb{E}(X,B)\to\mathbb{E}(A,B)\xrightarrow{\eta}k\ ;\ (\underline{a},\gamma)\mapsto \eta(\gamma a)
\end{eqnarray*}
is a non-degenerate bilinear form in the sense that the induced maps $\underline{\mathscr{C}}(A,X)\to\Dd\mathbb{E}(X,B)$, $a\mapsto\eta(-\circ a)$ and $\mathbb{E}(X,B)\to\Dd\underline{\mathscr{C}}(A,X)$, $\gamma\mapsto\eta(\gamma\circ-)$ are injective. Since $\mathscr{C}$ is Ext-finite, these maps are isomorphisms. Since they are functorial on $X$, we obtain (2).

(2)$\Rightarrow$(1) 
We have an isomorphism $\mathrm{End}_{\underline{\mathscr{C}}}(A)\simeq\Dd\mathbb{E}(A,B)$ of left $\mathrm{End}_{\mathscr{C}}(A)$-modules.
Since $A$ is endo-local, $T_A:=\mathrm{End}_{\underline{\mathscr{C}}}(A)/\mathrm{rad}\mathrm{End}_{\underline{\mathscr{C}}}(A)$ is a simple $\mathrm{End}_{\mathscr{C}}(A)$-module. 
Then $S_A:=\Dd(T_A)$ is a simple $\mathrm{End}_{\mathscr{C}}(A)$-submodule of $\mathbb{E}(A,B)$.
We show that any non-zero element in $S_A$ is an almost split extension.
By Theorem~\ref{PropASEquiv}(5)$\Rightarrow$(1), it suffices to show that $(S_A)a=0$ holds for any $a\in\mathscr{C}(A^{\prime},A)$ which is not a retraction.
Our isomorphism $\iota\colon\mathbb{E}(-,B)\simeq\Dd\underline{\mathscr{C}}(A,-)$ of functors on $\mathscr{C}$ gives a commutative diagram
\[\xymatrix@R=1.5em{
S_A\ar[r]\ar@{=}[d]&\mathbb{E}(A,B)\ar[d]^{\iota_A}\ar[rr]^{-\circ a}&&\mathbb{E}(A^{\prime},B)\ar[d]^{\iota_{A^{\prime}}}\\
\Dd(T_A)\ar[r]&\Dd\underline{\mathscr{C}}(A,A)\ar[rr]^{\Dd(\underline{a}\circ-)}&&\Dd\underline{\mathscr{C}}(A,A^{\prime}).
}\]
Since $a$ is not a retraction, the composition $\underline{\mathscr{C}}(A,A^{\prime})\xrightarrow{\underline{a}\circ-}\underline{\mathscr{C}}(A,A)\to T_A$ is zero. Therefore the composition $\Dd(T_A)\to \Dd\underline{\mathscr{C}}(A,A)\xrightarrow{\Dd(\underline{a}\circ-)} \Dd\underline{\mathscr{C}}(A,A^{\prime})$ is zero. By commutativity, we have $(S_A)a=0$ as desired.

(1)$\Rightarrow$(3) is shown similarly to (1)$\Rightarrow$(2) by using the bilinear form
\[\mathbb{E}(A,X)\times\overline{\mathscr{C}}(X,B)\to\mathbb{E}(A,B)\xrightarrow{\eta}k\ ;\ (\gamma,\overline{b})\mapsto \eta(b\gamma).\]

(3)$\Rightarrow$(1) is shown similarly to (2)$\Rightarrow$(1).
\end{proof}

Now we discuss global existence of almost split extensions in $\mathscr{C}$.
\begin{definition}\label{DefARDual}
Let $(\mathscr{C},\mathbb{E},\mathfrak{s})$ be a $k$-linear extriangulated category.
\begin{enumerate}[(1)]
 \item A \emph{right Auslander--Reiten--Serre (ARS) duality} is a pair $(\tau,\eta)$ of an additive functor $\tau\colon\underline{\mathscr{C}}\to\overline{\mathscr{C}}$ and a binatural isomorphism
\[\eta_{A,B}\colon\underline{\mathscr{C}}(A,B)\simeq \Dd\mathbb{E}(B,\tau A)\ \mbox{ for any }\ A,B\in\mathscr{C}.\]
 \item If moreover $\tau$ is an equivalence, we say that $(\tau,\eta)$ is an \emph{Auslander--Reiten--Serre (ARS) duality}.
\end{enumerate}
\end{definition}

\begin{remark}\label{RemIndec}
Let $(\tau,\eta)$ be a right Auslander--Reiten--Serre duality. Then $\tau:\underline{\mathscr{C}}\to\overline{\mathscr{C}}$ is fully faithful. In fact, for each $A,B\in\underline{\mathscr{C}}$, the inverse map of $\tau_{A,B}:\underline{\mathscr{C}}(A,B)\to\overline{\mathscr{C}}(\tau A,\tau B)$ is given as follows: Each $\overline{f}\in\overline{\mathscr{C}}(\tau A,\tau B)$ gives a morphism $\overline{f}\circ-:\mathbb{E}(-,\tau A)\to\mathbb{E}(-,\tau B)$ of functors on $\underline{\mathscr{C}}$. Applying $\mathbb{D}$ and using $\eta$, we obtain a morphism $\underline{\mathscr{C}}(B,-)\to\underline{\mathscr{C}}(A,-)$ of functors on $\underline{\mathscr{C}}$. By Yoneda's Lemma, this can be written as $-\circ\underline{g}$ by a unique morphism $\underline{g}\in\underline{\mathscr{C}}(A,B)$. Then the desired map is given by $\overline{f}\mapsto\underline{g}$.

In particular, if moreover $\mathscr{C}$ is Krull--Schmidt, then $\tau$ gives an injective map $\tau:\ind\underline{\mathscr{C}}=\{X\in\ind\mathscr{C}\mid\mbox{$X$ is non-projective}\}\to\ind\overline{\mathscr{C}}=\{X\in\ind\mathscr{C}\mid\mbox{$X$ is non-injective}\}$.
\end{remark}

Our aim in this section is to prove the following:
\begin{theorem}\label{main theorem}
Let $\mathscr{C}$ be a $k$-linear Ext-finite Krull--Schmidt extriangulated category. Then the following conditions are equivalent.
\begin{enumerate}[\rm(1)]
\item $\mathscr{C}$ has almost split extensions.
\item $\mathscr{C}$ has an Auslander--Reiten--Serre duality.
\end{enumerate}
\end{theorem}

This follows from the following more general result.

\begin{proposition}\label{right theorem}
Let $\mathscr{C}$ be a $k$-linear Ext-finite Krull--Schmidt extriangulated category. Then the following conditions are equivalent.
\begin{enumerate}[\rm(1)]
\item $\mathscr{C}$ has right almost split extensions.
\item $\mathscr{C}$ has a right Auslander--Reiten--Serre duality $(\tau,\eta)$.
\end{enumerate}
\end{proposition}

\subsection{Proofs of Proposition~\ref{right theorem} and Theorem~\ref{main theorem}}


It is convenient to start with the following general setting.
\begin{definition}
Let $(\mathscr{C},E,\mathscr{D})$ be a triple consisting of $k$-linear additive categories $\mathscr{C}$ and $\mathscr{D}$ and a $k$-bilinear functor $E\colon\mathscr{C}^\mathrm{op}\times\mathscr{D}\to\mod k$.
A \emph{right ARS duality} for $(\mathscr{C},E,\mathscr{D})$ is a pair $(F,\eta)$ of a $k$-linear functor $F\colon\mathscr{C}\to\mathscr{D}$ and a binatural isomorphism
\[\eta_{A,B}\colon\mathscr{C}(A,B)\simeq \Dd E(B,FA)\ \mbox{ for any }\ A,B\in\mathscr{C}.\]
If moreover $F$ is an equivalence, we say that $(F,\eta)$ is an \emph{ARS duality} for $(\mathscr{C},E,\mathscr{D})$.

Dually we define a \emph{left ARS duality} for $(\mathscr{C},E,\mathscr{D})$.
\end{definition}

The following is clear.

\begin{lemma}\label{construct left}
If $(F,\eta)$ is an ARS duality for $(\mathscr{C},E,\mathscr{D})$, then $(G,\zeta)$ is a left ARS duality for $(\mathscr{C},E,\mathscr{D})$, where $G$ is a quasi-inverse of $F$ and $\zeta_{A,B}$ is a composition
\[\mathscr{D}(A,B)\xrightarrow{G}\mathscr{C}(GA,GB)\xrightarrow{\eta_{GA,GB}}\Dd E(GB,FGA)\simeq \Dd E(GB,A)\]
for any $A,B\in\mathscr{D}$.
\end{lemma}

The following is an analogue of \cite[9.4]{GR} and \cite[I.1.4]{RV}.

\begin{lemma}\label{sufficient condition}
Let $(\mathscr{C},E,\mathscr{D})$ be a triple consisting of $k$-linear additive categories $\mathscr{C}$ and $\mathscr{D}$, and a $k$-bilinear functor $E\colon\mathscr{C}^\mathrm{op}\times\mathscr{D}\to\mod k$. Assume that we have the following.
\begin{enumerate}
\item[$\bullet$] A correspondence $F$ from objects in $\mathscr{C}$ to objects in $\mathscr{D}$.
\item[$\bullet$] A $k$-linear map $\eta_A\colon E(A,FA)\to k$ for any $A\in\mathscr{C}$ such that the compositions
\begin{eqnarray}\label{pairing}
&\mathscr{C}(A,B)\times E(B,FA)\to E(A,FA)\xrightarrow{\eta_A}k,&\\ \label{pairing 2}
&E(B,FA)\times\mathscr{D}(FA,FB)\to E(B,FB)\xrightarrow{\eta_B}k&
\end{eqnarray}
are non-degenerate bilinear forms for any $A,B\in\mathscr{C}$.
\end{enumerate}
Then we can extend $F$ to a fully faithful functor $F\colon\mathscr{C}\to\mathscr{D}$ such that the pair $(F,\eta)$ is a right ARS duality for $(\mathscr{C},E,\mathscr{D})$, where $\eta_{A,B}(f)(\delta)=\eta_A(\delta f)$.
\end{lemma}

\begin{proof}
Similarly as for $\mathbb{E}$, we use the notation
$E(c,d)(\gamma)=d\gamma c$
for any $\gamma\in E(C,D)$, $c\in\mathscr{C}(C^{\prime},C)$ and $d\in\mathscr{D}(D,D^{\prime})$.
Fix $A,B\in\mathscr{C}$. Since \eqref{pairing 2} is non-degenerate and $\dim_kE(B,FA)<\infty$, the induced map $\mathscr{D}(FA,FB)\to\Dd E(B,FA)$, $a\mapsto\eta_B(a\circ-)$ is an isomorphism. Thus for any $a\in\mathscr{C}(A,B)$, there exists a unique $a^{\prime}\in\mathscr{D}(FA,FB)$ such that
\begin{equation}\label{define F}
\eta_A(\gamma a)=\eta_B(a^{\prime}\gamma)\ \mbox{ for any }\ \gamma\in E(B,FA).\end{equation}
Writing $F(a):=a^{\prime}$, we have a map
\[F\colon\mathscr{C}(A,B)\to\mathscr{D}(FA,FB).\]
This is clearly a morphism of abelian groups. It is clear from definition that $F(1_A)=1_{FA}$ holds.

To prove that $F$ is a functor, fix $a\in\mathscr{C}(A,B)$ and $b\in\mathscr{C}(B,C)$. For any $\gamma\in E(C,FA)$, using \eqref{define F} three times, we have
\[ \eta_C(F(ba)\gamma)=\eta_A(\gamma(ba))=\eta_B(F(a)\gamma b)=\eta_C((F(b)F(a))\gamma).  \]
Since \eqref{pairing 2} is non-degenerate and $E(B,FA)$ is finite dimensional over $k$, this implies $F(ba)=F(b)F(a)$. Thus $F$ is a functor.

Moreover, for any $a^{\prime}\in\mathscr{D}(FA,FB)$, there exists unique $a\in\mathscr{C}(A,B)$ satisfying \eqref{define F}. Thus the functor $F\colon\mathscr{C}\to\mathscr{D}$ is fully faithful.

The non-degenerate bilinear form \eqref{pairing} gives an isomorphism
\[\eta_{A,B}\colon\mathscr{C}(A,B)\to \Dd E(B,FA)\]
for any $A,B\in\mathscr{C}$. To show that $\eta$ is binatural, fix $a\in\mathscr{C}(A^{\prime},A)$ and $b\in\mathscr{C}(B,B^{\prime})$, and consider the diagram
\[\xymatrix{
\mathscr{C}(A,B)\ar[rr]^{\eta_{A,B}}\ar[d]_{b\circ-\circ a}&&\Dd E(B,FA)\ar[d]^{\Dd E(b,F(a))}\\
\mathscr{C}(A^{\prime},B^{\prime})\ar[rr]^{\eta_{A^{\prime},B^{\prime}}}&&\Dd E(B^{\prime},FA^{\prime}).\\
}\]
This is commutative since, for any $f\in\mathscr{C}(A,B)$ and $\gamma\in E(B^{\prime},FA^{\prime})$, we have
\begin{eqnarray*}
\big(\Dd E(b,F(a))\eta_{A,B}\big)(f)(\gamma)&=&\eta_{A,B}(f)(F(a)\gamma b)%
\ =\ \eta_A((F(a)\gamma b)f)\\
&=&\eta_{A^{\prime}}(((\gamma b)f)a)\ =\ \eta_{A^{\prime},B^{\prime}}(b f a)(\gamma).
\end{eqnarray*}
Thus $\eta$ is binatural.
\end{proof}

Now we are ready to prove Proposition~\ref{right theorem}.

\begin{proof}[Proof of Proposition~\ref{right theorem}]
(2)$\Rightarrow$(1) is immediate from Proposition~\ref{local theorem}.

(1)$\Rightarrow$(2) For an indecomposable non-projective object $A$, by appealing to Proposition~\ref{local theorem}, we fix some object $FA$ such that $\underline{\mathscr{C}}(A,-)\simeq\Dd\mathbb{E}(-,FA)$ and we denote by $\delta_A\in\mathbb{E}(A,FA)$ an almost split extension.
Take any linear form $\eta_A\colon\mathbb{E}(A,FA)\to k$ satisfying $\eta_A(\delta_A)\neq0$. It follows from Lemma~\ref{LemARSeq3} that the bilinear forms
\begin{eqnarray*}
&\underline{\mathscr{C}}(A,-)\times\mathbb{E}(-,FA)\to\mathbb{E}(A,FA)\xrightarrow{\eta_A}k,&\\
&\mathbb{E}(A,-)\times\overline{\mathscr{C}}(-,FA)\to\mathbb{E}(A,FA)\xrightarrow{\eta_A}k&
\end{eqnarray*}
are non-degenerate. By the Krull--Schmidt property, we can extend this to any object in $\underline{\mathscr{C}}$. Applying Lemma~\ref{sufficient condition} to $(\mathscr{C},E,\mathscr{D}):=(\underline{\mathscr{C}},\mathbb{E},\overline{\mathscr{C}})$, we have a right ARS duality $(F,\eta)$ such that  $F\colon\underline{\mathscr{C}}\to\overline{\mathscr{C}}$ is fully faithful.
\end{proof}

Finally we prove Theorem~\ref{main theorem}.

\begin{proof}[Proof of Theorem~\ref{main theorem}]
(2)$\Rightarrow$(1) Assume that $\mathscr{C}$ has an ARS duality $(\tau,\eta)$.
Since this is a right ARS duality, $\mathscr{C}$ has right almost split extensions by Proposition~\ref{right theorem}.
By Lemma~\ref{construct left}, $\mathscr{C}$ has a left ARS duality. Therefore it has left almost split extensions by the dual of Proposition~\ref{right theorem}.

(1)$\Rightarrow$(2) By Proposition~\ref{right theorem}, $\mathscr{C}$ has a right ARS duality $(\tau,\eta)$ and $\tau\colon\underline{\mathscr{C}}\to\overline{\mathscr{C}}$ is fully faithful.
It remains to show that $\tau$ is dense. This follows from our assumption that $\mathscr{C}$ has left almost split extensions since $\tau$ sends $C$ to $A$ for each almost split extension $\delta\in\mathbb{E}(C,A)$.
\end{proof}

\subsection{Auslander--Reiten quivers}

As in the classical cases, we introduce the Auslander--Reiten quivers of extriangulated categories.
For a Krull--Schmidt category $\mathscr{C}$, we use the notation
\[({\rm rad}^i\mathscr{C}/{\rm rad}^{i+1}\mathscr{C})(A,B):=\frac{{\rm rad}^i\mathscr{C}(A,B)}{{\rm rad}^{i+1}\mathscr{C}(A,B)}.\]

\begin{definition}\label{define AR quiver}
\begin{enumerate}[\rm(1)]
\item A \emph{valued quiver} is a triple $Q=(Q_0,d,d')$ consisting of a set $Q_0$ and maps $d,d':Q_0\times Q_0\to\mathbb{Z}_{\ge0}$. It is called \emph{locally finite} if $\underset{Y\in Q_0}{\sum}d_{YX}<\infty$ and $\underset{Y\in Q_0}\sum d'_{XY}<\infty$ hold for each $X\in Q_0$. It is called \emph{symmetrizable} if there exists a map $c:Q_0\to \mathbb{Z}_{>0}$ such that $c_Xd_{XY}=d'_{XY}c_Y$ holds for each $X,Y\in Q_0$. In this case,  $c$ is called a \emph{symmetrizer}.

We often visualize $Q$ by regarding elements of $Q_0$ as vertices, and drawing a valued arrow $X\xrightarrow{(d_{XY},d^{\prime}_{XY})}Y$ for each $(X,Y)\in Q_0\times Q_0$ satisfying $d_{XY}+d'_{XY}\neq0$.
\item Let $\mathscr{C}$ be a Krull--Schmidt category.
For $X,Y\in\ind\mathscr{C}$, let
\begin{eqnarray*}
D_X:=(\mathscr{C}/{\rm rad}\mathscr{C})(X,X),&&{\rm Irr}(X,Y):=({\rm rad}\mathscr{C}/{\rm rad}^2\mathscr{C})(X,Y),\\
d_{XY}=\dim{\rm Irr}(X,Y)_{D_X},&&d^{\prime}_{XY}=\dim_{D_Y}{\rm Irr}(X,Y).
\end{eqnarray*}
The valued quiver $(\ind\mathscr{C},d,d')$ is called the \emph{Auslander--Reiten quiver} $\mathrm{AR}(\mathscr{C})$ of $\mathscr{C}$.
\end{enumerate}
\end{definition}

It is classical that $\mathrm{AR}(\mathscr{C})$ describes the terms of the sink and source morphisms.

\begin{proposition}\label{AR describe sink}
Let $\mathscr{C}$ be a Krull--Schmidt category, and $X\in\ind\mathscr{C}$.
\begin{enumerate}[\rm(1)]
\item If $a\in\mathscr{C}(Y,X)$ is a sink morphism, then it gives an isomorphism $a\circ-\colon(\mathscr{C}/{\rm rad}\mathscr{C})(-,Y)\simeq({\rm rad}\mathscr{C}/{\rm rad}^2\mathscr{C})(-,X)$.
Thus $Y\simeq\bigoplus_{W\in\ind\mathscr{C}}W^{\oplus d_{WX}}$ holds, and $d_{WX}<\infty$ for all $W\in\ind\mathscr{C}$.
\item If $b\in\mathscr{C}(X,Z)$ is a source morphism, then it gives an isomorphism $-\circ b\colon(\mathscr{C}/{\rm rad}\mathscr{C})(Z,-)\simeq({\rm rad}\mathscr{C}/{\rm rad}^2\mathscr{C})(X,-)$.
Thus $Z\simeq\bigoplus_{W\in\ind\mathscr{C}}W^{\oplus d^{\prime}_{XW}}$ holds, and $d^{\prime}_{XW}<\infty$ for all $W\in\ind\mathscr{C}$.
\end{enumerate}
\end{proposition}

\begin{proof}
Although this is well-known, we include a complete proof of (1) for convenience of the reader. We start with proving the first statement.
Since $a$ is right almost split, $a\circ-\colon\mathscr{C}(-,Y)\to{\rm rad}\mathscr{C}(-,X)$ is an epimorphism.
Thus it induces an epimorphism $a\circ-\colon(\mathscr{C}/{\rm rad}\mathscr{C})(-,Y)\to({\rm rad}\mathscr{C}/{\rm rad}^2\mathscr{C})(-,X)$.
To prove that this is a monomorphism, assume that $f\in\mathscr{C}(Z,Y)$ satisfies $af\in{\rm rad}^2\mathscr{C}(Z,X)$. Write $af=hg$ for $g\in{\rm rad}\mathscr{C}(Z,W)$ and $h\in{\rm rad}\mathscr{C}(W,X)$.
Take $h^{\prime}\in\mathscr{C}(W,Y)$ such that $h=ah^{\prime}$.
Since $a(f-h^{\prime}g)=af-hg=0$ holds and $a$ is right minimal, $f-h^{\prime}g$ belongs to ${\rm rad}\mathscr{C}$.
Thus $f=(f-h^{\prime}g)+h^{\prime}g$ also belongs to ${\rm rad}\mathscr{C}$, as desired.

By the first statement, for any $W\in\ind\mathscr{C}$, we have
\[\dim(\mathscr{C}/{\rm rad}\mathscr{C})(W,Y)_{D_W}=\dim({\rm rad}\mathscr{C}/{\rm rad}^2\mathscr{C})(W,X)_{D_W}=d_{WX}.\]
Thus the multiplicity of $W$ in $Y$ is $d_{WX}$, as desired.
\end{proof}

We immediately obtain the following basic properties of $\mathrm{AR}(\mathscr{C})$.

\begin{lemma}\label{AR quiver basic}
Let $\mathscr{C}$ be a Krull--Schmidt category.
\begin{enumerate}[\rm(1)]
\item $\mathrm{AR}(\mathscr{C})$ is locally finite if $\mathscr{C}$ has sink morphisms and source morphisms.
\item $\mathrm{AR}(\mathscr{C})$ is symmetrizable if $\mathscr{C}$ is $k$-linear over a field $k$ and $\mathscr{C}/{\rm rad}\mathscr{C}$ is Hom-finite. A symmetrizer is given by $c_X:=\dim_kD_X$ for each $X\in\ind\mathscr{C}$.
\end{enumerate}
\end{lemma}

\begin{proof}
(1) is immediate from Proposition \ref{AR describe sink}. (2) follows from
\[c_Xd_{XY}=\dim_kD_X\cdot\dim{\rm Irr}(X,Y)_{D_X}
=\dim_k{\rm Irr}(X,Y)=\dim_{D_Y}{\rm Irr}(X,Y)\cdot\dim_kD_Y=d'_{XY}c_Y.\qedhere\]
\end{proof}

The Auslander--Reiten quivers of extriangulated categories have the following structure.

\begin{definition}\label{define valued translation quiver}
\begin{enumerate}[\rm(1)]
\item A \emph{$\tau$-quiver} (=\emph{valued translation quiver}) is a quadruple $Q=(Q_0,d,d',\tau)$ consisting of the following data.
\begin{enumerate}[$\bullet$]
\item $(Q_0,d,d')$ is a valued quiver.
\item $\tau\colon Q_0\setminus Q_0^p\to Q_0\setminus Q_0^i$ is a bijection for some subsets $Q_0^p$ and $Q_0^i$ of $Q_0$. We visualize $\tau$ by drawing a dashed arrow $\xymatrix@C2em{X\ar@{-->}[r]&\tau X}$ for each $X\in Q_0\setminus Q_0^p$.
\item For $X\in Q_0\setminus Q_0^p$ and $Y\in Q_0$, $d_{YX}=d'_{\tau X,Y}$ holds.
\end{enumerate}
It is called \emph{locally finite} if $(Q,d,d')$ is locally finite. It is called \emph{symmetrizable} if $(Q,d,d')$ is symmetrizable and has a symmetrizer $c:Q_0\to\mathbb{Z}_{>0}$ satisfying $c_X=c_{\tau X}$ for each $X\in Q_0\setminus Q_0^p$.
It is called \emph{stable} if $Q_0^p=\emptyset=Q_0^i$.
\item Let $\mathscr{C}$ be a Krull--Schmidt extriangulated category with almost split extensions (Definition \ref{have ASE}). The \emph{Auslander--Reiten quiver} $\mathrm{AR}_{\rm ET}(\mathscr{C})=(Q_0,d,d',\tau)$ of $\mathscr{C}$ is defined as follows.
\begin{enumerate}[$\bullet$]
\item $(Q_0,d,d'):=\mathrm{AR}(\mathscr{C})$ (Definition \ref{define AR quiver}).
\item $Q_0^p:=\ind\mathrm{Proj}_{\mathbb{E}}\mathscr{C}$, $Q_0^i:=\ind\mathrm{Inj}_{\mathbb{E}}\mathscr{C}$.
\item $\tau C:=A$ if there exists an almost split extension in $\mathbb{E}(C,A)$ with $A,C\in Q_0$.
\end{enumerate}
\end{enumerate}
\end{definition}

Then $\mathrm{AR}_{\rm ET}(\mathscr{C})$ is a $\tau$-quiver by the following observation.

\begin{proposition}\label{AR quiver is translation}
Let $\mathscr{C}$ be a Krull--Schmidt extriangulated category with almost split extensions.
\begin{enumerate}[\rm(1)]
\item $\mathrm{AR}_{\rm ET}(\mathscr{C})$ is a $\tau$-quiver.
\item $\mathrm{AR}_{\rm ET}(\mathscr{C})$ is locally finite if $\mathscr{C}$ has sink morphisms and source morphisms.
\item $\mathrm{AR}_{\rm ET}(\mathscr{C})$ is symmetrizable if $\mathscr{C}$ is $k$-linear over a field $k$ and $\mathscr{C}/{\rm rad}\mathscr{C}$ is Hom-finite. A symmetrizer is given by $c_X:=\dim_kD_X$ for each $X\in\ind\mathscr{C}$.
\end{enumerate}
\end{proposition}

\begin{proof}
(1) $\tau$ is a well-defined bijection $Q_0\setminus Q_0^p\to Q_0\setminus Q_0^i$ by Proposition \ref{PropARSeq3} and our assumption that $\mathscr{C}$ has almost split extensions.
Let $\delta\in\mathbb{E}(C,A)$ be an almost split extension with $\mathfrak{s}(\delta)=[A\overset{x}{\longrightarrow}B\overset{y}{\longrightarrow}C]$. Then $x$ is a source morphism and $y$ is a sink morphism by Theorem \ref{PropASEquiv}. Thus for each $Y\in Q_0$, both $d_{YC}$ and $d'_{AY}$ give the multiplicity of $Y$ in $B$ by Proposition \ref{AR describe sink}, and hence $d_{YC}=d'_{AY}$ holds.

(2)(3) Immediate from Lemma \ref{AR quiver basic} and an isomorphism $D_X\simeq D_{\tau X}$ of $k$-algebras given in Proposition \ref{PropARSeq3}(3).
\end{proof}

\section{Stable module theory for extriangulated categories}\label{section_ARst}

\subsection{Definitions and results}

We develop the stable module theory for extriangulated categories following a series of works \cite{AR3} by Auslander--Reiten.

Throughout, let $(\mathscr{C},\mathbb{E},\mathfrak{s})$ be an extriangulated category. We denote by $\underline{\mathscr{C}}$ the stable category and by $\overline{\mathscr{C}}$ the costable category (see Definition~\ref{DefCoStable}). Notice that we do \emph{not} assume that $\mathscr{C}$ is Krull--Schmidt and/or $k$-linear for some field $k$.

We recall basic notions for functor categories.
For an additive category $\mathscr{D}$, a \emph{$\mathscr{D}$-module} is a contravariant additive functor from $\mathscr{D}$ to $\mathit{Ab}$. A \emph{morphism} between $\mathscr{D}$-modules is a natural transformation.
We denote by $\Mod\mathscr{D}$ the category of $\mathscr{D}$-modules, which forms an abelian category.
A $\mathscr{D}$-module $F$ is \emph{finitely presented} if there exists an exact sequence
\[\mathscr{D}(-,B)\to\mathscr{D}(-,A)\to F\to0\]
for some $A,B\in\mathscr{D}$. We denote by $\mod\mathscr{D}$ the category of finitely presented $\mathscr{D}$-modules. It is well-known (see \cite{A}) that the following conditions are equivalent.
\begin{enumerate}[$\bullet$]
\item $\mod\mathscr{D}$ forms an abelian category.
\item $\mod\mathscr{D}$ is closed under kernels.
\item $\mathscr{D}$ has weak kernels.
\end{enumerate}
In this case, we denote by $\proj\mathscr{D}$ (resp.\ $\inj\mathscr{D}$) the full subcategory of $\mod\mathscr{D}$ of projective (resp.\  injective) objects.

The following fundamental result generalizes the classical result due to Auslander--Reiten \cite{AR3} for the case $\mathscr{C}=\mod\mathscr{D}$ where $\mathscr{D}$ is a dualizing $k$-variety, as defined below. We call an additive functor $F:\mathscr{C}\to\mathscr{D}$ an \emph{equivalence up to direct summands} if it is fully faithful, and for each $D\in\mathscr{D}$, there exists $C\in\mathscr{C}$ such that $D$ is a direct summand of $FC$.

\begin{theorem}\label{fundamental theorem}
Let $\mathscr{C}$ be an extriangulated category with enough projectives and injectives.
\begin{enumerate}[\rm(1)]
\item $\mod\underline{\mathscr{C}}$ is an abelian category with enough projectives $\proj\underline{\mathscr{C}}=\mathrm{add}\{\underline{\mathscr{C}}(-,A)\mid A\in\mathscr{C}\}$ and enough injectives $\inj\underline{\mathscr{C}}=\mathrm{add}\{\mathbb{E}(-,A)\mid A\in\mathscr{C}\}$.
We have equivalences $\underline{\mathscr{C}}\to\proj\underline{\mathscr{C}}$ given by $A\mapsto\underline{\mathscr{C}}(-,A)$ and $\overline{\mathscr{C}}\to\inj\underline{\mathscr{C}}$ given by $A\mapsto\mathbb{E}(-,A)$ up to direct summands.
\item $\mod\overline{\mathscr{C}}^\mathrm{op}$ is an abelian category with enough projectives $\proj\overline{\mathscr{C}}^\mathrm{op}=\mathrm{add}\{\overline{\mathscr{C}}(A,-)\mid A\in\mathscr{C}\}$ and enough injectives $\inj\overline{\mathscr{C}}^\mathrm{op}=\mathrm{add}\{\mathbb{E}(A,-)\mid A\in\mathscr{C}\}$.
We have equivalences $\overline{\mathscr{C}}^\mathrm{op}\to\proj\overline{\mathscr{C}}^\mathrm{op}$ given by $A\mapsto\overline{\mathscr{C}}(A,-)$ and $\underline{\mathscr{C}}^\mathrm{op}\to\inj\overline{\mathscr{C}}^\mathrm{op}$ given by $A\mapsto\mathbb{E}(A,-)$ up to direct summands.
\end{enumerate}
\end{theorem}

The following is an immediate consequence.
\begin{proposition}\label{finiteness}
Let $\mathscr{C}$ be a $k$-linear extriangulated category with enough projectives and injectives.
Then $\mathscr{C}$ is Ext-finite if and only if $\overline{\mathscr{C}}$ is Hom-finite if and only if $\underline{\mathscr{C}}$ is Hom-finite.
\end{proposition}

\begin{proof}
We only prove the first equivalence.
For any $A\in\mathscr{C}$, there is an epimorphism $\underline{\mathscr{C}}(-,B)\to\mathbb{E}(-,A)$ and a monomorphism $\underline{\mathscr{C}}(-,A)\to\mathbb{E}(-,C)$ for some $B,C\in\mathscr{C}$ by Theorem~\ref{fundamental theorem}(1).
Thus $\mathscr{C}$ is Ext-finite if and only if $\underline{\mathscr{C}}$ is Hom-finite.
\end{proof}

Let $\mathscr{D}$ be a $k$-linear additive category. Then any $\mathscr{D}$-module $F$ can be regarded as a contravariant $k$-linear functor $F\colon\mathscr{D}\to\mathrm{Mod} k$. We define a $\mathscr{D}^\mathrm{op}$-module $\Dd F$ as the composition $\mathscr{D}\xrightarrow{F}\mathrm{Mod} k\xrightarrow{\Dd}\mathrm{Mod} k$.

\begin{definition}\cite{AR3}
We call $\mathscr{D}$ a \emph{dualizing $k$-variety} if the following conditions hold.
\begin{enumerate}[$\bullet$]
\item $\mathscr{D}$ is $k$-linear, Hom-finite and Krull--Schmidt.
\item For any $F\in\mod\mathscr{D}$, we have $\Dd F\in\mod\mathscr{D}^\mathrm{op}$.
\item For any $G\in\mod\mathscr{D}^\mathrm{op}$, we have $\Dd G\in\mod\mathscr{D}$.
\end{enumerate}
In this case, we have an equivalence $\Dd\colon\mod\mathscr{D}\simeq\mod\mathscr{D}^\mathrm{op}$.
\end{definition}

Now we have the following main result in Auslander--Reiten theory for extriangulated categories.
\begin{theorem}\label{main theorem 2}
Let $\mathscr{C}$ be an Ext-finite extriangulated category with enough projectives and enough injectives such that $\overline{\mathscr{C}}$ and $\underline{\mathscr{C}}$ are idempotent complete. 
Then the following conditions are equivalent.
\begin{enumerate}[\rm(1)]
\item $\mathscr{C}$ has an Auslander--Reiten--Serre duality.
\item $\overline{\mathscr{C}}$ is a dualizing $k$-variety.
\item $\underline{\mathscr{C}}$ is a dualizing $k$-variety.
\end{enumerate}
\end{theorem}

As an immediate consequence, we obtain the following result.

\begin{corollary}
Let $\mathscr{C}$ be an Ext-finite, Krull--Schmidt, extriangulated category with enough projectives and injectives. 
Then the following conditions are equivalent.
\begin{enumerate}[\rm(1)]
\item $\mathscr{C}$ has almost split extensions.
\item $\mathscr{C}$ has an Auslander--Reiten--Serre duality.
\item $\overline{\mathscr{C}}$ is a dualizing $k$-variety.
\item $\underline{\mathscr{C}}$ is a dualizing $k$-variety.
\end{enumerate}
\end{corollary}

\subsection{Proof of Theorems \ref{fundamental theorem} and \ref{main theorem 2}}
We frequently use the following observation.

\begin{lemma}\label{make deflation}
Let $\mathscr{C}$ be an extriangulated category with enough projectives.
For any morphism $f\colon A\to B$ in $\mathscr{C}$, there exist an $\mathfrak{s}$-deflation $f^{\prime}\colon A^{\prime}\to B$ in $\mathscr{C}$ and an isomorphism $\underline{h}\colon A\simeq A^{\prime}$ in $\underline{\mathscr{C}}$ satisfying $\underline{f}=\underline{f^{\prime} h}$.
\end{lemma}

\begin{proof}
Take an $\mathfrak{s}$-deflation $g\colon P\to B$ with a projective object $P\in\mathscr{C}$. Then $f^{\prime}:=[f\ g]\colon A\oplus P\to B$ is an $\mathfrak{s}$-deflation by the dual of \cite[Corollary 3.16]{NP}, and satisfies the desired property.
\end{proof}

Next we show the following property, which is not necessarily true for $\mathscr{C}$ itself.

\begin{proposition}\label{weak kernels}
Let $\mathscr{C}$ be an extriangulated category.
\begin{enumerate}[\rm(1)]
\item If $\mathscr{C}$ has enough projectives, then $\underline{\mathscr{C}}$ has weak kernels and $\mod\underline{\mathscr{C}}$ forms an abelian category.
\item If $\mathscr{C}$ has enough injectives, then $\overline{\mathscr{C}}$ has weak cokernels and $\mod\overline{\mathscr{C}}^\mathrm{op}$ forms an abelian category.
\end{enumerate}
\end{proposition}

\begin{proof}
We only prove (1). By Lemma~\ref{make deflation}, any morphism in $\underline{\mathscr{C}}$ can be represented by an $\mathfrak{s}$-deflation. Then the assertion follows from the long exact sequence associated with $\mathfrak{s}$-conflations (Theorem~\ref{long exact sequence}).
\end{proof}

Now, for a given $F\in\mod\underline{\mathscr{C}}$, we construct certain exact sequences.
Since $F$ is finitely presented, there exists an exact sequence
\[\underline{\mathscr{C}}(-,B)\xrightarrow{a\circ-}\underline{\mathscr{C}}(-,A)\to F\to0.\]
Without loss of generality, we can assume that there exists an $\mathfrak{s}$-conflation $C\to B\xrightarrow{a} A$ by Lemma~\ref{make deflation}.
By Theorem~\ref{long exact sequence}, we have exact sequences
\begin{eqnarray}\label{3 terms 1}
&\underline{\mathscr{C}}(-,C)\to\underline{\mathscr{C}}(-,B)\to\underline{\mathscr{C}}(-,A)\to F\to0,&\\ \label{3 terms 2}
&0\to F\to\mathbb{E}(-,C)\to\mathbb{E}(-,B)\to\mathbb{E}(-,A).&
\end{eqnarray}
The first sequence is the first three terms of a projective resolution of $F$. By the next observation, the second one is the first three terms of an injective resolution of $F$.

\begin{proposition}\label{E is injective}
Let $\mathscr{C}$ be an extriangulated category with enough projectives and injectives.
\begin{enumerate}[\rm(1)]
\item For any $X\in\mathscr{C}$, the $\underline{\mathscr{C}}$-module $\mathbb{E}(-,X)$  is finitely presented and injective in $\mod\underline{\mathscr{C}}$.
\item For any $X,Y\in\mathscr{C}$, we have $\overline{\mathscr{C}}(X,Y)\simeq(\mod\underline{\mathscr{C}})(\mathbb{E}(-,X),\mathbb{E}(-,Y))$ given by $a\mapsto (a\circ-)$.
\item For any $X\in\mathscr{C}$, the $\overline{\mathscr{C}}^\mathrm{op}$-module $\mathbb{E}(X,-)$ is finitely presented and injective in $\mod\overline{\mathscr{C}}^\mathrm{op}$.
\item For any $X,Y\in\mathscr{C}$, we have $\underline{\mathscr{C}}(X,Y)\simeq(\mod\overline{\mathscr{C}}^\mathrm{op})(\mathbb{E}(Y,-),\mathbb{E}(X,-))$ given by $a\mapsto (-\circ a)$.
\end{enumerate}
\end{proposition}

\begin{proof}
We only prove (1) and (2) since (3) and (4) are the dual statements.

Let $X\to I\to X^{\prime}$ be an $\mathfrak{s}$-conflation with an injective object $I\in\mathscr{C}$. By Theorem~\ref{long exact sequence}, we have an exact sequence
\[\underline{\mathscr{C}}(-,I)\to\underline{\mathscr{C}}(-,X^{\prime})\to\mathbb{E}(-,X)\to\mathbb{E}(-,I)=0.\]
Thus $\mathbb{E}(-,X)$ is a finitely presented $\underline{\mathscr{C}}$-module.
Put $\mathscr{M}=\mod\underline{\mathscr{C}}$. Applying $\mathscr{M}(-,\mathbb{E}(-,Y))$ and using Yoneda's Lemma, we have a commutative diagram of exact sequences
\[
\xy
(-47,7)*+{0}="0";
(-21,7)*+{\mathscr{M}(\mathbb{E}(-,X),\mathbb{E}(-,Y))}="2";
(19,7)*+{\mathscr{M}(\underline{\mathscr{C}}(-,X^{\prime}),\mathbb{E}(-,Y))}="4";
(58,7)*+{\mathscr{M}(\underline{\mathscr{C}}(-,I),\mathbb{E}(-,Y))}="6";
(-43,-7)*+{0=\overline{\mathscr{C}}(I,Y)}="10";
(-21,-7)*+{\overline{\mathscr{C}}(X,Y)}="12";
(19,-7)*+{\mathbb{E}(X^{\prime},Y)}="14";
(58,-7)*+{\mathbb{E}(I,Y)}="16";
{\ar "0";"2"};
{\ar "2";"4"};
{\ar "4";"6"};
{\ar "12";"2"};
{\ar@{=} "4";"14"};
{\ar@{=} "6";"16"};
{\ar "10";"12"};
{\ar "12";"14"};
{\ar "14";"16"};
{\ar@{}|{} "2";"14"};
{\ar@{}|{} "4";"16"};
\endxy
\]
where the lower sequence is exact by Theorem~\ref{long exact sequence}. Thus the assertion (2) follows.

For any $F\in\mathscr{M}$, consider the sequence \eqref{3 terms 1}.
Applying $\mathscr{M}(-,\mathbb{E}(-,X))$ and using Yoneda's Lemma, we have a commutative diagram
\[\xymatrix{
\mathscr{M}(\underline{\mathscr{C}}(-,A),\mathbb{E}(-,X))\ar[r]\ar@{=}[d]&\mathscr{M}(\underline{\mathscr{C}}(-,B),\mathbb{E}(-,X))\ar[r]\ar@{=}[d]&\mathscr{M}(\underline{\mathscr{C}}(-,C),\mathbb{E}(-,X))\ar@{=}[d]\\
\mathbb{E}(A,X)\ar[r]&\mathbb{E}(B,X)\ar[r]&\mathbb{E}(C,X)
}\]
where the homology of the upper sequence is $\mathrm{Ext}^1_{\mathscr{M}}(F,\mathbb{E}(-,X))$.
This is zero since the lower sequence is exact by Fact~\ref{FactExact}(2).
\end{proof}

We are ready to prove  Theorem~\ref{fundamental theorem}.

\begin{proof}[Proof of Theorem~\ref{fundamental theorem}]
We only prove (1) since (2) is dual.

The assertions for projectives follow from Yoneda's Lemma.
The assertions for injectives follow from Proposition~\ref{E is injective} and the exact sequence \eqref{3 terms 2}.
\end{proof}

Now we are ready to prove Theorem~\ref{main theorem 2}.

\begin{proof}[Proof of Theorem~\ref{main theorem 2}]
(1)$\Rightarrow$(3): Let $(\tau,\eta)$ be an ARS duality.
Since $\tau$ gives an equivalence $\underline{\mathscr{C}}\simeq\overline{\mathscr{C}}$, both $\mod\underline{\mathscr{C}}$ and $\mod\underline{\mathscr{C}}^\mathrm{op}$ are closed under kernels by Proposition~\ref{weak kernels}.

For $F\in\mod\underline{\mathscr{C}}$, take an exact sequence \eqref{3 terms 1}. Applying $\Dd$, we have the following commutative diagram of exact sequences.
\[\xymatrix{
0\ar[r]&\Dd F\ar[r]&\Dd \underline{\mathscr{C}}(-,A)\ar[r]\ar[d]^\wr&\Dd \underline{\mathscr{C}}(-,B)\ar[d]^\wr\\
&&\mathbb{E}(A,\tau-)\ar[r]&\mathbb{E}(B,\tau-).
}\]
Since the $\overline{\mathscr{C}}^\mathrm{op}$-modules $\mathbb{E}(A,-)$ and $\mathbb{E}(B,-)$ are finitely presented by Theorem~\ref{fundamental theorem}, the $\underline{\mathscr{C}}^\mathrm{op}$-modules $\mathbb{E}(A,\tau-)$ and $\mathbb{E}(B,\tau-)$ are finitely presented.
Since $\mod\underline{\mathscr{C}}^\mathrm{op}$ is closed under kernels as remarked above, the $\underline{\mathscr{C}}^\mathrm{op}$-module $\Dd F$ is finitely presented.

For any $G\in\mod\underline{\mathscr{C}}^\mathrm{op}$, one can show $\Dd G\in\mod\underline{\mathscr{C}}$ by a dual argument. Thus $\underline{\mathscr{C}}$ is a dualizing $k$-variety.

(3)$\Rightarrow$(1): Since $\underline{\mathscr{C}}$ is idempotent complete, we have an equivalence
\begin{equation}\label{eq 1}
\underline{\mathscr{C}}\simeq(\proj\underline{\mathscr{C}}^\mathrm{op})^\mathrm{op}\ \mbox{ given by }\ A\mapsto\underline{\mathscr{C}}(A,-).
\end{equation}
Since $\underline{\mathscr{C}}$ is a dualizing $k$-variety, we have an equivalence
\begin{equation}\label{eq 2}
(\proj\underline{\mathscr{C}}^\mathrm{op})^\mathrm{op}\simeq\inj\underline{\mathscr{C}}\ \mbox{ given by }\ F\mapsto \Dd F.
\end{equation}
Since $\overline{\mathscr{C}}$ is idempotent complete, by Theorem~\ref{fundamental theorem}, we have an equivalence
\begin{equation}\label{eq 3}
\overline{\mathscr{C}}\simeq\inj\underline{\mathscr{C}}\ \mbox{ given by }\ A\mapsto\mathbb{E}(-,A).
\end{equation}
Composing \eqref{eq 1}, \eqref{eq 2} and a quasi-inverse of \eqref{eq 3}, we define an equivalence $\tau\colon\underline{\mathscr{C}}\simeq\overline{\mathscr{C}}$.
By construction, we have an isomorphism
\[\Dd\underline{\mathscr{C}}(A,-)\simeq\mathbb{E}(-,\tau A)\]
of functors for any $A\in\mathscr{C}$. Thus we have an ARS duality.

(1)$\Leftrightarrow$(2): The proof is dual to (1)$\Leftrightarrow$(3).
\end{proof}

\section{Induced almost split extensions}\label{section_ARsubcat}

In this section, we discuss the following three general methods to construct a new extriangulated category from a given one $(\mathscr{C},\mathbb{E},\mathfrak{s})$, where the first and the third were initiated in fundamental works \cite{ASo} and \cite{AS} respectively. 
\begin{enumerate}[$\bullet$]
\item Replace $\mathbb{E}$ by some additive subfunctor (Proposition \ref{PropClosed}).
\item Replace $\mathscr{C}$ by its ideal quotient by projective-injective objects (Corollary \ref{CorIQuot}).
\item Replace $\mathscr{C}$ by some extension-closed subcategory (Definition \ref{RemETrExtClosed}).
\end{enumerate}
We will show that the existence of almost split extensions in $(\mathscr{C},\mathbb{E},\mathfrak{s})$ is inherited by these new extriangulated categories, see Proposition~\ref{PropoRelARExt} for the first, Proposition~\ref{PropQuotARExt} for the second, and Theorem~\ref{CorSubARExt} for the third.


\subsection{Induced extriangulated structures}

Let us recall the definitions and results from \cite{HLN1}. We do not assume that $\mathscr{C}$ is  $k$-linear, nor Krull--Schmidt until Remark~\ref{RemIfInAddition}. Additive subfunctors $\mathbb{F}\subseteq\mathbb{E}$ for extriangulated categories $(\mathscr{C},\mathbb{E},\mathfrak{s})$ have been considered in \cite{ZH}.
\begin{definition}
Let $\mathbb{F}\subseteq\mathbb{E}$ be an additive subfunctor. Define $\mathfrak{s}|_{\mathbb{F}}$ to be the restriction of $\mathfrak{s}$ to $\mathbb{F}$. Namely, it is defined by $\mathfrak{s}|_{\mathbb{F}}(\delta)=\mathfrak{s}(\delta)$ for any $\mathbb{F}$-extension $\delta$.
\end{definition}

\begin{claim}\label{ClaimPreExTr}(cf.\ \cite[Claim~3.9 for $n=1$]{HLN1})
For any additive subfunctor $\mathbb{F}\subseteq\mathbb{E}$, the triplet $(\mathscr{C},\mathbb{F},\mathfrak{s}|_{\mathbb{F}})$ satisfies {\rm (ET1),(ET2),(ET3),(ET3)$^\mathrm{op}$}.
\end{claim}
\begin{proof}
This immediately follows from the definitions of these conditions.
\end{proof}

Thus we may speak of $\mathfrak{s}|_{\mathbb{F}}$-conflations (resp.\ $\mathfrak{s}|_{\mathbb{F}}$-inflations, $\mathfrak{s}|_{\mathbb{F}}$-deflations) and $\mathfrak{s}|_{\mathbb{F}}$-triangles. The following condition on $\mathbb{F}$ gives a necessary and sufficient condition for $(\mathscr{C},\mathbb{F},\mathfrak{s}|_{\mathbb{F}})$ to be an extriangulated category (Proposition~\ref{PropClosed}).

\begin{definition}[cf.\ \cite{DRSSK,ASo}]\label{DefClosed}
Let $\mathbb{F}\subseteq\mathbb{E}$ be a additive subfunctor.
\begin{enumerate}[\rm(1)]
\item $\mathbb{F}\subseteq\mathbb{E}$ is {\it closed on the right} if $\mathbb{F}(-,A)\overset{x\circ -}{\longrightarrow}\mathbb{F}(-,B)\overset{y\circ -}{\longrightarrow}\mathbb{F}(-,C)$ is exact for any $\mathfrak{s}|_{\mathbb{F}}$-conflation $A\overset{x}{\longrightarrow}B\overset{y}{\longrightarrow}C$.
\item $\mathbb{F}\subseteq\mathbb{E}$ is {\it closed on the left} if $\mathbb{F}(C,-)\overset{-\circ y}{\longrightarrow}\mathbb{F}(B,-)\overset{-\circ x}{\longrightarrow}\mathbb{F}(A,-)$ is exact for any $\mathfrak{s}|_{\mathbb{F}}$-conflation $A\overset{x}{\longrightarrow}B\overset{y}{\longrightarrow}C$.
\end{enumerate}
\end{definition}

The following has been shown in \cite{HLN1} more generally for any {\it $n$-exangulated category}, which recovers the notion of extriangulated category when $n=1$.
\begin{lemma}[{\cite[Lemma~3.15 for $n=1$]{HLN1}}]\label{LemClosed}
For any additive subfunctor $\mathbb{F}\subseteq\mathbb{E}$, the following are equivalent.
\begin{enumerate}[\rm(1)]
\item $\mathbb{F}$ is closed on the right.
\item $\mathbb{F}$ is closed on the left.
\end{enumerate}
Thus in the following argument, we simply say $\mathbb{F}\subseteq\mathbb{E}$ is {\it closed}, if one of them is satisfied.
\end{lemma}

\begin{proposition}[{\cite[Proposition~3.16 for $n=1$]{HLN1}}]\label{PropClosed}
For any additive subfunctor $\mathbb{F}\subseteq\mathbb{E}$, the following are equivalent.
\begin{enumerate}[\rm(1)]
\item $\mathfrak{s}|_{\mathbb{F}}$-inflations are closed under composition.
\item $\mathfrak{s}|_{\mathbb{F}}$-deflations are closed under composition.
\item $\mathbb{F}\subseteq\mathbb{E}$ is closed.
\item $(\mathscr{C},\mathbb{F},\mathfrak{s}|_{\mathbb{F}})$ satisfies {\rm (ET4)}.
\item $(\mathscr{C},\mathbb{F},\mathfrak{s}|_{\mathbb{F}})$ satisfies {\rm (ET4)$^\mathrm{op}$}.
\item $(\mathscr{C},\mathbb{F},\mathfrak{s}|_{\mathbb{F}})$ is extriangulated.
\end{enumerate}
\end{proposition}

Recall that, for an additive category $\mathscr{C}$ and a full subcategory $\mathscr{D}$, we call a morphism $f:D\to C$ a \emph{right $\mathscr{D}$-approximation} of $C\in\mathscr{C}$ if $D\in\mathscr{D}$ and $f\circ-:\mathscr{C}(D',D)\to\mathscr{C}(D',C)$ is surjective for each $D'\in\mathscr{D}$. We call $\mathscr{D}$ \emph{contravariantly finite} in $\mathscr{C}$ if each object in $\mathscr{C}$ has a right $\mathscr{D}$-approximation. Dually we define a \emph{left $\mathscr{D}$-approximation} and a \emph{covariantly finite subcategory}.
We call $\mathscr{D}$ \emph{functorially finite} if it is contravariantly and covariantly finite.

A basic construction of closed additive subfunctors is the following. 

\begin{definition-proposition}[{\cite[Definition~3.18, Proposition~3.19 for $n=1$]{HLN1}}]\label{DefItoF}
Let $\mathscr{D}\subseteq\mathscr{C}$ be a full subcategory. Define subfunctors $\mathbb{E}_{\mathscr{D}}$ and $\mathbb{E}^{\mathscr{D}}$ of $\mathbb{E}$ by
\begin{eqnarray*}
\mathbb{E}_{\mathscr{D}}(C,A)&=&\{\delta\in\mathbb{E}(C,A)\mid \mathscr{C}(D,C)\overset{\delta\circ-}{\longrightarrow}\mathbb{E}(D,A)\ \text{is zero for any}\ D\in\mathscr{D}\},\\
\mathbb{E}^{\mathscr{D}}(C,A)&=&\{\delta\in\mathbb{E}(C,A)\mid \mathscr{C}(A,D)\overset{-\circ\delta}{\longrightarrow}\mathbb{E}(C,D)\ \text{is zero for any}\ D\in\mathscr{D}\}.
\end{eqnarray*}
Then these are closed additive subfunctors of $\mathbb{E}$.
\end{definition-proposition}

\begin{remark}\label{RemItoF}
Let $\mathscr{D}\subseteq\mathscr{C}$ be as above. An $\mathfrak{s}$-conflation
\begin{equation}\label{sItriangle}
A\overset{x}{\longrightarrow}D\overset{y}{\longrightarrow}C\quad(D\in\mathscr{D})
\end{equation}
is an $\mathfrak{s}|_{\mathbb{E}^{\mathscr{D}}}$-conflation if and only if $x$ is a left $\mathscr{D}$-approximation. Dually, $(\ref{sItriangle})$ is an $\mathfrak{s}|_{\mathbb{E}_{\mathscr{D}}}$-conflation if and only if $y$ is a right $\mathscr{D}$-approximation.
\end{remark}

The following has been shown in \cite[Proposition 3.30]{NP}.
\begin{fact}\label{FactEbar}
Let $\mathscr{D}\subseteq\mathscr{C}$ be an additive full subcategory.
If $\mathscr{D}$ satisfies $\mathscr{D}\subseteq\mathrm{Proj}_{\mathbb{E}}\mathscr{C}\cap\mathrm{Inj}_{\mathbb{E}}\mathscr{C}$, then the extriangulated structure of $\mathscr{C}$ induces an extriangulated structure 
$(\mathscr{C}/\mathscr{D},\mathbb{E}/\mathscr{D},\mathfrak{s}/\mathscr{D})$ on the ideal quotient $\mathscr{C}/\mathscr{D}$, where the functor $\mathbb{E}/\mathscr{D}\colon(\mathscr{C}/\mathscr{D})^\mathrm{op}\times(\mathscr{C}/\mathscr{D})\to\mathit{Ab}$ makes the following diagram commutative,
\[\xymatrix{
(\mathscr{C}/\mathscr{D})^\mathrm{op}\times(\mathscr{C}/\mathscr{D})\ar[r]^(.7){\mathbb{E}/\mathscr{D}}&\mathit{Ab}\ar@{=}[d]\\
\mathscr{C}^\mathrm{op}\times\mathscr{C}\ar[u]\ar[r]^{\mathbb{E}}&\mathit{Ab}}\]
and the realization $(\mathfrak{s}/\mathscr{D})(\delta)$ of $\delta\in(\mathbb{E}/\mathscr{D})(C,A)=\mathbb{E}(C,A)$ is the image of $\mathfrak{s}(\delta)$ in $\mathscr{C}/\mathscr{D}$.
\end{fact}

Definition-Proposition~\ref{DefItoF} and Fact~\ref{FactEbar} give the following observation.

\begin{corollary}\label{CorIQuot}
Let $\mathscr{D}\subseteq\mathscr{C}$ be an additive full subcategory.
If we put $\mathbb{F}=\mathbb{E}_{\mathscr{D}}\cap\mathbb{E}^{\mathscr{D}}$, then $\mathbb{F}\subseteq\mathbb{E}$ is closed and thus $(\mathscr{C},\mathbb{F},\mathfrak{s}|_{\mathbb{F}})$ is an extriangulated category by Proposition~\ref{PropClosed}.
Moreover, since $\mathscr{D}\subseteq\mathrm{Proj}_{\mathbb{F}}\mathscr{C}\cap\mathrm{Inj}_{\mathbb{F}}\mathscr{C}$ holds, we obtain an extriangulated category $(\mathscr{C}/\mathscr{D},\mathbb{F}/\mathscr{D},(\mathfrak{s}|_{\mathbb{F}})/\mathscr{D})$.
\end{corollary}

\subsection{Induced almost split sequences}

First let us consider almost split extensions in relative extriangulated categories.
\begin{proposition}\label{PropoRelARExt}
Let $\mathbb{F}\subseteq\mathbb{E}$ be any closed subfunctor. If $(\mathscr{C},\mathbb{E},\mathfrak{s})$ has $($resp.\ left, or right$)$ almost split extensions, then so does $(\mathscr{C},\mathbb{F},\mathfrak{s}|_{\mathbb{F}})$.
\end{proposition}
\begin{proof}
Suppose that $(\mathscr{C},\mathbb{E},\mathfrak{s})$ has right almost split extensions. If $C\in\mathscr{C}$ is non-projective in $(\mathscr{C},\mathbb{F},\mathfrak{s}|_{\mathbb{F}})$, then there is some $X\in\mathscr{C}$ and some non-zero $\theta\in\mathbb{F}(C,X)$. Since it is non-projective in $(\mathscr{C},\mathbb{E},\mathfrak{s})$ by Remark~\ref{RemProjInj}, there is an almost split $\mathbb{E}$-extension $\delta\in\mathbb{E}(C,A)$. Then by Lemma~\ref{LemARSeq3}, there is some $a\in\mathscr{C}(X,A)$ satisfying $a\theta=\delta$. This implies $\delta\in\mathbb{F}(C,A)$.
By Remark~\ref{RemRelAR2} $\delta$ is an almost split $\mathbb{F}$-extension. This shows that $(\mathscr{C},\mathbb{F},\mathfrak{s}|_{\mathbb{F}})$ has right almost split extensions.
Similarly for left almost split extensions.
\end{proof}

Secondly, let us consider almost split extensions in ideal quotients.

\begin{proposition}\label{PropQuotARExt}
Let $\mathscr{D}\subseteq\mathscr{C}$ be an additive full subcategory
satisfying $\mathscr{D}\subseteq\mathrm{Proj}_{\mathbb{E}}\mathscr{C}\cap\mathrm{Inj}_{\mathbb{E}}\mathscr{C}$.
Let $\delta\in\mathbb{E}(C,A)$ be any almost split extension. 
Then $\delta$ gives an almost split extension $\delta\in\widetilde{\mathbb{E}}(C,A)$ in $(\widetilde{\mathscr{C}},\widetilde{\mathbb{E}},\widetilde{\mathfrak{s}})=(\mathscr{C}/\mathscr{D},\mathbb{E}/\mathscr{D},\mathfrak{s}/\mathscr{D})$.
\end{proposition}

\begin{proof}
We remark that $(\widetilde{\mathscr{C}},\widetilde{\mathbb{E}},\widetilde{\mathfrak{s}})$ is an extriangulated category by Fact~\ref{FactEbar}. 
Since $A$ is endo-local and non-injective, it satisfies $[\mathscr{D}](A,A)\subseteq\mathrm{rad}\mathrm{End}_{\mathscr{C}}(A)$. Similarly for $[\mathscr{D}](C,C)\subseteq\mathrm{rad}\mathrm{End}_{\mathscr{C}}(C)$. Thus we have
\begin{eqnarray}
&\mathrm{rad}\mathrm{End}_{\mathscr{C}/\mathscr{D}}(A)=(\mathrm{rad}\mathrm{End}_{\mathscr{C}}(A))/[\mathscr{D}](A,A),&\label{radQuot}\\
&\mathrm{rad}\mathrm{End}_{\mathscr{C}/\mathscr{D}}(C)=(\mathrm{rad}\mathrm{End}_{\mathscr{C}}(C))/[\mathscr{D}](C,C).&\nonumber
\end{eqnarray}
It suffices to show that $\delta$ satisfies conditions (a),(b) of Proposition~\ref{LemEquivARExt} in $(\widetilde{\mathscr{C}},\widetilde{\mathbb{E}},\widetilde{\mathfrak{s}})$. By $(\ref{radQuot})$, condition {\rm (a)} for $\delta\in\widetilde{\mathbb{E}}(C,A)$ also holds in $(\widetilde{\mathscr{C}},\widetilde{\mathbb{E}},\widetilde{\mathfrak{s}})$. Condition {\rm (b)} in $(\widetilde{\mathscr{C}},\widetilde{\mathbb{E}},\widetilde{\mathfrak{s}})$ follows immediately from that in $(\mathscr{C},\mathbb{E},\mathfrak{s})$.
\end{proof}

Thirdly, let us consider almost split extensions in extension-closed subcategories. 

\begin{definition}\label{RemETrExtClosed}
Let $(\mathscr{C},\mathbb{E},\mathfrak{s})$ be an extriangulated category, and let $\mathscr{D}\subseteq\mathscr{C}$ be an extension-closed full subcategory. If we define $\mathbb{E}^{\prime}$ to be the restriction of $\mathbb{E}$ onto $\mathscr{D}^{\mathrm{op}}\times\mathscr{D}$, and define $\mathfrak{s}^{\prime}$ by restricting $\mathfrak{s}$, then $(\mathscr{D},\mathbb{E}^{\prime},\mathfrak{s}^{\prime})$ becomes an extriangulated category. By definition, it satisfies $\mathbb{E}^{\prime}(D_1,D_2)=\mathbb{E}(D_1,D_2)$ for any $D_1,D_2\in\mathscr{D}$.
\end{definition}


\begin{remark}\label{RemEasy}
It is obvious from the definition that if $\delta\in\mathbb{E}(C,A)$ is an almost split extension in $(\mathscr{C},\mathbb{E},\mathfrak{s})$ with $A,C\in\mathscr{D}$, then it is an almost split extension in $(\mathscr{D},\mathbb{E}^{\prime},\mathfrak{s}^{\prime})$.
\end{remark}

We will prove the following result, which generalizes \cite[Theorem 2.4(b)]{AS}.

\begin{theorem}\label{CorSubARExt}
Let $(\mathscr{C},\mathbb{E},\mathfrak{s})$ be a Krull--Schmidt extriangulated category, and let $\mathscr{D}\subseteq\mathscr{C}$ be an extension-closed full subcategory which is closed by direct summands and
contravariantly (resp.\ covariantly) finite in $\mathscr{C}$. If $\mathscr{C}$ has right (resp.\ left) almost split extensions, then so does $\mathscr{D}$.
\end{theorem}

We need the following variation of the famous Wakamatsu's Lemma \cite[Lemma 1.3]{AR5}.

\begin{lemma}[Wakamatsu's Lemma]\label{LemSubARExt}
Let $(\mathscr{C},\mathbb{E},\mathfrak{s})$ be an extriangulated category.
Let $\mathscr{D}\subseteq\mathscr{C}$ be an extension-closed full subcategory.
Let $A\in\mathscr{C}$ be any object.
If $d\in\mathscr{C}(D^A,A)$ is a minimal right $\mathscr{D}$-approximation with $D^A\in\mathscr{D}$, then
\[ d\circ-\colon\mathbb{E}(D,D^A)\to\mathbb{E}(D,A) \]
is monomorphic for any $D\in\mathscr{D}$.
\end{lemma}
\begin{proof}
Let $\theta\in\mathbb{E}(D,D^A)$ be any element, with $\mathfrak{s}(\theta)=[D^A\overset{x}{\longrightarrow}X\overset{y}{\longrightarrow}D]$. By the extension-closedness of $\mathscr{D}\subseteq\mathscr{C}$, we have $X\in\mathscr{D}$. Suppose $d\theta=0$. Then there exists $f\in\mathscr{C}(X,A)$ satisfying $fx=d$. Since $d$ is a right $\mathscr{D}$-approximation, we obtain $g\in\mathscr{D}(X,D^A)$ which makes
\[
\xy
(-12,0)*+{D^A}="0";
(4,0)*+{}="1";
(0,8)*+{X}="2";
(12,0)*+{D^A}="4";
(-4,0)*+{}="5";
(0,-8)*+{A}="6";
{\ar^{x} "0";"2"};
{\ar^{g} "2";"4"};
{\ar_{d} "0";"6"};
{\ar^{f} "2";"6"};
{\ar^{d} "4";"6"};
{\ar@{}|{} "0";"1"};
{\ar@{}|{} "4";"5"};
\endxy
\]
commutative. By the minimality of $d$, it follows that $gx$ is an automorphism. In particular $x$ is a section, which means $\theta=0$.
\end{proof}

\begin{remark}\label{RemIfInAddition}
If in addition $d$ is a deflation, the above lemma is nothing but \cite[Lemma 2.3]{CZZ}.
\end{remark}

The following is an extriangulated analogue of \cite[Theorem 2.3]{K} and \cite[Theorem 3.1]{J}.

\begin{proposition}\label{PropSubARExt}
Let $(\mathscr{C},\mathbb{E},\mathfrak{s})$ be a Krull--Schmidt extriangulated category and let $\mathscr{D}\subseteq\mathscr{C}$ be an extension-closed full subcategory which is closed by direct summands. Suppose that $C\in\mathscr{D}$ is non-projective with respect to the restriction $\mathbb{E}^{\prime}$ of $\mathbb{E}$ to $\mathscr{D}$. If there is an almost split extension $\delta\in\mathbb{E}(C,A)$ in $\mathscr{C}$ and if there is a minimal right $\mathscr{D}$-approximation $D^A\overset{d}{\longrightarrow}A$ with $D^A\in\mathscr{D}$, then there is a direct summand $D_0$ of $D^A$ and an almost split extension $\mu\in\mathbb{E}^{\prime}(C,D_0)$ in $\mathscr{D}$.
\end{proposition}

\begin{proof}
Since $C$ is non-projective with respect to $\mathbb{E}^{\prime}$, there exists $Y\in\mathscr{D}$ and a non-split $\theta\in\mathbb{E}^{\prime}(C,Y)$. By Lemma~\ref{LemARSeq3}{\rm (c)}, there is $a\in\mathscr{C}(Y,A)$ such that $a\theta=\delta$. Since $d$ is a right $\mathscr{D}$-approximation, there is $e\in\mathscr{C}(Y,D^A)$ which satisfies $de=a$. 
Since $d(e\theta)=\delta\ne0$, we have $e\theta\ne0$. 
By Lemma~\ref{Rem2.5}, it suffices to show that $e\theta$ satisfies {\rm (AS2)}.
Let $h\in\mathscr{D}(D,C)$ be any non-retraction. Since $\delta$ is almost split, the element $(e\theta)h\in\mathbb{E}^{\prime}(D,D^A)$ satisfies
\[ d(e\theta h)=\delta h=0. \]
By Lemma~\ref{LemSubARExt}, it follows that $e\theta h=0$.
\end{proof}

We are ready to prove Theorem \ref{CorSubARExt}.

\begin{proof}[Proof of Theorem \ref{CorSubARExt}]
If $\mathscr{D}$ is contravariantly finite in $\mathscr{C}$, then any object in $\mathscr{C}$ admits a minimal right $\mathscr{D}$-approximation, as in \cite[Corollary~1.4]{KS}. Thus Proposition~\ref{PropSubARExt} shows that $\mathscr{D}$ has right almost split extensions. Dually for the case where $\mathscr{D}$ is covariantly finite in $\mathscr{C}$.
\end{proof}

\subsection{Extension-closed subcategories of derived categories}\label{section:example from hereditary}

Throughout this subsection, let $A$ be a finite dimensional algebra over a field $k$. The following follows from our previous results.

\begin{proposition}
Let $\mathscr{C}$ be an extension-closed full subcategory of $\kbproj$.
If $\mathscr{C}$ is functorially finite, then $\mathscr{C}$ has almost split extensions.
\end{proposition}

\begin{proof}
It is well-known (e.g.\ \cite[Section 10.4]{KeDG}\cite[Theorem 3.7]{IR1}) that the Nakayama functor $\nu\colon\kbproj\simeq\operatorname{K^b}(\inj A)$ gives a relative Auslander--Reiten--Serre duality 
\[\mathrm{Hom}_{\operatorname{D^b}(\mod A)}(X,Y)\simeq\Dd\mathrm{Hom}_{\operatorname{D^b}(\mod A)}(Y,\nu X)\]
for $X\in\kbproj$ and $Y\in\operatorname{D^b}(\mod A)$.

Let $X\in\kbproj$ be an indecomposable object which is non-projective in $\mathscr{C}$.
Then there is an almost split extension $\delta\in\mathrm{Ext}^1_{\operatorname{D^b}(\mod A)}(X,\nu X[-1])$ in $\operatorname{D^b}(\mod A)$ by Proposition~\ref{local theorem}.
Applying Proposition~\ref{PropSubARExt} to $(\mathscr{D},\mathscr{C}):=(\mathscr{C},\operatorname{D^b}(\mod A))$, $X$ has an almost split extension in $\mathscr{C}$. Thus $\mathscr{C}$ has right almost split extensions.

On the other hand, since $\nu(\mathscr{C})$ is a functorially finite extension-closed subcategory of $\operatorname{K^b}(\inj A)$, the dual argument shows that $\mathscr{C}$ has left almost split extensions.
Thus the assertion holds.
\end{proof}

\begin{example}\label{example: n-term}
Let $A$ be a finite dimensional $k$-algebra, and $n$ a non-negative integer.
An object $P=(P^i,d^i)$ is called \emph{$n$-term} if $P^i=0$ holds for all $i>0$ and all $i\le -n$.
Then the full subcategory $\mathscr{C}$ of $\kbproj$ consisting of all $n$-term complexes is extension-closed and functorially finite in $\kbproj$.
Therefore $\mathscr{C}$ is an extriangulated category with almost split extensions.
\end{example}

\begin{proof}
Clearly $\mathscr{C}$ is extension-closed.
Since $\mathscr{C}$ can be written as $(\proj A)*(\proj A)[1]*\cdots*(\proj A)[n-1]$ and each $(\proj A)[i]$ is functorially finite in $\kbproj$,
so is $\mathscr{C}$ (e.g.\ \cite[Theorem 1.3]{C}).
\end{proof}


\renewcommand{\arraystretch}{0.6}
In the rest, we give examples to explain results in this section.
Let $k$ be a field, and let $\operatorname{D^b}(kA_3)$ be the bounded derived category of the path algebra of the quiver
\[ A_3:\qquad 1\longleftarrow2\longleftarrow3. \]
We know that $\mathrm{AR}_{\rm ET}(\operatorname{D^b}(kA_3))$ is as follows (\cite[5.6]{Ha1}).
\[\begin{tikzpicture}[scale=0.5, fl/.style={->,>=latex}]
\foreach \x in {-1,0,1,2} {
   \draw[fl] (3.6*\x,1.5) -- (3.6*\x+1,0.5) ;
   \draw[fl] (3.6*\x+1.8,-0.5) -- (3.6*\x+2.8,-1.5) ;
   \draw[fl] (3.6*\x,-1.5) -- (3.6*\x+1,-0.5) ;
   \draw[fl] (3.6*\x+1.8,0.5) -- (3.6*\x+2.8,1.5) ;
   \draw[fl, dashed] (3.6*\x+2.4,2) -- (3.6*\x+0.4,2) ;
   \draw[fl, dashed] (3.6*\x+0.6,0) -- (3.6*\x-1.4,0) ;
   \draw[fl, dashed] (3.6*\x+2.4,-2) -- (3.6*\x+0.4,-2) ;
};
   \draw[fl, dashed] (3.6*3+0.6,0) -- (3.6*3-1.4,0) ;
\draw (-5.8,2) node {$\cdots$} ;
\draw (-4,2) node {$3[-1]$} ;
\draw (-0.4,2) node {$\bsm3\\2\\1\esm$} ;
\draw (3.2,2) node {$1[1]$} ;
\draw (6.8,2) node {$2[1]$} ;
\draw (10.4,2) node {$3[1]$} ;
\draw (12.2,2) node {$\cdots$} ;
\draw (-5.8,0) node {$\cdots$} ;
\draw (-2.2,0) node {$\bsm2\\1\esm$} ;
\draw (1.4,0) node {$\bsm3\\2\esm$} ;
\draw (5,0) node {$\bsm2\\1\esm$[1]} ;
\draw (8.6,0) node {$\bsm3\\2\esm$[1]} ;
\draw (12.2,0) node {$\cdots$} ;
\draw (-5.8,-2) node {$\cdots$} ;
\draw (-4,-2) node {$1$} ;
\draw (-0.4,-2) node {$2$} ;
\draw (3.2,-2) node {$3$} ;
\draw (6.8,-2) node {$\bsm3\\2\\1\esm[1]$} ;
\draw (10.4,-2) node {$1[2]$} ;
\draw (12.2,-2) node {$\cdots$} ;
\end{tikzpicture}
\]

Let $(t^{\le0},t^{\ge0})$ be the standard $t$-structure on $\operatorname{D^b}(kA_3)$, and put $\mathscr{C}=t^{\le0}\cap t^{\ge-1}$. Since $\mathscr{C}\subseteq\operatorname{D^b}(kA_3)$ is extension-closed, it has an induced extriangulated structure, which we denote by $(\mathscr{C},\mathbb{E},\mathfrak{s})$. The indecomposables which belong to $\mathscr{C}$ are
\begin{itemize}
\item[{\rm (i)}] $\mathbb{E}$-projectives: $1,\bsm2\\1\esm,\bsm3\\2\\1\esm$,
\item[{\rm (ii)}] $\mathbb{E}$-injectives: $\bsm3\\2\\1\esm [1],\bsm3\\2\esm [1],3[1]$,
\item[{\rm (ii)}] the rest: $2,\bsm3\\2\esm,1[1],3,\bsm2\\1\esm [1],2[1]$.\end{itemize}
As in Remark~\ref{RemEasy}, we see that almost split sequences in $\operatorname{D^b}(kA_3)$ starting from objects in {\rm (i), (ii)} are also split sequences in $\mathscr{C}$. Similarly for almost split sequences ending at objects in {\rm (ii),(iii)}.
Thus $\mathrm{AR}_{\rm ET}(\mathscr{C},\mathbb{E},\mathfrak{s})$ is as follows. 
The $\mathbb{E}$-projective (resp.\ $\mathbb{E}$-injective) objects are highlighted by a vertical line to their left (resp.\ to their right).
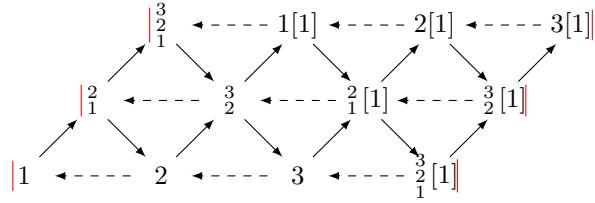
\[
\begin{tikzpicture}[scale=0.5, fl/.style={->,>=latex}]
\foreach \x in {-1,0,1,2} {
   \draw[fl] (3.6*\x+1.8,0.5) -- (3.6*\x+2.8,1.5) ;
   \draw[fl] (3.6*\x,-1.5) -- (3.6*\x+1,-0.5) ;
};
\foreach \x in {0,1,2} {
   \draw[fl] (3.6*\x,1.5) -- (3.6*\x+1,0.5) ;
   \draw[fl] (3.6*\x-1.8,-0.5) -- (3.6*\x-0.8,-1.5) ;
   \draw[fl, dashed] (3.6*\x+2.4,2) -- (3.6*\x+0.4,2) ;
   \draw[fl, dashed] (3.6*\x+0.6,0) -- (3.6*\x-1.4,0) ;
   \draw[fl, dashed] (3.6*\x-1.2,-2) -- (3.6*\x-3.2,-2) ;
};
\draw (-0.4,2) node {$\bsm3\\2\\1\esm$} ;
\draw (3.2,2) node {$1[1]$} ;
\draw (6.8,2) node {$2[1]$} ;
\draw (10.4,2) node {$3[1]$} ;
\draw (-2.2,0) node {$\bsm2\\1\esm$} ;
\draw (1.4,0) node {$\bsm3\\2\esm$} ;
\draw (5,0) node {$\bsm2\\1\esm$[1]} ;
\draw (8.6,0) node {$\bsm3\\2\esm$[1]} ;
\draw (-4,-2) node {$1$} ;
\draw (-0.4,-2) node {$2$} ;
\draw (3.2,-2) node {$3$} ;
\draw (6.8,-2) node {$\bsm3\\2\\1\esm[1]$} ;
%
\draw[red](-0.7,2.46) -- (-0.7,1.54) ;
\draw[red](-2.5,0.42) -- (-2.5,-0.42) ;
\draw[red](-4.3,-1.6) -- (-4.3,-2.4) ;
\draw[red](10.95,2.4) -- (10.95,1.6) ;
\draw[red](9.2,0.42) -- (9.2,-0.42) ;
\draw[red](7.4,-1.56) -- (7.4,-2.44) ;
\end{tikzpicture}
\]

Put $\mathscr{D}=\mathrm{add}\bsm3\\2\\1\esm [1]$ and let $(\mathbb{F},\mathfrak{t})=(\mathbb{E}_{\mathscr{D}},\mathfrak{s}|_{\mathbb{E}_{\mathscr{D}}})$ be the induced relative extriangulated structure in $(\mathscr{C},\mathbb{E},\mathfrak{s})$. Then together with the objects in {\rm (i)}, the object $\bsm3\\2\\1\esm [1]$ becomes $\mathbb{F}$-projective. We have $\mathbb{F}(X,3)=0$ for any indecomposable $X\in\mathscr{C}$, and thus $3$ becomes $\mathbb{F}$-injective, together with objects in {\rm (ii)}.
We see that the almost split extension $\delta\in\mathbb{E}(\bsm3\\2\esm[1],\bsm2\\1\esm[1])$ belongs to $\mathbb{F}(\bsm3\\2\esm[1],\bsm2\\1\esm[1])$, and thus the almost split sequence in $(\mathscr{C},\mathbb{E},\mathfrak{s})$ ending at $\bsm3\\2\esm[1]$ is also an almost split sequence in $(\mathscr{C},\mathbb{F},\mathfrak{t})$. Similarly for the other almost split extensions in $\mathbb{E}(C,A)$ with $A,C\notin\mathscr{D}$. Thus $\mathrm{AR}_{\rm ET}(\mathscr{C},\mathbb{F},\mathfrak{t})$ is as follows.
\[
\begin{tikzpicture}[scale=0.5, fl/.style={->,>=latex}]
\foreach \x in {-1,0,1,2} {
   \draw[fl] (3.6*\x+1.8,0.5) -- (3.6*\x+2.8,1.5) ;
   \draw[fl] (3.6*\x,-1.5) -- (3.6*\x+1,-0.5) ;
};
\foreach \x in {0,1,2} {
   \draw[fl] (3.6*\x,1.5) -- (3.6*\x+1,0.5) ;
   \draw[fl] (3.6*\x-1.8,-0.5) -- (3.6*\x-0.8,-1.5) ;
   \draw[fl, dashed] (3.6*\x+2.4,2) -- (3.6*\x+0.4,2) ;
   \draw[fl, dashed] (3.6*\x+0.6,0) -- (3.6*\x-1.4,0) ;
};
\foreach \x in {0,1} {
   \draw[fl, dashed] (3.6*\x-1.2,-2) -- (3.6*\x-3.2,-2) ;
};
\draw (-0.4,2) node {$\bsm3\\2\\1\esm$} ;
\draw (3.2,2) node {$1[1]$} ;
\draw (6.8,2) node {$2[1]$} ;
\draw (10.4,2) node {$3[1]$} ;
\draw (-2.2,0) node {$\bsm2\\1\esm$} ;
\draw (1.4,0) node {$\bsm3\\2\esm$} ;
\draw (5,0) node {$\bsm2\\1\esm$[1]} ;
\draw (8.6,0) node {$\bsm3\\2\esm$[1]} ;
\draw (-4,-2) node {$1$} ;
\draw (-0.4,-2) node {$2$} ;
\draw (3.2,-2) node {$3$} ;
\draw (6.8,-2) node {$\bsm3\\2\\1\esm[1]$} ;
%
\draw[red](-0.7,2.46) -- (-0.7,1.54) ;
\draw[red](-2.5,0.42) -- (-2.5,-0.42) ;
\draw[red](-4.3,-1.6) -- (-4.3,-2.4) ;
\draw[red](10.95,2.4) -- (10.95,1.6) ;
\draw[red](9.2,0.42) -- (9.2,-0.42) ;
\draw[red](7.4,-1.56) -- (7.4,-2.44) ;
\draw[red](6.2,-1.56) -- (6.2,-2.44) ;
\draw[red](3.5,-1.57) -- (3.5,-2.43) ;
\end{tikzpicture}
\]

Since $\bsm3\\2\\1\esm [1]$ is both $\mathbb{F}$-projective and $\mathbb{F}$-injective, we obtain the quotient extriangulated category $(\mathscr{C}/\mathscr{D},\mathbb{F}/\mathscr{D},\mathfrak{t}/\mathscr{D})$.
The ideal $[\mathscr{D}]$ satisfies $[\mathscr{D}](C,C)=0$ for any indecomposable object $C\in\mathscr{C}$ which is not in $\mathscr{D}$. Thus by Proposition~\ref{PropQuotARExt}, each almost split extension in $(\mathscr{C},\mathbb{F},\mathfrak{t})$, which corresponds to a dashed arrow in the above quiver, gives an almost split extension in $(\mathscr{C}/\mathscr{D},\mathbb{F}/\mathscr{D},\mathfrak{t}/\mathscr{D})$. Hence $\mathrm{AR}_{\rm ET}(\mathscr{C}/\mathscr{D},\mathbb{F}/\mathscr{D},\mathfrak{t}/\mathscr{D})$ becomes as follows.
\[
\begin{tikzpicture}[scale=0.5, fl/.style={->,>=latex}]
\foreach \x in {-1,0,1,2} {
   \draw[fl] (3.6*\x+1.8,0.5) -- (3.6*\x+2.8,1.5) ;
};
\foreach \x in {-1,0,1} {
   \draw[fl] (3.6*\x,-1.5) -- (3.6*\x+1,-0.5) ;
};
\foreach \x in {0,1,2} {
   \draw[fl] (3.6*\x,1.5) -- (3.6*\x+1,0.5) ;
   \draw[fl, dashed] (3.6*\x+2.4,2) -- (3.6*\x+0.4,2) ;
   \draw[fl, dashed] (3.6*\x+0.6,0) -- (3.6*\x-1.4,0) ;
};
\foreach \x in {0,1} {
   \draw[fl] (3.6*\x-1.8,-0.5) -- (3.6*\x-0.8,-1.5) ;
   \draw[fl, dashed] (3.6*\x-1.2,-2) -- (3.6*\x-3.2,-2) ;
};
\draw (-0.4,2) node {$\bsm3\\2\\1\esm$} ;
\draw (3.2,2) node {$1[1]$} ;
\draw (6.8,2) node {$2[1]$} ;
\draw (10.4,2) node {$3[1]$} ;
\draw (-2.2,0) node {$\bsm2\\1\esm$} ;
\draw (1.4,0) node {$\bsm3\\2\esm$} ;
\draw (5,0) node {$\bsm2\\1\esm$[1]} ;
\draw (8.6,0) node {$\bsm3\\2\esm$[1]} ;
\draw (-4,-2) node {$1$} ;
\draw (-0.4,-2) node {$2$} ;
\draw (3.2,-2) node {$3$} ;
%
\draw[red](-0.7,2.46) -- (-0.7,1.54) ;
\draw[red](-2.5,0.42) -- (-2.5,-0.42) ;
\draw[red](-4.3,-1.6) -- (-4.3,-2.4) ;
\draw[red](10.95,2.4) -- (10.95,1.6) ;
\draw[red](9.2,0.42) -- (9.2,-0.42) ;
\draw[red](3.5,-1.57) -- (3.5,-2.43) ;
\end{tikzpicture}
\]

We remark that the almost split sequence
\[ \bsm2\\1\esm [1]\overset{x}{\longrightarrow}2[1]\oplus\bsm3\\2\\1\esm [1]\overset{y}{\longrightarrow}\bsm3\\2\esm [1]\overset{\delta}{\dashrightarrow}  \]
in $(\mathscr{C},\mathbb{F},\mathfrak{t})$ induces an almost split sequence
\[ \bsm2\\1\esm [1]\overset{\widetilde{x}}{\longrightarrow}2[1]\overset{\widetilde{y}}{\longrightarrow}\bsm3\\2\esm [1]\overset{\delta}{\dashrightarrow}  \]
in $(\mathscr{C}/\mathscr{D},\mathbb{F}/\mathscr{D},\mathfrak{t}/\mathscr{D})$.
\renewcommand{\arraystretch}{1}

\section{Sink and source sequences in extriangulated categories}\label{section:sink and source}

The following notion is basic to study the additive structure of an extriangulated category.

\begin{definition}\label{define sink sequence}
Let $\mathscr{C}$ be a Krull--Schmidt category. For an indecomposable object $X$ in $\mathscr{C}$, we call a complex 
\begin{equation}\label{sink sequence of X}
A_\ell\xrightarrow{f_\ell}\cdots\xrightarrow{f_3} A_2\xrightarrow{f_2} A_1\xrightarrow{f_1} X
\end{equation}
with $\ell\in\mathbb{Z}_{\ge1}\cup\{\infty\}$ a \emph{sink sequence} if each $f_i$ is right minimal and the following sequence is exact.
\begin{equation}\label{sink sequence of X2}
\mathscr{C}(-,A_\ell)\xrightarrow{f_\ell\circ-}\cdots\xrightarrow{f_3\circ-}\mathscr{C}(-,A_2)\xrightarrow{f_2\circ-}\mathscr{C}(-,A_1)\xrightarrow{f_1\circ-}({\rm rad}\mathscr{C})(-,X)\to0.
\end{equation}
It is called a \emph{sink resolution if either $A_\ell=0$ or $\ell=\infty$ holds. In this case, the sequence \eqref{sink sequence of X2} gives a minimal projective resolution in $\mathrm{Mod}\mathscr{C}$.}
Dually, we define a \emph{source sequence} (\emph{source resolution}) of an indecomposable object.
\end{definition}

Almost split sequences give sink sequences of indecomposable non-projective objects and source sequences of indecomposable non-injective objects at once.

\begin{proposition}\label{ass gives sink sequence}
Let $\mathscr{C}$ be a Krull--Schmidt extriangulated category.
If $A\to B\to C$ is an almost split sequence in $\mathscr{C}$, then this is a sink sequence of $C$ and a source sequence of $A$.
\end{proposition}

\begin{proof}
The assertion is immediate from Fact~\ref{FactExact}(1) and Proposition~\ref{PropARSeq1}.
\end{proof}

\begin{remark}
In Proposition~\ref{ass gives sink sequence}, it is difficult in general to describe sink (resp.\ source) resolutions of $C$ (resp.\ $A$) except for the following special cases.
\begin{enumerate}[\rm(1)]
\item If $\mathscr{C}$ is an exact category, then $0\to A\to B\to C$ is a sink resolution of $C$, and $A\to B\to C\to 0$ is a source resolution of $A$.
\item If $\mathscr{C}$ is a triangulated category with suspension functor $[1]$, then $\cdots\to C[-2]\to A[-1]\to B[-1]\to C[-1]\to A\to B\to C$ is a sink resolution of $C$, and 
$A\to B\to C\to A[1]\to B[1]\to C[1]\to A[2]\to\cdots$ is a source resolution of $A$.
\end{enumerate}
\end{remark}

Now we study sink (resp.\ source) sequences of projective (resp.\ injective) objects.

\begin{theorem}\label{weak kernel of right almost split}
Let $\mathscr{C}$ be a Krull--Schmidt extriangulated category.
\begin{enumerate}[\rm(1)]
\item Assume that $\mathscr{C}$ has enough injectives. If $B\xrightarrow{g} A\xrightarrow{f} P$ is a sink sequence of an indecomposable projective object $P$, then $B$ is injective.
\item Assume that $\mathscr{C}$ has enough projectives. If $I\to A\to B$ is a source sequence of an indecomposable injective object $I$, then $B$ is projective.
\end{enumerate}
\end{theorem}

\begin{proof}
(1) Take an $\mathfrak{s}$-conflation $B\overset{i}{\longrightarrow}I\overset{j}{\longrightarrow}X$ with $I$ injective. By \cite[Proposition 1.20]{LNa}, we have the following commutative diagram such that $B \xrightarrow{\left[\bsm i\\-g \esm\right]}I\oplus A\xrightarrow{[h\ x^{\prime}]} Z$ is an $\mathfrak{s}$-conflation.
\[
\xymatrix@C4em{
B\ar[r]^i\ar[d]^g&I\ar[r]^j\ar[d]^h&X\ar@{=}[d]\\
A\ar[r]^{x^{\prime}}&Z\ar[r]^{y^{\prime}}&X
}
\]
Since $fg=0$, there exists $b\in\mathscr{C}(Z,P)$ such that $b [h\ x^{\prime}]=[0\ f]$.

If $b$ is a retraction, then $[0\ f]=b [h\ x^{\prime}]\in\mathscr{C}(I\oplus A,P)$ is a deflation, and hence a retraction since $P$ is projective. This is a contradiction since $f$ belongs to ${\rm rad}\mathscr{C}$.
Therefore $b$ is not a retraction. Since $f$ is a sink morphism, there is $c\in\mathscr{C}(Z,A)$ satisfying $b = fc$. Since $f = bx^{\prime} = fcx^{\prime}$ holds and $f$ is right minimal, $cx^{\prime}$ is an isomorphism in $\mathscr{C}$. Since $g$ is a weak kernel of $f$ and $f(cx^{\prime})^{-1}ch=fch=bh=0$ holds, there is $d\in\mathscr{C}(I,B)$ such that $(cx^{\prime})^{-1}ch=gd$.
\[
\xymatrix@R1.5em@C4em{
B\ar[r]^i\ar[d]^g&I\ar[rr]^d\ar[d]^h&&B\ar[d]^g\\
A\ar[r]^{x^{\prime}}\ar[d]^f&Z\ar[r]^c\ar[d]^b&A\ar[d]^f\ar[r]^{(cx')^{-1}}&A\ar[d]^f\\
P\ar@{=}[r]&P\ar@{=}[r]&P\ar@{=}[r]&P
}
\]
Since $g=gdi$ holds and $g$ is right minimal, $di$ is an isomorphism in $\mathscr{C}$. Thus $i$ is a section and $B$ is injective. (2) is dual to (1).
\end{proof}

\begin{remark}\label{Ai with big i}
In contrast to Theorem \ref{weak kernel of right almost split},
the terms $A_i$ with $i\ge3$ in the sink sequence \eqref{sink sequence of X} of an indecomposable projective object are not necessarily injective.
For example, consider the extriangulated category $(\mathscr{C},\mathbb{F},\mathfrak{t})$ in Section~\ref{section:example from hereditary}. Then the projective object ${\bsm3\\2\\1\esm}[1]$ has the following sink resolution
\[0\to{\bsm2\\1\esm}\to{\bsm3\\2\\1\esm}\to 3\to{\bsm2\\1\esm}[1]\to {\bsm3\\2\\1\esm}[1],\]
where ${\bsm2\\1\esm}$ and ${\bsm3\\2\\1\esm}$ are non-injective.
\end{remark}

In Theorem~\ref{weak kernel of right almost split}, it is difficult in general to describe sink (resp.\ source) resolution of $P$ (resp.\ $I$), as pointed out in Remark \ref{Ai with big i}.
In the rest of this subsection, we study two special cases where we can describe sink (resp.\ source) resolutions of projective (resp.\ injective) objects.

The first case is given by cotilting modules.
Let $\Lambda$ be a finite dimensional $k$-algebra, and let $U\in\mod\Lambda$ be a cotilting $\Lambda$-module with ${\rm inj.dim}\,U=n$ (that is, $\Dd U$ is a tilting $\Lambda^\mathrm{op}$-module with ${\rm proj.dim}\,\Dd U=n$). The classical results given in \cite[Section 5]{AR5} imply that the full subcategory
\[{}^\perp U:=\{X\in\mod\Lambda\mid{\rm Ext}^i_\Lambda(X,U)=0\mbox{ for all $i>0$}\}\]
of $\mod\Lambda$ is an exact category with enough projectives $\mathsf{add}\Lambda$, enough injectives $\mathsf{add} U$ and almost split extensions.
The following result is an application of Auslander--Buchweitz theory \cite{AB}.

\begin{proposition}\label{sink for cotilting}
Under the above setting, the following assertions hold.
\begin{enumerate}[\rm(1)]
\item Each indecomposable projective object $P\in{}^\perp U$ has a sink resolution
\[0\to I_n\to\cdots\to I_2\to A_1\to P\]
in ${}^\perp U$ such that each $I_i$ is injective in ${}^\perp U$, where the sequence is $0\to A_1\to P$ if $n\le 1$.
\item Each indecomposable injective object $I\in{}^\perp U$ has a source resolution
\[I\to A^1\to P^2\to\cdots\to P^n\to0\]
in ${}^\perp U$ such that each $P^i$ is projective in ${}^\perp U$, where the sequence is $I\to A^1\to 0$ if $n\le 1$.
\end{enumerate}
\end{proposition}

\begin{proof}
(1) If $n\le 1$, then the assertion is clear.
In the rest, we assume $n\ge2$. 

Let $({}^\perp U)^\perp:=\{X\in\mod\Lambda\mid{\rm Ext}^i_\Lambda({}^\perp U,X)=0\ \mbox{for all $i>0$}\}$. The following results are basic in Auslander--Buchweitz theory \cite{AB}, see \cite[Theorem 5.5, Proposition 3.3]{AR5}.
\begin{itemize}
\item Any $C\in\mod\Lambda$ admits an exact sequence $0\to Y_C\xrightarrow{b} X_C\xrightarrow{a} C\to0$ in $\mod\Lambda$ such that $X_C\in{}^\perp U$, $Y_C\in({}^\perp U)^\perp$ and $a$ is right minimal.
\item $({}^\perp U)^\perp$ consists of $Y\in\mod\Lambda$ that admits an exact sequence $0\to U_\ell\xrightarrow{c_\ell}\cdots\xrightarrow{c_1} U_0\xrightarrow{c_0} Y\to0$ such that  $U_i\in\mathsf{add} U$, $c_i$ is right minimal and $\ell=\max\{i\ge1\mid{\rm Ext}^i_\Lambda(Y,U)\neq0\}$. 
\end{itemize}
Applying above results to $C:={\rm rad}\,P$ and $Y:=Y_{{\rm rad}\,P}$, we obtain exact sequences
\[0\to Y_{{\rm rad}\,P}\xrightarrow{b} X_{{\rm rad}\,P}\xrightarrow{a}{\rm rad}\,P\to0\ \mbox{ and }\ 0\to U_\ell\xrightarrow{c_\ell}\cdots\xrightarrow{c_1} U_0\xrightarrow{c_0} Y_{{\rm rad}\,P}\to0.\]
It is straightforward to check that these sequences are exact after applying ${\rm Hom}_\Lambda({}^\perp U,-)$.
Since ${\rm Ext}^i_\Lambda({\rm rad}\,P,U)={\rm Ext}^{i+1}_\Lambda(P/{\rm rad}\,P,U)=0$ holds for all $i\ge n$, we have ${\rm Ext}^i_\Lambda(Y_{{\rm rad}\,P},U)={\rm Ext}^{i+1}_\Lambda({\rm rad}\,P,U)=0$ holds for all $i\ge n-1$.
Thus we can assume $\ell=n-2$.
By combining the above sequences and the inclusion $\iota\colon{\rm rad}\,P\to P$, we obtain a sink resolution
\[0\to U_{n-2}\xrightarrow{c_{n-2}}\cdots\xrightarrow{c_1} U_0\xrightarrow{bc_0} X_{{\rm rad}\,P}\xrightarrow{\iota a} P.\]
Setting $I_{i+2}:=U_i$, we obtain a desired sequence.

(2) Let $\Gamma:={\rm End}_\Lambda(U)$. 
Then we have a duality $\mathbb{R}\mathrm{Hom}_{\Gamma}(-,U):\operatorname{D^b}(\mod\Gamma)\simeq \operatorname{D^b}(\mod\Lambda)$ (e.g.\ \cite[Corollary 2.11]{M}), which restricts to a duality
${\rm Hom}_\Gamma(-,U)\colon{}^\perp_\Gamma U\simeq {}^\perp_\Lambda U$ of exact categories which restricts to dualities $\mathsf{add}_\Gamma\Gamma\simeq\mathsf{add}_\Lambda U$ and $\mathsf{add}_\Gamma U\simeq\mathsf{add}_\Lambda\Lambda$.
Since it sends sink resolutions to source resolutions, the assertion follows from (1).
\end{proof}

\begin{example}\label{An example}
For $n\ge3$, let $A_n$ be the following quiver
\[\xymatrix@C1.5em{A_n:&1&2\ar[l]&3\ar[l]&\cdots\ar[l]&n-1\ar[l]&n\ar[l]}\]
and $\Lambda$ the quotient of $kA_n$ modulo the longest path.
Then $P_{n-1}=I_1$ and $P_n=I_2$ holds.
Fix $1\le \ell\le n-1$, and let
\[U:=(\bigoplus_{i=\ell}^nP_i)\oplus(\bigoplus_{i=3}^{\ell+1}I_i).\]
This is a cotilting $\Lambda$-module with ${\rm inj.dim}\,U\le2$, and we have
\[{}^\perp U=\mathsf{add}(U,\mod\Gamma)\ \mbox{ for }\ \Gamma:=\Lambda/(\sum_{i=\ell+1}^ne_i)\]
by an explicit calculation.
The almost split sequences in ${}^\perp U$ are those in $\mod\Gamma$ and
\[I_i^\Gamma\to I_{i+1}^\Gamma\oplus I_i\to I_{i+1}\ \mbox{ for }\ 2\le i\le\ell,\]
where $I_i^\Gamma$ is the injective hull of $S_i$ in $\mod\Gamma$. The sink resolutions for indecomposable projectives in ${}^\perp U$ are
\[0\to P_1,\ 0\to P_{i-1}\to P_i\ \mbox{for $2\le i\le n-1$, and}\ 0\to P_\ell\to P_{n-1}\oplus(P_\ell/S_1)\to P_n,\]
and the source resolutions for indecomposable injectives in ${}^\perp U$ are
\begin{eqnarray*}
&P_\ell\to P_{\ell+1}\oplus(P_\ell/S_1)\to P_n\to0,\ P_i\to P_{i+1}\to0\ \mbox{ for $\ell+1\le i\le n-2$,}&\\
&I_i\to I_{i+1}\to0\ \mbox{ for $1\le i\le\ell$, and}\ I_{\ell+1}\to0.&
\end{eqnarray*}
For example, $\mathrm{AR}_{\rm ET}({}^\perp U)$ for $n=7$ and $\ell=4$ is the following.
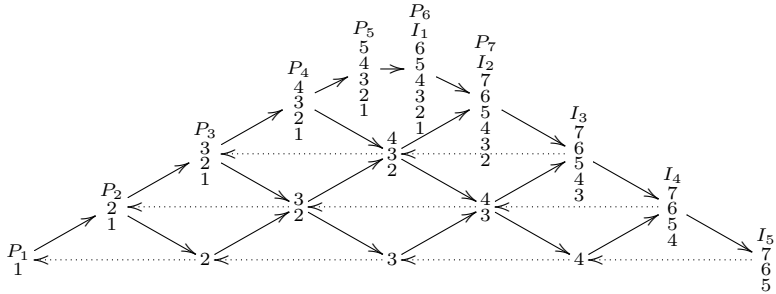
\[\begin{xy}
0;<7pt,0pt>:<0pt,4pt>::
(0,0)*{{\begin{smallmatrix}P_1\\ 1\end{smallmatrix}}}="1",
(5,5)*{{\begin{smallmatrix}P_2\\ 2\\ 1\end{smallmatrix}}}="2",
(10,0)*{{\begin{smallmatrix}2\end{smallmatrix}}}="3",
(10,10)*{{\begin{smallmatrix}P_3\\ 3\\ 2\\ 1\end{smallmatrix}}}="4",
(15,5)*{{\begin{smallmatrix}3\\ 2\end{smallmatrix}}}="5",
(15,15)*{{\begin{smallmatrix}P_4\\ 4\\ 3\\ 2\\ 1\end{smallmatrix}}}="6",
(20,0)*{{\begin{smallmatrix}3\end{smallmatrix}}}="7",
(20,10)*{{\begin{smallmatrix}4\\ 3\\ 2\end{smallmatrix}}}="8",
(25,5)*{{\begin{smallmatrix}4\\ 3\end{smallmatrix}}}="9",
(25,15)*{{\begin{smallmatrix}P_7\\ I_2\\ 7\\ 6\\ 5\\ 4\\ 3\\ 2\end{smallmatrix}}}="10",
(30,0)*{{\begin{smallmatrix}4\end{smallmatrix}}}="11",
(30,10)*{{\begin{smallmatrix}I_3\\ 7\\ 6\\ 5\\ 4\\ 3\end{smallmatrix}}}="12",
(35,5)*{{\begin{smallmatrix}I_4\\ 7\\ 6\\ 5\\ 4\end{smallmatrix}}}="13",
(40,0)*{{\begin{smallmatrix}I_5\\ 7\\ 6\\ 5\end{smallmatrix}}}="14",
(18.5,18)*{{\begin{smallmatrix}P_5\\ 5\\ 4\\ 3\\ 2\\ 1\end{smallmatrix}}}="p",
(21.5,18)*{{\begin{smallmatrix}P_6\\ I_1\\ 6\\ 5\\ 4\\ 3\\ 2\\ 1\end{smallmatrix}}}="q",
\ar"1";"2",
\ar"2";"3",
\ar"2";"4",
\ar"3";"5",
\ar"4";"5",
\ar"4";"6",
\ar"5";"7",
\ar"5";"8",
\ar"6";"8",
\ar"7";"9",
\ar"8";"9",
\ar"8";"10",
\ar"9";"11",
\ar"9";"12",
\ar"10";"12",
\ar"11";"13",
\ar"12";"13",
\ar"13";"14",
\ar"6";"p",
\ar"p";"q",
\ar"q";"10",
\ar@{.>}"3";"1",
\ar@{.>}"7";"3",
\ar@{.>}"11";"7",
\ar@{.>}"14";"11",
\ar@{.>}"5";"2",
\ar@{.>}"9";"5",
\ar@{.>}"13";"9",
\ar@{.>}"8";"4",
\ar@{.>}"12";"8",
\end{xy}\]
\end{example}

The second case is given by Cohen--Macaulay representations.

Let $R$ be a complete local Cohen--Macaulay ring with canonical module $\omega_R$.
Let $\Lambda$ be an \emph{$R$-order} (that is, an $R$-algebra which is maximal Cohen--Macaulay as an $R$-module), and let
\[\mathsf{CM}\Lambda:=\{X\in\mod\Lambda\mid\mbox{$X$ is maximal Cohen--Macaulay as an $R$-module}\}.\]
be the category of \emph{Cohen--Macaulay $\Lambda$-modules}.
This is an exact category with enough projectives $\mathsf{add}\,\Lambda$ and enough injectives $\mathsf{add}\,\omega_\Lambda$ for $\omega_\Lambda:={\rm Hom}_R(\Lambda,\omega_R)$. There is an equivalence
\[\nu:=\omega_\Lambda\otimes_\Lambda-:\mathsf{add}\,\Lambda\simeq\mathsf{add}\,\omega_\Lambda\]
called the \emph{Nakayama functor}. The category $\mathsf{CM}\Lambda$ has almost split extensions if and only if $\Lambda$ is an \emph{isolated singularity} (that is, $\Lambda\otimes_RR_{\mathfrak p}$ has global dimension $\dim R_{\mathfrak p}$ for any non-maximal prime ideal ${\mathfrak p}$ of $R$) \cite{A3,Y,LW}.
As in the case of Proposition~\ref{sink for cotilting}, we have the following observation.

\begin{proposition}\label{sink for CM}
Let $R$ be a complete local Cohen--Macaulay ring of $\dim R=d\ge2$, and $\Lambda$ an $R$-order.
\begin{enumerate}[\rm(1)]
\item Each indecomposable projective object $P\in\mathsf{CM}\Lambda$ has a sink resolution
\[0\to I_d\to\cdots\to I_2\to A_1\to P\]
such that each $I_i$ is injective in $\mathsf{CM}\Lambda$ and $I_d=\nu(P)$.
\item Each indecomposable injective object $I\in\mathsf{CM}\Lambda$ has a source resolution
\[I\to A^1\to P^2\to\cdots\to P^d\to0\]
such that each $P^i$ is projective in $\mathsf{CM}\Lambda$ and $P^d=\nu^{-1}(I)$.
\end{enumerate}
\end{proposition}

\begin{proof}
This is the case $n=1$ in \cite[Theorem 3.4.3]{I5}. Although the proof there is written in a slightly more special setting, it works without substantial changes.
\end{proof}

\begin{example}\label{A1 example}
Let $R=k[[x,y,u,v]]/(xy-uv)$ be a simple singularity of type $A_1$ of dimension three.
Then $\mathsf{CM} R$ has 3 indecomposable objects up to isomorphisms: $R$, and ideals $(x,u)$ and $(x,v)$ \cite{Y,LW},
and $\mathrm{AR}_{\rm ET}(\mathsf{CM} R)$ is the following.
\[\xymatrix@R1em@C2.5em{
(x,u)\ar@{=>}[dr]^{a=1}_{b=v/x}&&(x,v)\ar@{=>}[dr]^e_f\ar@{.>}[ll]&&(x,u)\ar@{.>}[ll]\\
&R\ar@{=>}[ur]^{c=v}_{d=x}\ar@{=>}[dr]^{g=u}_{h=x}&&R\ar@{=>}[ur]^g_h\ar@{=>}[dr]^c_d\\
(x,v)\ar@{=>}[ur]^{e=1}_{f=u/x}&&(x,u)\ar@{=>}[ur]^a_b\ar@{.>}[ll]&&(x,v)\ar@{.>}[ll]}\]
The almost split sequences are
\[0\to (x,u)\xrightarrow{{\left[\begin{smallmatrix}a\\ b\end{smallmatrix}\right]}}R^{\oplus 2}\xrightarrow{{\left[\begin{smallmatrix}c&-d\end{smallmatrix}\right]}}(x,v)\to0\ \mbox{ and }\ 0\to (x,v)\xrightarrow{{\left[\begin{smallmatrix}e\\ f\end{smallmatrix}\right]}}R^{\oplus 2}\xrightarrow{{\left[\begin{smallmatrix}g&-h\end{smallmatrix}\right]}}(x,u)\to0.\]
The sink resolution of $R$ is
\[0\to R\xrightarrow{{\left[\begin{smallmatrix}x\\ y\\ u\\ v\end{smallmatrix}\right]}}R^{\oplus 4}\xrightarrow{{\left[\begin{smallmatrix}0&-h&0&g\\ -g&0&h&0\\ 0&d&-c&0\\ c&0&0&-d\end{smallmatrix}\right]}}(x,u)^{\oplus2}\oplus(x,v)^{\oplus2}\xrightarrow{{\left[\begin{smallmatrix}a&b&e&f\end{smallmatrix}\right]}} R,\]
and the source resolution of $R$ is
\[0\to R\xrightarrow{{\left[\begin{smallmatrix}d\\ c\\ h\\ g\end{smallmatrix}\right]}}(x,v)^{\oplus2}\oplus(x,u)^{\oplus2}\xrightarrow{{\left[\begin{smallmatrix}0&-f&0&b\\ -e&0&a&0\\ 0&e&-b&0\\ f&0&0&-a\end{smallmatrix}\right]}}R^{\oplus4}\xrightarrow{{\left[\begin{smallmatrix}x&y&u&v\end{smallmatrix}\right]}} R.\]
The category $\mathsf{CM} R$ is presented by the Auslander--Reiten quiver with mesh relations
\[ca=db\ \mbox{ and }\ ge=hf\]
and non-mesh relations from $R$ to $R$
\[ec=bh,\ fc=bg,\ ed=ah\ \mbox{ and }\ fd=ag.\]
\end{example}

\section{Stable categories of extriangulated categories}\label{section_tau-categories}

In this section, we study the structure of (co)stable categories of extriangulated categories $\mathscr{C}$ as additive categories. 
Our approach is based on results on $\tau$-categories which we recall in the next subsection.

\subsection{Reminders on $\tau$-categories}

We start with recalling the notion of $\tau$-categories introduced in \cite[2.1]{I1}.

\begin{definition}\cite[2.1]{I1}\label{define tau category}
Let $\mathscr{D}$ be a Krull--Schmidt additive category.
A \emph{right $\tau$-sequence} of $C$ is a sink sequence (Definition~\ref{define sink sequence})
\begin{equation}\label{right tau}
A\xrightarrow{x} B\xrightarrow{y} C
\end{equation}
of $C$ such that either $x$ is left almost split or $A=0$ holds.
We call $\mathscr{D}$ a \emph{right $\tau$-category} if each indecomposable object in $\mathscr{D}$ has a right $\tau$-sequence.
Dually we define a \emph{left $\tau$-sequence} and a \emph{left $\tau$-category}.

We call $\mathscr{D}$ a \emph{$\tau$-category} if it is a left and right $\tau$-category.
In this case, each right $\tau$-sequence \eqref{right tau} is either a left $\tau$-sequence or satisfies $A=0$ \cite[Theorem 2.3]{I1}. We call \eqref{right tau} a \emph{$\tau$-sequence} in the former case, and call $C$ \emph{$\tau$-projective} in the latter case. Dually we define a \emph{$\tau$-injective} object.
We call a $\tau$-category $\mathscr{D}$ \emph{strict} if, for each $\tau$-sequence \eqref{right tau} in $\mathscr{D}$, $x$ is a monomorphism and $y$ is an epimorphism.

The \emph{Auslander--Reiten quiver} $\mathrm{AR}_\tau(\mathscr{D})=(Q_0,d,d',\tau)$ of a $\tau$-category $\mathscr{D}$ is a $\tau$-quiver defined as follows (cf.\ Proposition \ref{AR quiver is translation}).
\begin{enumerate}[$\bullet$]
\item $(Q_0,d,d'):=\mathrm{AR}(\mathscr{D})$ (Definition \ref{define AR quiver}).
\item $Q_0^p:=\{X\in Q_0\mid \mbox{$X$ is $\tau$-projective}\}$, $Q_0^i:=\{X\in Q_0\mid \mbox{$X$ is $\tau$-injective}\}$.
\item $\tau C:=A$ if there exists a $\tau$-sequence $A\to B\to C$ with $A,C\in Q_0$.
\end{enumerate}
\end{definition}


The class of $\tau$-categories contains various important additive categories.

\begin{example}\label{example of tau-category}
\begin{enumerate}[\rm(1)]
\item Let $\Lambda$ be a finite dimensional algebra over a field $k$. Then $\mod\Lambda$ is a strict $\tau$-category.
In fact, almost split extensions are $\tau$-sequences, the complex $0\to {\rm rad} P\to P$ with indecomposable projective $P$ is a right $\tau$-sequence, and the complex $I\to I/{\rm soc} I\to 0$ with indecomposable injective $I$ is a left $\tau$-sequence.
\item A triangulated category with almost split extensions is a $\tau$-category. In fact, almost split sequences with non-zero middle terms give $\tau$-sequences. On the other hand, an exact category with almost split extensions is not necessarily a $\tau$-category, see (4)(5) below.
\item \cite[Proposition 8.4]{I1} The complete mesh category of a $\tau$-species (Definition ~\ref{define tau-species}) is a $\tau$-category (see Proposition \ref{property of mesh category} for details).
\item \cite[Theorem 2.1]{I4} Let $\Lambda$ be a finite dimensional $k$-algebra, and let $U$ be a cotilting $\Lambda$-module. If ${\rm inj.dim}\,U\le 1$, then ${}^\perp U$ is a strict $\tau$-category.
\item Let $R$ be a complete local Cohen--Macaulay ring of $\dim R=d$, and $\Lambda$ an $R$-order which is an isolated singularity. Then $\mathsf{CM}\Lambda$ is a $\tau$-category if and only if it is a strict $\tau$-category if and only if $d\le 2$ (see \cite[Example 2.2(2)]{I1} and Proposition  \ref{sink for CM}).
On the other hand, the stable and costable categories $\underline{\mathsf{CM}}\Lambda$ and $\overline{\mathsf{CM}}\Lambda$ are always $\tau$-categories \cite[Theorem 3.4.5]{I5}.
\end{enumerate}
\end{example}

Note that $\mathrm{AR}_\tau(\mathscr{D})$ and $\mathrm{AR}_{\rm ET}(\mathscr{D})$ given in Definition \ref{define valued translation quiver} are related as follows.

\begin{remark}
Assume that $\mathscr{D}=(\mathscr{D},\mathbb{E},\mathfrak{s})$ is an extriangulated category which is a $\tau$-category. Then $\mathrm{AR}_\tau(\mathscr{D})$ and $\mathrm{AR}_{\rm ET}(\mathscr{D})$ are usually the same, but not always. Their valued quiver parts are of course the same. Also the maps $\tau:Q_0\setminus Q_0^p\to Q_0\setminus Q_0^i$ in $\mathrm{AR}_\tau(\mathscr{D})$ and $\tau_{\rm ET}:Q_0\setminus Q_{0,{\rm ET}}^p\to Q_0\setminus Q_{0,{\rm ET}}^i$ in $\mathrm{AR}_{\rm ET}(\mathscr{D})$ coincide over $Q_0\setminus(Q_0^p\cup Q_{0,{\rm ET}}^p)$ since almost split sequences with non-zero middle terms give $\tau$-sequences.
But $Q_0^p$ and $Q_{0,{\rm ET}}^p$ are sometimes different, as illustrated in the examples below.
\begin{enumerate}[\rm(1)]
\item In a triangulated category $\mathscr{D}=\operatorname{K^b}(\proj k)$, we have $Q_0^p=Q_0$ and $Q_{0,{\rm ET}}^p=\emptyset$.
\item In an exact category $\mathscr{D}=\mathsf{CM} R$ with $R=k[[x,y]]$, we have $Q_0^p=\emptyset$ and $Q_{0,{\rm ET}}^p=Q_0$.
\end{enumerate}
\end{remark}

The key notion of ladders was introduced by Igusa--Todorov \cite[Definition 2.14]{IT}. Here we need the following modified version \cite{I1}.

\begin{definition}\cite[3.2]{I1}
Let $\mathscr{D}$ be a right $\tau$-category. A commutative diagram
\begin{equation}\label{ladder diagram}
\xymatrix{
Y_0&Y_1\ar[l]_{f_1}&Y_2\ar[l]_{f_2}&Y_3\ar[l]_{f_3}&\cdots\ar[l]_{f_4}\\
X_0\ar[u]_{a_0}&X_1\ar[l]_{g_1}\ar[u]_{a_1}&X_2\ar[l]_{g_2}\ar[u]_{a_2}&X_3\ar[l]_{g_3}\ar[u]_{a_3}&\cdots\ar[l]_{g_4}
}\end{equation}
in $\mathscr{D}$ is called a \emph{right ladder} of $a_0\in\mathscr{D}(X_0,Y_0)$ if the following condition is satisfied for all $i\ge0$.
\begin{enumerate}[\rm$\bullet$]
\item There exists $h_{i+1}\in\mathscr{D}(U_{i+1},X_i)$ such that
\[X_{i+1}\oplus U_{i+1}\xrightarrow{\left[\bsm a_{i+1}&0\\ -g_{i+1}&h_{i+1} \esm\right]}Y_{i+1}\oplus X_i\xrightarrow{[f_{i+1}\ a_i]}Y_i\]
is a direct sum of right $\tau$-sequences.
\end{enumerate}
\end{definition}

The existence theorem of ladders was first shown by Igusa--Todorov \cite[Theorem 2.15]{IT} for artin algebras under certain technical assumptions. 
Later it was proved for arbitrary $\tau$-categories by a much simpler method in the following form.

\begin{theorem}[Existence Theorem of Ladders]\label{existence theorem}${}$\cite[Theorem 3.3(2)]{I1}
Let $\mathscr{D}$ be a $\tau$-category and $A\in\mathscr{D}$ an object.
\begin{enumerate}[\rm(1)]
\item The source morphism of $A$ has a right ladder.
\item The zero morphism $0\in\mathscr{D}(0,A)$ has a right ladder.
\end{enumerate}
\end{theorem}

As an application of Theorem~\ref{existence theorem}, Igusa--Todorov's Radical Layers Theorem \cite[Theorem 4.3]{IT} was proved for an arbitrary $\tau$-category in \cite[Theorem 4.2]{I1} (cf.\ Corollary \ref{radical layer} below).

Next we introduce certain functions $\theta_i$, which are central in Auslander--Reiten Combinatorics (cf.\ Corollary \ref{calculate dimension} below). Notice that the sequence $(\theta_iX)_{i\ge0}$ is a modification of an \emph{additive function stopping at $X$} due to Gabriel \cite[Section 6.5]{G} and a \emph{hammock} due to Brenner \cite{B,RV}. 
 
\begin{definition}\label{define theta_i}
Let $\mathscr{D}$ be a $\tau$-category.
\begin{enumerate}[\rm(1)]
\item We denote by $K_0(\mathscr{D})$ the Grothendieck group of the additive category $\mathscr{D}$. Since $\mathscr{D}$ is Krull--Schmidt, $K_0(\mathscr{D})$ is the free abelian group with basis $\ind\mathscr{D}$.
We identify the set of isomorphism classes of objects in $\mathscr{C}$ with the submonoid $K_0(\mathscr{D})_+$ of $K_0(\mathscr{D})$ generated by $\ind\mathscr{D}$.
Any element $X\in K_0(\mathscr{D})$ can be written uniquely as 
\[X=X_+-X_-\]
for some objects $X_+,X_-\in K_0(\mathscr{D})_+$ which do not have common non-zero direct summands.
\item \cite[7.2]{I1}\cite[1.3.3]{I4} We denote a right $\tau$-sequence of $X\in\ind\mathscr{D}$ by
\begin{equation}\label{right tau of X}
\tau X\to\theta X\to X.
\end{equation}
We extend $\theta$ and $\tau$ canonically to the monoid endomorphisms of $K_0(\mathscr{D})_+$.
For each $i\ge0$, we define the map $\theta_i:K_0(\mathscr{D})_+\to K_0(\mathscr{D})_+$ inductively as follows:
Let $\theta_0={\rm id}$ and $\theta_1=\theta$. For $i\ge2$ and $X\in\mathscr{D}$, let
\begin{equation}\label{theta_i}
\theta_iX=(\theta(\theta_{i-1}X)-\tau(\theta_{i-2}X))_+.
\end{equation}
\item Let $Q=(Q_0,d,d',\tau)$ be a $\tau$-quiver (Definition \ref{define valued translation quiver}), $\mathbb{Z}Q$ (resp.\ $\mathbb{N}Q_0$) the free abelian group (resp.\ monoid) with basis $Q_0$. Each element $X\in\mathbb{Z}Q_0$ can be written uniquely as $X=X_+-X_-$ for some elements $X_+,X_-\in\mathbb{N}Q_0$ whose supports are disjoint. 

For each $X\in Q_0$, let
\[\theta X:=\bigoplus_{W\in Q_0}d_{WX}W\in\mathbb{N}Q_0.\]
We extend $\theta$ and $\tau$ linearly to the monoid endomorphisms of $\mathbb{N}Q_0$.
For each $i\ge0$, we define the map $\theta_i:\mathbb{N}Q_0\to\mathbb{N}Q_0$ inductively as follows: Let $\theta_0={\rm id}$ and $\theta_1=\theta$. For $i\ge2$ and $X\in\mathbb{N}Q_0$, define $\theta_iX$ by the equality \eqref{theta_i}.

We call $Q$ \emph{strict} if $\theta_iX=\theta(\theta_{i-1}X)-\tau(\theta_{i-2}X)$ holds for all $i\ge2$ and $X\in Q_0$ (in other words, one can drop $(-)_+$ in \eqref{theta_i}).
\end{enumerate}
\end{definition}

Let $\mathscr{D}$ be a $\tau$-category with $Q:=\mathrm{AR}_\tau(\mathscr{D})$ (Definition \ref{define tau category}). Then $K_0(\mathscr{D})_+$ is identified with $\mathbb{N}Q_0$, and the two definitions of $\theta_i$ given in (2) and (3) above clearly coincide. An important result is that
the map $\theta_i$ is a monoid endomorphism \cite[7.2(1)]{I1} (cf.\ Corollary \ref{calculate dimension}(2) below). This is not completely obvious from the definition because of the piecewise linear property of the map $(-)_+$.

\begin{example}\label{example of theta_i}
\begin{enumerate}[\rm(1)]
\item Let $Q$ be a connected valued quiver, and $\mathbb{Z}Q$ the corresponding $\tau$-quiver. Then $\mathbb{Z}Q$ is strict if and only if $Q$ is non-Dynkin. We illustrate this by the following two examples.
\item Consider the $\tau$-quiver $\mathbb{Z}E_6$. For example, $\mathbb{Z}E_6=\mathrm{AR}_\tau(\operatorname{K^b}(\proj\Lambda))$ for the path algebra $\Lambda$ of a Dynkin quiver of type $E_6$ (Example \ref{example of tau-category}(2)).
\[
\resizebox{\textwidth}{!}{
\begin{xy} 0;<14pt,0pt>:<0pt,14pt>::
(6,6) *+{\cdots},
(6,2) *+{\cdots},
(46,6) *+{\cdots},
(46,2) *+{\cdots},
(8,8) *+{\bullet} ="40",
%
(8,4) *+{\bullet} ="21",
(10,6) *+{\bullet} ="31",
(12,8) *+{\bullet} ="41",
(10,4) *+{\bullet} ="51",
%
(8,0) *+{\bullet} ="02",
(10,2) *+{\bullet} ="12",
(12,4) *+{\bullet} ="22",
(14,6) *+{\bullet} ="32",
(16,8) *+{\bullet} ="42",
(14,4) *+{\bullet} ="52",
%
(12,0) *+{\bullet} ="03",
(14,2) *+{\bullet} ="13",
(16,4) *+{\bullet} ="23",
(18,6) *+{\bullet} ="33",
(20,8) *+{\bullet} ="43",
(18,4) *+{\bullet} ="53",
%
(16,0) *+{\bullet} ="04",
(18,2) *+{\bullet} ="14",
(20,4) *+{\bullet} ="24",
(22,6) *+{\bullet} ="34",
(24,8) *+{\bullet} ="44",
(22,4) *+{\bullet} ="54",
%
(20,0) *+{\bullet} ="05",
(22,2) *+{\bullet} ="15",
(24,4) *+{\bullet} ="25",
(26,6) *+{\bullet} ="35",
(28,8) *+{\bullet} ="45",
(26,4) *+{\bullet} ="55",
%
(24,0) *+{\bullet} ="06",
(26,2) *+{\bullet} ="16",
(28,4) *+{\bullet} ="26",
(30,6) *+{\bullet} ="36",
(32,8) *+{\bullet} ="46",
(30,4) *+{\bullet} ="56",
%
(28,0) *+{\bullet} ="07",
(30,2) *+{\bullet} ="17",
(32,4) *+{\bullet} ="27",
(34,6) *+{\bullet} ="37",
(36,8) *+{\bullet} ="47",
(34,4) *+{\bullet} ="57",
%
(32,0) *+{\bullet} ="08",
(34,2) *+{\bullet} ="18",
(36,4) *+{\bullet} ="28",
(38,6) *+{Z_3} ="38",
(40,8) *+{\bullet} ="48",
(38,4) *+{Z_2} ="58",
%
(36,0) *+{\bullet} ="09",
(38,2) *+{Z_1} ="19",
(40,4) *+{Y_2} ="29",
(42,6) *+{\bullet} ="39",
(44,8) *+{\bullet} ="49",
(42,4) *+{\bullet} ="59",
%
(40,0) *+{Y_1} ="010",
(42,2) *+{X} ="110",
(44,4) *+{\bullet} ="210",
(44,0) *+{\bullet} ="011",
"40", {\ar"31"},
%
"21", {\ar"31"},
"31", {\ar"41"},
"21", {\ar"51"},
"21", {\ar"12"},
"31", {\ar"22"},
"41", {\ar"32"},
"51", {\ar"22"},
%
"02", {\ar"12"},
"12", {\ar"22"},
"22", {\ar"32"},
"32", {\ar"42"},
"22", {\ar"52"},
"12", {\ar"03"},
"22", {\ar"13"},
"32", {\ar"23"},
"42", {\ar"33"},
"52", {\ar"23"},
%
"03", {\ar"13"},
"13", {\ar"23"},
"23", {\ar"33"},
"33", {\ar"43"},
"23", {\ar"53"},
"13", {\ar"04"},
"23", {\ar"14"},
"33", {\ar"24"},
"43", {\ar"34"},
"53", {\ar"24"},
%
"04", {\ar"14"},
"14", {\ar"24"},
"24", {\ar"34"},
"34", {\ar"44"},
"24", {\ar"54"},
"14", {\ar"05"},
"24", {\ar"15"},
"34", {\ar"25"},
"44", {\ar"35"},
"54", {\ar"25"},
%
"05", {\ar"15"},
"15", {\ar"25"},
"25", {\ar"35"},
"35", {\ar"45"},
"25", {\ar"55"},
"15", {\ar"06"},
"25", {\ar"16"},
"35", {\ar"26"},
"45", {\ar"36"},
"55", {\ar"26"},
%
"06", {\ar"16"},
"16", {\ar"26"},
"26", {\ar"36"},
"36", {\ar"46"},
"26", {\ar"56"},
"16", {\ar"07"},
"26", {\ar"17"},
"36", {\ar"27"},
"46", {\ar"37"},
"56", {\ar"27"},
%
"07", {\ar"17"},
"17", {\ar"27"},
"27", {\ar"37"},
"37", {\ar"47"},
"27", {\ar"57"},
"17", {\ar"08"},
"27", {\ar"18"},
"37", {\ar"28"},
"47", {\ar"38"},
"57", {\ar"28"},
%
"08", {\ar"18"},
"18", {\ar"28"},
"28", {\ar"38"},
"38", {\ar"48"},
"28", {\ar"58"},
"18", {\ar"09"},
"28", {\ar"19"},
"38", {\ar"29"},
"48", {\ar"39"},
"58", {\ar"29"},
%
"09", {\ar"19"},
"19", {\ar"29"},
"29", {\ar"39"},
"39", {\ar"49"},
"29", {\ar"59"},
"19", {\ar"010"},
"29", {\ar"110"},
"39", {\ar"210"},
"59", {\ar"210"},
"010", {\ar"110"},
"110", {\ar"210"},
"110", {\ar"011"},
\end{xy}
}
\]
For $X$ given in the previous picture, we have $\theta X=Y_1\oplus Y_2$, $\theta_2X=Z_1\oplus Z_2\oplus Z_3$ and so on. The following picture shows $\theta_iX$ for $i\ge0$.
\[
\resizebox{\textwidth}{!}{
\begin{xy} 0;<14pt,0pt>:<0pt,14pt>::
(6,6) *+{\cdots},
(6,2) *+{\cdots},
(46,6) *+{\cdots},
(46,2) *+{\cdots},
(8,-1) *+{\theta_{17}X},
(10,-1) *+{\theta_{16}X},
(12,-1) *+{\theta_{15}X},
(14,-1) *+{\theta_{14}X},
(16,-1) *+{\theta_{13}X},
(18,-1) *+{\theta_{12}X},
(20,-1) *+{\theta_{11}X},
(22,-1) *+{\theta_{10}X},
(24,-1) *+{\theta_9X},
(26,-1) *+{\theta_8X},
(28,-1) *+{\theta_7X},
(30,-1) *+{\theta_6X},
(32,-1) *+{\theta_5X},
(34,-1) *+{\theta_4X},
(36,-1) *+{\theta_3X},
(38,-1) *+{\theta_2X},
(40,-1) *+{\theta X},
(42,-1) *+{X},
(8,8) *+{0} ="40",
%
(8,4) *+{0} ="21",
(10,6) *+{0} ="31",
(12,8) *+{0} ="41",
(10,4) *+{0} ="51",
%
(8,0) *+{0} ="02",
(10,2) *+{0} ="12",
(12,4) *+{0} ="22",
(14,6) *+{0} ="32",
(16,8) *+{0} ="42",
(14,4) *+{0} ="52",
%
(12,0) *+{0} ="03",
(14,2) *+{0} ="13",
(16,4) *+{0} ="23",
(18,6) *+{0} ="33",
(20,8) *+{0} ="43",
(18,4) *+{0} ="53",
%
(16,0) *+{0} ="04",
(18,2) *+{0} ="14",
(20,4) *+{0} ="24",
(22,6) *+{1} ="34",
(24,8) *+{1} ="44",
(22,4) *+{0} ="54",
%
(20,0) *+{0} ="05",
(22,2) *+{0} ="15",
(24,4) *+{1} ="25",
(26,6) *+{1} ="35",
(28,8) *+{0} ="45",
(26,4) *+{1} ="55",
%
(24,0) *+{0} ="06",
(26,2) *+{1} ="16",
(28,4) *+{2} ="26",
(30,6) *+{1} ="36",
(32,8) *+{1} ="46",
(30,4) *+{1} ="56",
%
(28,0) *+{1} ="07",
(30,2) *+{2} ="17",
(32,4) *+{2} ="27",
(34,6) *+{2} ="37",
(36,8) *+{1} ="47",
(34,4) *+{1} ="57",
%
(32,0) *+{1} ="08",
(34,2) *+{1} ="18",
(36,4) *+{2} ="28",
(38,6) *+{1} ="38",
(40,8) *+{0} ="48",
(38,4) *+{1} ="58",
%
(36,0) *+{0} ="09",
(38,2) *+{1} ="19",
(40,4) *+{1} ="29",
(42,6) *+{0} ="39",
(44,8) *+{\bullet} ="49",
(42,4) *+{0} ="59",
%
(40,0) *+{1} ="010",
(42,2) *+{1} ="110",
(44,4) *+{\bullet} ="210",
(44,0) *+{\bullet} ="011",
"40", {\ar"31"},
%
"21", {\ar"31"},
"31", {\ar"41"},
"21", {\ar"51"},
"21", {\ar"12"},
"31", {\ar"22"},
"41", {\ar"32"},
"51", {\ar"22"},
%
"02", {\ar"12"},
"12", {\ar"22"},
"22", {\ar"32"},
"32", {\ar"42"},
"22", {\ar"52"},
"12", {\ar"03"},
"22", {\ar"13"},
"32", {\ar"23"},
"42", {\ar"33"},
"52", {\ar"23"},
%
"03", {\ar"13"},
"13", {\ar"23"},
"23", {\ar"33"},
"33", {\ar"43"},
"23", {\ar"53"},
"13", {\ar"04"},
"23", {\ar"14"},
"33", {\ar"24"},
"43", {\ar"34"},
"53", {\ar"24"},
%
"04", {\ar"14"},
"14", {\ar"24"},
"24", {\ar"34"},
"34", {\ar"44"},
"24", {\ar"54"},
"14", {\ar"05"},
"24", {\ar"15"},
"34", {\ar"25"},
"44", {\ar"35"},
"54", {\ar"25"},
%
"05", {\ar"15"},
"15", {\ar"25"},
"25", {\ar"35"},
"35", {\ar"45"},
"25", {\ar"55"},
"15", {\ar"06"},
"25", {\ar"16"},
"35", {\ar"26"},
"45", {\ar"36"},
"55", {\ar"26"},
%
"06", {\ar"16"},
"16", {\ar"26"},
"26", {\ar"36"},
"36", {\ar"46"},
"26", {\ar"56"},
"16", {\ar"07"},
"26", {\ar"17"},
"36", {\ar"27"},
"46", {\ar"37"},
"56", {\ar"27"},
%
"07", {\ar"17"},
"17", {\ar"27"},
"27", {\ar"37"},
"37", {\ar"47"},
"27", {\ar"57"},
"17", {\ar"08"},
"27", {\ar"18"},
"37", {\ar"28"},
"47", {\ar"38"},
"57", {\ar"28"},
%
"08", {\ar"18"},
"18", {\ar"28"},
"28", {\ar"38"},
"38", {\ar"48"},
"28", {\ar"58"},
"18", {\ar"09"},
"28", {\ar"19"},
"38", {\ar"29"},
"48", {\ar"39"},
"58", {\ar"29"},
%
"09", {\ar"19"},
"19", {\ar"29"},
"29", {\ar"39"},
"39", {\ar"49"},
"29", {\ar"59"},
"19", {\ar"010"},
"29", {\ar"110"},
"39", {\ar"210"},
"59", {\ar"210"},
"010", {\ar"110"},
"110", {\ar"210"},
"110", {\ar"011"},
\end{xy}
}
\]
Thus $\theta_iX=0$ holds for all $i\ge 11$. Moreover, for $i\ge2$, the equality
\[\theta(\theta_{i-1}X)-\tau(\theta_{i-2}X)=\left\{\begin{array}{cc}
\theta_iX&i\neq12\\ -\nu^{-1}X&i=12
\end{array}\right.\]
holds, where $12$ is the Coxeter number of the root system of type $E_6$. Thus $\mathbb{Z}E_6$ is not strict.
It is well-known that these observations hold for all Dynkin quivers, e.g.\ \cite[Section 6.5]{G}.
\item Consider the $\tau$-quiver $\mathbb{Z}\widetilde{E}_6$. For example, $\mathbb{Z}\widetilde{E}_6=\mathrm{AR}_\tau(\mathsf{CM}^{\mathbb{Z}}R)$ for the category $\mathsf{CM}^{\mathbb{Z}}R$ of $\mathbb{Z}$-graded Cohen--Macaulay modules over a simple singularity $R=\mathbb{C}[x,y,z]/(x^4+y^3+z^2)$ of type $E_6$ with $(\mathrm{deg}\,x,\mathrm{deg}\,y,\mathrm{deg}\,z)=(3,4,6)$  (Example \ref{example of tau-category}(5), \cite[Section 8]{GL2}).
\[
\resizebox{\textwidth}{!}{
\begin{xy} 0;<14pt,0pt>:<0pt,14pt>::
(6,6) *+{\cdots},
(6,2) *+{\cdots},
(46,6) *+{\cdots},
(46,2) *+{\cdots},
(8,8) *+{\bullet} ="40",
(8,6) *+{\bullet} ="60",
(8,4) *+{\bullet} ="21",
(10,6) *+{\bullet} ="31",
(12,8) *+{\bullet} ="41",
(10,5) *+{\bullet} ="51",
(12,6) *+{\bullet} ="61",
(8,0) *+{\bullet} ="02",
(10,2) *+{\bullet} ="12",
(12,4) *+{\bullet} ="22",
(14,6) *+{\bullet} ="32",
(16,8) *+{\bullet} ="42",
(14,5) *+{\bullet} ="52",
(16,6) *+{\bullet} ="62",
(12,0) *+{\bullet} ="03",
(14,2) *+{\bullet} ="13",
(16,4) *+{\bullet} ="23",
(18,6) *+{\bullet} ="33",
(20,8) *+{\bullet} ="43",
(18,5) *+{\bullet} ="53",
(20,6) *+{\bullet} ="63",
(16,0) *+{\bullet} ="04",
(18,2) *+{\bullet} ="14",
(20,4) *+{\bullet} ="24",
(22,6) *+{\bullet} ="34",
(24,8) *+{\bullet} ="44",
(22,5) *+{\bullet} ="54",
(24,6) *+{\bullet} ="64",
(20,0) *+{\bullet} ="05",
(22,2) *+{\bullet} ="15",
(24,4) *+{\bullet} ="25",
(26,6) *+{\bullet} ="35",
(28,8) *+{\bullet} ="45",
(26,5) *+{\bullet} ="55",
(28,6) *+{\bullet} ="65",
(24,0) *+{\bullet} ="06",
(26,2) *+{\bullet} ="16",
(28,4) *+{\bullet} ="26",
(30,6) *+{\bullet} ="36",
(32,8) *+{\bullet} ="46",
(30,5) *+{\bullet} ="56",
(32,6) *+{\bullet} ="66",
(28,0) *+{\bullet} ="07",
(30,2) *+{\bullet} ="17",
(32,4) *+{\bullet} ="27",
(34,6) *+{\bullet} ="37",
(36,8) *+{\bullet} ="47",
(34,5) *+{\bullet} ="57",
(36,6) *+{\bullet} ="67",
(32,0) *+{\bullet} ="08",
(34,2) *+{\bullet} ="18",
(36,4) *+{\bullet} ="28",
(38,6) *+{Z_3} ="38",
(40,8) *+{\bullet} ="48",
(38,5) *+{Z_2} ="58",
(40,6) *+{\bullet} ="68",
(36,0) *+{\bullet} ="09",
(38,2) *+{Z_1} ="19",
(40,4) *+{Y_2} ="29",
(42,6) *+{\bullet} ="39",
(44,8) *+{\bullet} ="49",
(42,5) *+{\bullet} ="59",
(44,6) *+{\bullet} ="69",
(40,0) *+{Y_1} ="010",
(42,2) *+{X} ="110",
(44,4) *+{\bullet} ="210",
(44,0) *+{\bullet} ="011",
"40", {\ar"31"},
"60", {\ar"51"},
"21", {\ar"31"},
"31", {\ar"41"},
"21", {\ar"51"},
"51", {\ar"61"},
"21", {\ar"12"},
"31", {\ar"22"},
"41", {\ar"32"},
"51", {\ar"22"},
"61", {\ar"52"},
"02", {\ar"12"},
"12", {\ar"22"},
"22", {\ar"32"},
"32", {\ar"42"},
"22", {\ar"52"},
"52", {\ar"62"},
"12", {\ar"03"},
"22", {\ar"13"},
"32", {\ar"23"},
"42", {\ar"33"},
"52", {\ar"23"},
"62", {\ar"53"},
"03", {\ar"13"},
"13", {\ar"23"},
"23", {\ar"33"},
"33", {\ar"43"},
"23", {\ar"53"},
"53", {\ar"63"},
"13", {\ar"04"},
"23", {\ar"14"},
"33", {\ar"24"},
"43", {\ar"34"},
"53", {\ar"24"},
"63", {\ar"54"},
"04", {\ar"14"},
"14", {\ar"24"},
"24", {\ar"34"},
"34", {\ar"44"},
"24", {\ar"54"},
"54", {\ar"64"},
"14", {\ar"05"},
"24", {\ar"15"},
"34", {\ar"25"},
"44", {\ar"35"},
"54", {\ar"25"},
"64", {\ar"55"},
"05", {\ar"15"},
"15", {\ar"25"},
"25", {\ar"35"},
"35", {\ar"45"},
"25", {\ar"55"},
"55", {\ar"65"},
"15", {\ar"06"},
"25", {\ar"16"},
"35", {\ar"26"},
"45", {\ar"36"},
"55", {\ar"26"},
"65", {\ar"56"},
"06", {\ar"16"},
"16", {\ar"26"},
"26", {\ar"36"},
"36", {\ar"46"},
"26", {\ar"56"},
"56", {\ar"66"},
"16", {\ar"07"},
"26", {\ar"17"},
"36", {\ar"27"},
"46", {\ar"37"},
"56", {\ar"27"},
"66", {\ar"57"},
"07", {\ar"17"},
"17", {\ar"27"},
"27", {\ar"37"},
"37", {\ar"47"},
"27", {\ar"57"},
"57", {\ar"67"},
"17", {\ar"08"},
"27", {\ar"18"},
"37", {\ar"28"},
"47", {\ar"38"},
"57", {\ar"28"},
"67", {\ar"58"},
"08", {\ar"18"},
"18", {\ar"28"},
"28", {\ar"38"},
"38", {\ar"48"},
"28", {\ar"58"},
"58", {\ar"68"},
"18", {\ar"09"},
"28", {\ar"19"},
"38", {\ar"29"},
"48", {\ar"39"},
"58", {\ar"29"},
"68", {\ar"59"},
"09", {\ar"19"},
"19", {\ar"29"},
"29", {\ar"39"},
"39", {\ar"49"},
"29", {\ar"59"},
"59", {\ar"69"},
"19", {\ar"010"},
"29", {\ar"110"},
"39", {\ar"210"},
"59", {\ar"210"},
"010", {\ar"110"},
"110", {\ar"210"},
"110", {\ar"011"},
\end{xy}
}
\]
For $X$ given in the previous picture, we have $\theta X=Y_1\oplus Y_2$, $\theta_2X=Z_1\oplus Z_2\oplus Z_3$ and so on. The following picture shows $\theta_iX$ for $i\ge0$.
\[
\resizebox{\textwidth}{!}{
\begin{xy} 0;<14pt,0pt>:<0pt,14pt>::
(6,6) *+{\cdots},
(6,2) *+{\cdots},
(46,6) *+{\cdots},
(46,2) *+{\cdots},
(8,-2) *+{\theta_{17}X},
(10,-2) *+{\theta_{16}X},
(12,-2) *+{\theta_{15}X},
(14,-2) *+{\theta_{14}X},
(16,-2) *+{\theta_{13}X},
(18,-2) *+{\theta_{12}X},
(20,-2) *+{\theta_{11}X},
(22,-2) *+{\theta_{10}X},
(24,-2) *+{\theta_9X},
(26,-2) *+{\theta_8X},
(28,-2) *+{\theta_7X},
(30,-2) *+{\theta_6X},
(32,-2) *+{\theta_5X},
(34,-2) *+{\theta_4X},
(36,-2) *+{\theta_3X},
(38,-2) *+{\theta_2X},
(40,-2) *+{\theta X},
(42,-2) *+{X},
(8,8) *+{3} ="40",
(8,6) *+{3} ="60",
(8,4) *+{9} ="21",
(10,6) *+{6} ="31",
(12,8) *+{3} ="41",
(10,5) *+{6} ="51",
(12,6) *+{3} ="61",
(8,0) *+{3} ="02",
(10,2) *+{5} ="12",
(12,4) *+{8} ="22",
(14,6) *+{5} ="32",
(16,8) *+{2} ="42",
(14,5) *+{5} ="52",
(16,6) *+{2} ="62",
(12,0) *+{2} ="03",
(14,2) *+{5} ="13",
(16,4) *+{7} ="23",
(18,6) *+{4} ="33",
(20,8) *+{2} ="43",
(18,5) *+{4} ="53",
(20,6) *+{2} ="63",
(16,0) *+{3} ="04",
(18,2) *+{5} ="14",
(20,4) *+{6} ="24",
(22,6) *+{4} ="34",
(24,8) *+{2} ="44",
(22,5) *+{4} ="54",
(24,6) *+{2} ="64",
(20,0) *+{2} ="05",
(22,2) *+{3} ="15",
(24,4) *+{5} ="25",
(26,6) *+{3} ="35",
(28,8) *+{1} ="45",
(26,5) *+{3} ="55",
(28,6) *+{1} ="65",
(24,0) *+{1} ="06",
(26,2) *+{3} ="16",
(28,4) *+{4} ="26",
(30,6) *+{2} ="36",
(32,8) *+{1} ="46",
(30,5) *+{2} ="56",
(32,6) *+{1} ="66",
(28,0) *+{2} ="07",
(30,2) *+{3} ="17",
(32,4) *+{3} ="27",
(34,6) *+{2} ="37",
(36,8) *+{1} ="47",
(34,5) *+{2} ="57",
(36,6) *+{1} ="67",
(32,0) *+{1} ="08",
(34,2) *+{1} ="18",
(36,4) *+{2} ="28",
(38,6) *+{1} ="38",
(40,8) *+{0} ="48",
(38,5) *+{1} ="58",
(40,6) *+{0} ="68",
(36,0) *+{0} ="09",
(38,2) *+{1} ="19",
(40,4) *+{1} ="29",
(42,6) *+{0} ="39",
(44,8) *+{\bullet} ="49",
(42,5) *+{0} ="59",
(44,6) *+{\bullet} ="69",
(40,0) *+{1} ="010",
(42,2) *+{1} ="110",
(44,4) *+{\bullet} ="210",
(44,0) *+{\bullet} ="011",
"40", {\ar"31"},
"60", {\ar"51"},
"21", {\ar"31"},
"31", {\ar"41"},
"21", {\ar"51"},
"51", {\ar"61"},
"21", {\ar"12"},
"31", {\ar"22"},
"41", {\ar"32"},
"51", {\ar"22"},
"61", {\ar"52"},
"02", {\ar"12"},
"12", {\ar"22"},
"22", {\ar"32"},
"32", {\ar"42"},
"22", {\ar"52"},
"52", {\ar"62"},
"12", {\ar"03"},
"22", {\ar"13"},
"32", {\ar"23"},
"42", {\ar"33"},
"52", {\ar"23"},
"62", {\ar"53"},
"03", {\ar"13"},
"13", {\ar"23"},
"23", {\ar"33"},
"33", {\ar"43"},
"23", {\ar"53"},
"53", {\ar"63"},
"13", {\ar"04"},
"23", {\ar"14"},
"33", {\ar"24"},
"43", {\ar"34"},
"53", {\ar"24"},
"63", {\ar"54"},
"04", {\ar"14"},
"14", {\ar"24"},
"24", {\ar"34"},
"34", {\ar"44"},
"24", {\ar"54"},
"54", {\ar"64"},
"14", {\ar"05"},
"24", {\ar"15"},
"34", {\ar"25"},
"44", {\ar"35"},
"54", {\ar"25"},
"64", {\ar"55"},
"05", {\ar"15"},
"15", {\ar"25"},
"25", {\ar"35"},
"35", {\ar"45"},
"25", {\ar"55"},
"55", {\ar"65"},
"15", {\ar"06"},
"25", {\ar"16"},
"35", {\ar"26"},
"45", {\ar"36"},
"55", {\ar"26"},
"65", {\ar"56"},
"06", {\ar"16"},
"16", {\ar"26"},
"26", {\ar"36"},
"36", {\ar"46"},
"26", {\ar"56"},
"56", {\ar"66"},
"16", {\ar"07"},
"26", {\ar"17"},
"36", {\ar"27"},
"46", {\ar"37"},
"56", {\ar"27"},
"66", {\ar"57"},
"07", {\ar"17"},
"17", {\ar"27"},
"27", {\ar"37"},
"37", {\ar"47"},
"27", {\ar"57"},
"57", {\ar"67"},
"17", {\ar"08"},
"27", {\ar"18"},
"37", {\ar"28"},
"47", {\ar"38"},
"57", {\ar"28"},
"67", {\ar"58"},
"08", {\ar"18"},
"18", {\ar"28"},
"28", {\ar"38"},
"38", {\ar"48"},
"28", {\ar"58"},
"58", {\ar"68"},
"18", {\ar"09"},
"28", {\ar"19"},
"38", {\ar"29"},
"48", {\ar"39"},
"58", {\ar"29"},
"68", {\ar"59"},
"09", {\ar"19"},
"19", {\ar"29"},
"29", {\ar"39"},
"39", {\ar"49"},
"29", {\ar"59"},
"59", {\ar"69"},
"19", {\ar"010"},
"29", {\ar"110"},
"39", {\ar"210"},
"59", {\ar"210"},
"010", {\ar"110"},
"110", {\ar"210"},
"110", {\ar"011"},
\end{xy}
}
\]
Thus $\theta_iX\neq0$ for all $i$, and moreover $\mathbb{Z}\widetilde{E_6}$ is strict.
This is a consequence of the equality
\[\theta_{2i-1}X+\theta_{2i}X=c^iX\]
for each $i\ge1$, where $c$ is the Coxeter transformation.  These observations hold true for all non-Dynkin quivers.
\end{enumerate}
\end{example}

In the next section, we will use the following characterizations of strictness.

\begin{proposition}\label{strict}
Let $\mathscr{D}$ be a $\tau$-category with $\bigcap_{i\ge0}{\rm rad}^{i}\mathscr{D}=0$. Then {\rm(1)$\Leftrightarrow$(2)$\Leftrightarrow$(3)$\Leftarrow$(4)} hold for the following conditions.
\begin{enumerate}[\rm(1)]
\item $\mathscr{D}$ is a strict $\tau$-category.
\item $\theta_iX=\theta(\theta_{i-1}X)-\tau(\theta_{i-2}X)$ holds for all $i\ge2$ and $X\in\ind\mathscr{D}$.
\item $\mathrm{AR}_\tau(\mathscr{D})$ is a strict $\tau$-quiver (Definition \ref{define theta_i}).
\item For each $X\in\ind\mathscr{D}$, there exists a $\tau$-projective object $P\in\ind\mathscr{D}$ such that $\mathscr{D}(P,X)\neq0$.
\end{enumerate}
\end{proposition}

\begin{proof}
(1)$\Leftrightarrow$(2) is \cite[7.4(1)]{I1}. (2)$\Leftrightarrow$(3) is clear. (4)$\Rightarrow$(1) is dual to \cite[7.4(2)]{I1}.
\end{proof}


To state Reconstruction Theorem (cf.\ Corollary \ref{gr=mesh} below), we need the notions below introduced in \cite{IT2} as a `modulated translation quiver' and its mesh category.

\begin{definition}\cite[Definition 8.3]{I1}\cite[Definition 1.7]{IT2}\label{define tau-species}
\begin{enumerate}[\rm(1)]
\item A \emph{species} is ${\mathcal Q}=(Q_0,D(X),M(X,Y))$ consisting of the following data.
\begin{enumerate}[$\bullet$]
\item $Q_0$ is a set.
\item For $X,Y\in Q_0$, $D(X)$ is a division ring, and $M(X,Y)$ is a $(D(Y),D(X))$-bimodule. 
\end{enumerate}
It is called \emph{locally finite} if for each $X\in Q_0$,
\[\underset{Y\in Q_0}{\sum}\dim M(Y,X)_{D(Y)}<\infty\ \mbox{ and }\ \underset{Y\in Q_0}\sum\dim_{D(Y)}M(X,Y)<\infty.\]
In this case, we call $(Q_0,d,d')$ the \emph{underlying valued quiver} of ${\mathcal Q}$, where $d_{XY}:=\dim M(X,Y)_{D(X)}$ and $d'_{XY}:=\dim_{D(Y)}M(X,Y)$ for each $X,Y\in Q_0$.
\item A \emph{$\tau$-species} is ${\mathcal Q}=(Q_0,D(X),M(X,Y),\tau,a_X,b_{XY})$ consisting of the following data.
\begin{enumerate}[$\bullet$]
\item $(Q_0,D(X),M(X,Y))$ is a species.
\item $\tau\colon Q_0\setminus Q_0^p\to Q_0\setminus Q_0^i$ is a bijection for subsets $Q_0^p$ and $Q_0^i$ of $Q_0$.
\item For $X\in Q_0\setminus Q_0^p$, $a_X\colon D(X)\simeq D(\tau X)$ is an isomorphism of rings.
\item For $X\in Q_0\setminus Q_0^p$ and $Y\in Q_0$, $b_{XY}\colon{\rm Hom}_{D(Y)}(M(\tau X,Y),D(Y))\simeq M(Y,X)$ is an isomorphism of $(D(X),D(Y))$-bimodules, where $M(\tau X,Y)$ is regarded as a $D(X)$-module via $a_X$.
\end{enumerate}
It is called \emph{locally finite} if the species $(Q_0,D(X),M(X,Y))$ is locally finite. In this case, we call $(Q_0,d,d',\tau)$ the \emph{underlying $\tau$-quiver} of ${\mathcal Q}$, where $(Q_0,d,d')$ is the underlying valued quiver of the species $(Q_0,D(X),M(X,Y))$.
\end{enumerate}
\end{definition}

We need the following elementary observation.

\begin{lemma}\label{lift quiver to species}\cite[4.2.1]{I3}
For each symmetrizable $\tau$-quiver $Q$, there exists a $\tau$-species $\mathcal{Q}$ whose underlying $\tau$-quiver is $Q$.
\end{lemma}


The Auslander--Reiten quiver of a $\tau$-category gives rise to a $\tau$-species.

\begin{example}\cite[Definition 9.1]{I1}\cite[Definition 2.4]{IT2}\label{define AR species}
Let $\mathscr{D}$ be a $\tau$-category. Then the \emph{Auslander--Reiten species} $\mathrm{AR}_\tau^{\rm sp}(\mathscr{D})$ of $\mathscr{D}$ defined as follows is a $\tau$-species.
\begin{enumerate}[$\bullet$]
\item $Q_0$, $Q_0^p$, $Q_0^i$ and $\tau$ are given in Definition \ref{define tau category}.
\item For $X,Y\in Q_0$, $D(X):=({\mathscr D}/{\rm rad}{\mathscr D})(X,X)$ and $M(X,Y):=({\rm rad}{\mathscr D}/{\rm rad}^2{\mathscr D})(X,Y)$.
\item For $X\in Q_0\setminus Q_0^p$, take a $\tau$-sequence $\tau X\to\theta X\to X$. Then $a_X\colon D(X)\to D(\tau X)$ is a map sending the class of $f\in\mathrm{End}_{\mathscr{D}}(X)$ to the class of $f^{\prime\prime}\in\mathrm{End}_{\mathscr{D}}(\tau X)$ making the following diagram commutative.
\[\xymatrix{
\tau X\ar[r]\ar[d]^{f^{\prime\prime}}&\theta X\ar[r]\ar[d]^{f^{\prime}}&X\ar[d]^f\\
\tau X\ar[r]&\theta X\ar[r]&X
}\]
\item For $Y\in Q_0$, $b_{XY}$ is the composition
\begin{eqnarray*}
{\rm Hom}_{D(Y)}(M(\tau X,Y),D(Y))&\simeq&{\rm Hom}_{D(Y)}(({\mathscr D}/{\rm rad}{\mathscr D})(\theta X,Y),D(Y))
\simeq({\mathscr D}/{\rm rad}{\mathscr D})(Y,\theta X)\\
&\simeq&M(Y,X),
\end{eqnarray*}
where the first and the third isomorphisms are given by Proposition~\ref{AR describe sink} and the second one is induced by the composition
$({\mathscr D}/{\rm rad}{\mathscr D})(Y,\theta X)\times({\mathscr D}/{\rm rad}{\mathscr D})(\theta X,Y)\to D(Y)$.
\end{enumerate}
\end{example}

Conversely, each $\tau$-species gives rise to a $\tau$-category as follows.

\begin{definition}\cite[Definition 8.1,\ 8.3.1]{I1}\label{define mesh category}
\begin{enumerate}[\rm(1)]
\item Let ${\mathcal Q}=(Q_0,D(X),M(X,Y))$ be a locally finite species. The \emph{complete path category} $\mathsf{P}({\mathcal Q})$ of ${\mathcal Q}$ is a Krull--Schmidt category with $\ind\mathsf{P}({\mathcal Q})=Q_0$ defined as follows: Fix $X,Y\in Q_0$, let $\displaystyle \mathsf{P}(X,Y):=\prod_{i\ge0}\mathsf{P}_i(X,Y)$, where $\mathsf{P}_0(X,Y):=D_X$ if $X=Y$ and $0$ otherwise, and
\[\mathsf{P}_i(X,Y):=\bigoplus_{Z_1,\ldots,Z_{i-1}\in Q_0}M(Z_{i-1},Y)\otimes_{D(Z_{i-1})}\cdots\otimes_{D(Z_1)}M(X,Z_1)\ \mbox{ for }i\ge1.\]
\item Let ${\mathcal Q}=(Q_0,D(X),M(X,Y),\tau,a_X,b_{XY})$ be a locally finite $\tau$-species. The \emph{complete mesh category} $\mathsf{M}({\mathcal Q})$ of ${\mathcal Q}$ is defined as follows.
\begin{enumerate}[$\bullet$]
\item Let $\mathsf{P}({\mathcal Q})$ be the complete path category of the species $(Q_0,D(X),M(X,Y))$.
\item For $X\in Q_0\setminus Q_0^p$, the \emph{mesh relation} is the element $\gamma_X\in M(Y,X)\otimes_{D(Y)}M(\tau X,Y)\simeq{\rm End}_{D(Y)}(M(\tau X,Y))$ corresponding to $1_{M(\tau X,Y)}$, where the isomorphism is given by $b_{XY}$.
\item $\mathsf{M}({\mathcal Q})$ is the quotient of $\mathsf{P}({\mathcal Q})$ by the closure of the ideal generated by all mesh relations with respect to the ${\rm rad}\mathsf{P}({\mathcal Q})$-adic topology.
\end{enumerate}
\end{enumerate}
\end{definition}

We have the following important observation.

\begin{proposition}\label{property of mesh category}\cite[Proposition 8.4, Theorem 10.2]{I1}
Let ${\mathcal Q}$ be a locally finite $\tau$-species.
Then $\mathscr{D}:=\mathsf{M}({\mathcal Q})$ is a $\tau$-category satisfying $\bigcap_{i\ge0}{\rm rad}^i\mathscr{D}=0$ and $\mathrm{AR}_\tau^{\rm sp}(\mathscr{D})=\mathcal{Q}$.
\end{proposition}


\subsection{Stable categories of extriangulated categories}
We apply results on $\tau$-categories to study the structure of stable categories $\underline{\mathscr{C}}$ of extriangulated categories $\mathscr{C}$ as additive categories. 
The following result shows that almost split sequences in $\mathscr{C}$ give sink (resp.\ source) resolutions in $\underline{\mathscr{C}}$ (resp.\ $\overline{\mathscr{C}}$) (cf.\ Proposition~\ref{ass gives sink sequence}). 

\begin{proposition}
Let $\mathscr{C}$ be a Krull--Schmidt extriangulated category and $A\to B\to C$ an almost split sequence in $\mathscr{C}$.
\begin{enumerate}[\rm(1)]
\item If $\mathscr{C}$ has enough projectives, then a sink resolution of $C$ in $\underline{\mathscr{C}}$ is given by a direct summand of
$\cdots\to\Omega^2C\to\Omega A\to\Omega B\to\Omega C\to A\to B\to C$.
\item If $\mathscr{C}$ has enough injectives, then a source resolution of $A$ in $\overline{\mathscr{C}}$ is given by a direct summand of
$A\to B\to C\to\Sigma A\to\Sigma B\to\Sigma C\to\Sigma^2A\to\cdots$.
\end{enumerate}
\end{proposition}

\begin{proof}
The assertions follows from the exact sequences
given in Theorem~\ref{long exact sequence}.
\end{proof}

Motivated by Example~\ref{example of tau-category}(5), we prove the following much more general result. We refer to Definition~\ref{have ASE} and Remark~\ref{C/P and C/P} for necessary definitions and to Proposition~\ref{Proposition6.3} for a more detailed result.

\begin{theorem}\label{stable categories are tau}
Let $\mathscr{C}$ be a Krull--Schmidt extriangulated category with almost split extensions.
\begin{enumerate}[\rm(1)]
\item If $\mathscr{C}$ has enough injectives and sink morphisms, then $\overline{\mathscr{C}}$ is a $\tau$-category.
\item If $\mathscr{C}$ has enough projectives and source morphisms, then $\underline{\mathscr{C}}$ is a $\tau$-category.
\end{enumerate}
Moreover $\mathrm{AR}_{\rm ET}(\mathscr{C})$ determines $\mathrm{AR}_\tau(\overline{\mathscr{C}})$ in {\rm(1)} (resp. $\mathrm{AR}_\tau(\underline{\mathscr{C}})$ in {\rm(2)}).
\end{theorem}

Before proving Theorem~\ref{stable categories are tau}, we give some examples.

\begin{example}
\begin{enumerate}[\rm(1)]
\item We consider the exact category ${}^\perp U$ in Example~\ref{An example}. One can check that ${}^\perp U$ is a $\tau$-category if and only if $\ell=n-1$ or $n-2$, while $\underline{{}^\perp U}$ and $\overline{{}^\perp U}$ are $\tau$-categories for each $\ell$ by Theorem~\ref{stable categories are tau}.
For example, $\mathrm{AR}_\tau(\underline{{}^\perp U})$ and $\mathrm{AR}_\tau(\overline{{}^\perp U})$ for $n=7$ and $\ell=4$ are the following.
\[\begin{xy}
0;<5pt,0pt>:<0pt,4pt>::
(10,0)*{{\begin{smallmatrix}2\end{smallmatrix}}}="3",
(15,5)*{{\begin{smallmatrix}3\\ 2\end{smallmatrix}}}="5",
(20,0)*{{\begin{smallmatrix}3\end{smallmatrix}}}="7",
(20,10)*{{\begin{smallmatrix}4\\ 3\\ 2\end{smallmatrix}}}="8",
(25,5)*{{\begin{smallmatrix}4\\ 3\end{smallmatrix}}}="9",
(30,0)*{{\begin{smallmatrix}4\end{smallmatrix}}}="11",
(30,10)*{{\begin{smallmatrix}I_3\\ 7\\ 6\\ 5\\ 4\\ 3\end{smallmatrix}}}="12",
(35,5)*{{\begin{smallmatrix}I_4\\ 7\\ 6\\ 5\\ 4\end{smallmatrix}}}="13",
(40,0)*{{\begin{smallmatrix}I_5\\ 7\\ 6\\ 5\end{smallmatrix}}}="14",
\ar"3";"5",
\ar"5";"7",
\ar"5";"8",
\ar"7";"9",
\ar"8";"9",
\ar"9";"11",
\ar"9";"12",
\ar"11";"13",
\ar"12";"13",
\ar"13";"14",
\ar@{.>}"7";"3",
\ar@{.>}"11";"7",
\ar@{.>}"14";"11",
\ar@{.>}"9";"5",
\ar@{.>}"13";"9",
\ar@{.>}"12";"8",
\end{xy}\ \ \ \begin{xy}
0;<5pt,0pt>:<0pt,4pt>::
(0,0)*{{\begin{smallmatrix}P_1\\ 1\end{smallmatrix}}}="1",
(5,5)*{{\begin{smallmatrix}P_2\\ 2\\ 1\end{smallmatrix}}}="2",
(10,0)*{{\begin{smallmatrix}2\end{smallmatrix}}}="3",
(10,10)*{{\begin{smallmatrix}P_3\\ 3\\ 2\\ 1\end{smallmatrix}}}="4",
(15,5)*{{\begin{smallmatrix}3\\ 2\end{smallmatrix}}}="5",
(20,0)*{{\begin{smallmatrix}3\end{smallmatrix}}}="7",
(20,10)*{{\begin{smallmatrix}4\\ 3\\ 2\end{smallmatrix}}}="8",
(25,5)*{{\begin{smallmatrix}4\\ 3\end{smallmatrix}}}="9",
(30,0)*{{\begin{smallmatrix}4\end{smallmatrix}}}="11",
\ar"1";"2",
\ar"2";"3",
\ar"2";"4",
\ar"3";"5",
\ar"4";"5",
\ar"5";"7",
\ar"5";"8",
\ar"7";"9",
\ar"8";"9",
\ar"9";"11",
\ar@{.>}"3";"1",
\ar@{.>}"7";"3",
\ar@{.>}"11";"7",
\ar@{.>}"5";"2",
\ar@{.>}"9";"5",
\ar@{.>}"8";"4",
\end{xy}\]
\item We consider the exact category $\mathsf{CM} R$ in Example~\ref{A1 example}. This is not a $\tau$-category by Example~\ref{example of tau-category}(5), while $\underline{\mathsf{CM}} R=\overline{\mathsf{CM}} R$ is a $\tau$-category by Theorem~\ref{stable categories are tau}.
Moreover, $\mathrm{AR}_\tau(\underline{\mathsf{CM}} R)$ is the following.
\[\xymatrix@R1em@C2em{
(x,u)&&(x,v)\ar@{.>}[ll]&&(x,u)\ar@{.>}[ll]}\]
\end{enumerate}
\end{example}

To prove Theorem~\ref{stable categories are tau}, we need the following modification of \cite[1.4(2)]{I2}.

\begin{lemma}\label{Lemma6.1}
Let $\mathscr{C}$ be a Krull--Schmidt extriangulated category and $\mathscr{D}$ an additive full subcategory of $\mathscr{C}$.
\begin{enumerate}[\rm(1)]
\item If $x\in\mathscr{C}(A,B)$ is a right (resp.\ left) almost split morphism in $\mathscr{C}$ and if $B\notin\mathscr{D}$ (resp. $A\notin\mathscr{D}$), then $\overline{x}\in(\mathscr{C}/\mathscr{D})(A,B)$ is a right (resp.\ left) almost split morphism in $\mathscr{C}/\mathscr{D}$.
\item Let $A\overset{x}{\longrightarrow}B\overset{y}{\longrightarrow}C$ be an almost split sequence such that $C\notin\mathscr{D}$. Then a right $\tau$-sequence of $C$ in $\mathscr{C}/\mathscr{D}$ is given by $A\overset{\overline{x}}{\longrightarrow}B\overset{\overline{y}}{\longrightarrow}C$ if $B\notin\mathscr{D}$, and by $0\overset{}{\longrightarrow}0\overset{}{\longrightarrow}C$ if $B\in\mathscr{D}$.
\item Let $A\overset{x}{\longrightarrow}B\overset{y}{\longrightarrow}C$ be an almost split sequence such that $A\notin\mathscr{D}$. Then a left $\tau$-sequence of $A$ in $\mathscr{C}/\mathscr{D}$ is given by $A\overset{\overline{x}}{\longrightarrow}B\overset{\overline{y}}{\longrightarrow}C$ if $B\notin\mathscr{D}$, and by $A\overset{}{\longrightarrow}0\overset{}{\longrightarrow}0$ if $B\in\mathscr{D}$.
\end{enumerate}
\end{lemma}

\begin{proof}
(1) This is clear since ${\rm rad}(\mathscr{C}/\mathscr{D})(X,B)=(({\rm rad}\mathscr{C})/[\mathscr{D}])(X,B)$ holds for any $X\in\mathscr{C}$.

(2) $\overline{y}$ is right almost split by (1), and $\overline{x}$ is a weak kernel of $\overline{y}$ by Lemma~\ref{tensor}. Since $A\in\ind\mathscr{C}$, a sink sequence of $C$ in $\mathscr{C}/\mathscr{D}$ is given by  $A\overset{\overline{x}}{\longrightarrow}B\overset{\overline{y}}{\longrightarrow}C$ if $B\notin\mathscr{D}$, and $0\overset{}{\longrightarrow}0\overset{}{\longrightarrow}C$ if $B\in\mathscr{D}$. Clearly the latter sequence gives a right $\tau$-sequence of $C$ in $\mathscr{C}/\mathscr{D}$, and so is the former one since the dual argument shows that either $\overline{x}$ is left almost split or $A\in\mathscr{D}$ holds.

(3) is dual to (2).
\end{proof}

Another key observation is the following.

\begin{proposition}\label{Proposition6.2}
Let $\mathscr{C}$ be an extriangulated category with enough injectives.
Let $P\in\mathscr{C}$ be an indecomposable projective non-injective object, and $f\in\mathscr{C}(A,P)$ a sink morphism.
Then $0\to A\xrightarrow{\overline{f}}P$ is a right $\tau$-sequence in $\overline{\mathscr{C}}$.
\end{proposition}

\begin{proof}
The proof is parallel to that of Theorem~\ref{weak kernel of right almost split}. It suffices to show that $\overline{f}$ is a monomorphism in $\overline{\mathscr{C}}$.
We will show that any $g\in\mathscr{C}(B,A)$ satisfying $\overline{fg}=0$ satisfies $\overline{g}=0$. Take an $\mathfrak{s}$-conflation $B\overset{i}{\longrightarrow}I\overset{j}{\longrightarrow}X$ with $I$ injective. By \cite[Proposition 1.20]{LNa}, we have the following commutative diagram such that $B \xrightarrow{\left[\bsm i\\-g \esm\right]}I\oplus A\xrightarrow{[h\ x^{\prime}]} Z$ is an $\mathfrak{s}$-conflation.
\[
\xymatrix@R1.5em@C4em{
B\ar[r]^i\ar[d]^g&I\ar[r]^j\ar[d]^h&X\ar@{=}[d]\\
A\ar[r]^{x^{\prime}}&Z\ar[r]^{y^{\prime}}&X
}
\]
By the assumption of $\overline{fg}=0$, there exists $a\in\mathscr{C}(I,P)$ such that $fg=ai$. Thus there exists $b\in\mathscr{C}(Z,P)$ such that $b [h\ x^{\prime}]=[a\ f]$.

If $b$ is a retraction, then $[a\ f]=b [h\ x^{\prime}]\in\mathscr{C}(I\oplus A,P)$ is a deflation, and hence a retraction since $P$ is projective. This is a contradiction since both $a$ and $f$ belong to ${\rm rad}\mathscr{C}$ by non-injectivity of $P$.
Therefore $b$ is not a retraction. Since $f$ is a sink morphism, there is $c\in\mathscr{C}(Z,A)$ satisfying $b = fc$. Since $f = bx^{\prime} = fcx^{\prime}$ holds and $f$ is right minimal, $cx^{\prime}$ is an isomorphism in $\mathscr{C}$.
Thus there is a left inverse $z\in\mathscr{C}(Z,A)$ of $x^{\prime}$. This gives $g=z h i$, hence $\overline{g}=0$ as desired.
\end{proof}

Now we are able to prove the following main observations.

\begin{proposition}\label{Proposition6.3}
Let $\mathscr{C}$ be a Krull--Schmidt extriangulated category.
\begin{enumerate}[\rm(1)]
\item If $\mathscr{C}$ has enough injectives and left almost split extensions, then $\overline{\mathscr{C}}$ is a left $\tau$-category.
\item If $\mathscr{C}$ has enough injectives and sink morphisms, then
$\overline{\mathscr{C}}$ is a right $\tau$-category.
\item If $\mathscr{C}$ has enough projectives and right almost split extensions, then $\underline{\mathscr{C}}$ is a right $\tau$-category.
\item If $\mathscr{C}$ has enough projectives and source morphisms, then
$\underline{\mathscr{C}}$ is a left $\tau$-category.
\end{enumerate}
\end{proposition}

\begin{proof}
(1) Since $\mathscr{C}$ has enough injectives, we have $\overline{\mathscr{C}}=\mathscr{C}/\mathscr{I}$ for $\mathscr{I}:=\mathrm{Inj}_{\mathbb{E}}\mathscr{C}$ by Remark \ref{C/P and C/P}. Fix any $A\in\ind\overline{\mathscr{C}}=\ind\mathscr{C}\setminus\ind\mathscr{I}$. 
Since $\mathscr{C}$ has left almost split extensions, there exists an almost split sequence $A\overset{x}{\longrightarrow}B\overset{y}{\longrightarrow}C$. By Lemma \ref{Lemma6.1}(3), it gives a left $\tau$-sequence of $A$ in $\overline{\mathscr{C}}$.

(2) Fix any $C\in\ind\overline{\mathscr{C}}=\ind\mathscr{C}\setminus\ind\mathscr{I}$. If $C$ is projective, then a sink morphism $B\overset{y}{\longrightarrow}C$ gives a right $\tau$-sequence $0\overset{}{\longrightarrow}B\overset{\overline{y}}{\longrightarrow}C$ by Proposition \ref{Proposition6.2}.
If $C$ is non-projective, then Lemma \ref{sink imply right ASE} shows that there exists an almost split sequence $A\overset{x}{\longrightarrow}B\overset{y}{\longrightarrow}C$. By Lemma \ref{Lemma6.1}(2), it gives a right $\tau$-sequence of $C$ in $\overline{\mathscr{C}}$.

(3) and (4) are dual to (1) and (2) respectively.
\end{proof}

We are ready to prove Theorem~\ref{stable categories are tau}.

\begin{proof}[Proof of Theorem~\ref{stable categories are tau}]
The assertions are immediate from Proposition~\ref{Proposition6.3}.
\end{proof}

Theorem~\ref{stable categories are tau} enables us to apply results on $\tau$-categories to (co)stable categories of extriangulated categories. In the rest of this subsection, we give some applications.

\begin{corollary}[Radical Layers Theorem]\label{radical layer}
Let $\mathscr{C}$ be a Krull--Schmidt extriangulated category which has enough projectives, almost split extensions and source morphisms, and let $\mathscr{D}=\underline{\mathscr{C}}$.
\begin{enumerate}[\rm(1)]
\item Let $A\overset{x}{\longrightarrow}B\overset{y}{\longrightarrow}C$ be an almost split sequence in $\mathscr{C}$ such that $A$ is non-projective. Then the following sequences are exact for each $i\ge0$, where ${\rm rad}^{-1}\mathscr{D}:=\mathscr{D}$.
\begin{eqnarray*}
&{\rm rad}^{i-1}\mathscr{D}(-,A)\xrightarrow{x\circ-}{\rm rad}^i\mathscr{D}(-,B)\xrightarrow{y\circ-}{\rm rad}^{i+1}\mathscr{D}(-,C)\to0,&\\
&{\rm rad}^{i-1}\mathscr{D}(C,-)\xrightarrow{-\circ y}{\rm rad}^i\mathscr{D}(B,-)\xrightarrow{-\circ x}{\rm rad}^{i+1}\mathscr{D}(A,-)\to0.&
\end{eqnarray*}
\item Let $I\xrightarrow{x}A$ be a source morphism in $\mathscr{C}$ of an indecomposable injective non-projective object~$I$. Then we have an isomorphism of functors for each $i\ge0$.
\[ -\circ x\colon{\rm rad}^i\mathscr{D}(A,-)\xrightarrow{\sim}{\rm rad}^{i+1}\mathscr{D}(I,-).\]
\end{enumerate}
\end{corollary}

\begin{proof}
(1) 
By Theorem \ref{existence theorem}, the morphism $x\in\mathscr{C}(A,B)$ has a right ladder. We apply \cite[Theorem 4.2]{I1} to $a_0:=x$ and $n:=0$. Since $L={\rm rad}\mathscr{D}(-,C)$, we obtain the first exact sequence
\[{\rm rad}^{i-1}\mathscr{D}(-,A)\xrightarrow{x\circ-}{\rm rad}^i\mathscr{D}(-,B)\xrightarrow{y\circ-}{\rm rad}^{i+1}\mathscr{D}(-,C)\to0.\]
Dually, we obtain the second exact sequence
\[{\rm rad}^{i-1}\mathscr{D}(C,-)\xrightarrow{-\circ y}{\rm rad}^i\mathscr{D}(B,-)\xrightarrow{-\circ x}{\rm rad}^{i+1}\mathscr{D}(A,-)\to0.\]

(2) By the dual of Proposition \ref{Proposition6.2}, we have an isomorphism $-\circ x\colon\mathscr{D}(A,-)\xrightarrow{\sim}{\rm rad}\mathscr{D}(I,-)$. Taking ${\rm rad}^i$ of these functors, we obtain the desired isomorphism.
\end{proof}

We leave it to the reader to state the dual results for the costable category $\overline{\mathscr{C}}$. From Corollary~\ref{radical layer}, we obtain the following exact sequences and an isomorphism, which explain the term ``Radical Layers Theorem'':
\begin{eqnarray}\label{radical layer2}
&&({\rm rad}^{i-1}\mathscr{D}/{\rm rad}^i\mathscr{D})(-,A)\xrightarrow{x\circ -}({\rm rad}^i\mathscr{D}/{\rm rad}^{i+1}\mathscr{D})(-,B)\xrightarrow{y\circ-}({\rm rad}^{i+1}\mathscr{D}/{\rm rad}^{i+2}\mathscr{D})(-,C)\to0,\nonumber \\
&&({\rm rad}^{i-1}\mathscr{D}/{\rm rad}^i\mathscr{D})(C,-)\xrightarrow{-\circ y}({\rm rad}^i\mathscr{D}/{\rm rad}^{i+1}\mathscr{D})(B,-)\xrightarrow{-\circ x}({\rm rad}^{i+1}\mathscr{D}/{\rm rad}^{i+2}\mathscr{D})(A,-)\to0,\nonumber\\
&&-\circ x\colon({\rm rad}^{i}\mathscr{D}/{\rm rad}^{i+1}\mathscr{D})(A,-)\xrightarrow{\sim}({\rm rad}^{i+1}\mathscr{D}/{\rm rad}^{i+2}\mathscr{D})(I,-).\nonumber
\end{eqnarray}



Let $\mathscr{C}$ be a Krull--Schmidt extriangulated category which has enough projectives, almost split extensions and source morphisms, and let $\mathscr{D}=\underline{\mathscr{C}}$.
Then $\mathscr{D}$ is a $\tau$-category and $\mathrm{AR}_\tau(\mathscr{D})$ is determined by $\mathrm{AR}_{\rm ET}(\mathscr{C})$ (Theorem \ref{stable categories are tau}). On the other hand,
the functions $(\theta_i)_{i\ge0}$ for $\mathscr{D}$ given in Definition~\ref{define theta_i} are determined by $\mathrm{AR}_\tau(\mathscr{D})$.
Therefore the next result enables us to calculate the dimensions of $\mathscr{D}(X,Y)$ and $\mathbb{E}(X,Y)$ from $\mathrm{AR}_\tau(\mathscr{D})$. An example will be given in Example~\ref{calculate dimension example}. We denote by
\[\langle-,-\rangle:K_0(\mathscr{D})\times K_0(\mathscr{D})\to\mathbb{Z}\]
the bilinear form such that $\langle X,Y\rangle$ is the Kronecker delta $\delta_{XY}$ for $X,Y\in\ind\mathscr{D}$.

\begin{corollary}[Auslander--Reiten Combinatorics]\label{calculate dimension}
Let $\mathscr{C}$ be a Krull--Schmidt extriangulated category which has enough projectives (resp.\ injectives), almost split extensions and source (resp.\ sink) morphisms, and let $\mathscr{D}=\underline{\mathscr{C}}$ (resp.\ $\overline{\mathscr{C}}$).
\begin{enumerate}[\rm(1)]
\item For each object $X\in\mathscr{D}$ and $i\ge0$, there is an isomorphism of functors
\[(\mathscr{D}/{\rm rad}\mathscr{D})(-,\theta_iX)\simeq({\rm rad}^{i}\mathscr{D}/{\rm rad}^{i+1}\mathscr{D})(-,X).\]
\item (Sign-coherence) $\theta_i$ is a monoid endomorphism of $K_0(\mathscr{D})_+$, that is, for each $X,Y\in\mathscr{D}$, we have $\theta_i(X\oplus Y)\simeq\theta_iX\oplus\theta_iY$.
\item Assume $\bigcap_{i\ge0}{\rm rad}^{i}\mathscr{D}=0$. Then for each $X\in\mathscr{D}$ and $Y\in\ind\mathscr{D}$, we have
\[ {\rm length}\mathscr{D}(Y,X)_{{\rm End}_{\mathscr{D}}(Y)}=\sum_{i\ge0}\langle Y,\theta_iX\rangle.\] 
This is equal to ${\rm length}_{{\rm End}_{\mathscr{C}}(\tau Y)}\mathbb{E}(X,\tau Y)$ if $\mathscr{C}$ has Auslander--Reiten--Serre duality. 
\end{enumerate}
\end{corollary}

\begin{proof}
(1) Let $a_0:=0\in\mathscr{D}(0,X)$. By Theorem \ref{existence theorem}, $a_0$ has a right ladder \eqref{ladder diagram}. By \cite[Theorem 4.1]{I1}, we have an exact sequence of functors
\[\mathscr{D}(-,X_i)\xrightarrow{a_i\circ-}\mathscr{D}(-,Y_i)\to{\rm rad}^{i}\mathscr{D}(-,X)\to0.\]
Thus there is an isomorphism of functors
\[(\mathscr{D}/{\rm rad}\mathscr{D})(-,Y_i)\simeq({\rm rad}^{i}\mathscr{D}/{\rm rad}^{i+1}\mathscr{D})(-,X).\]
By \cite[Theorem 7.1]{I1}, $\theta_iX=Y_i$ holds for each $i\ge0$.
Thus the assertion follows.

(2) The assertion follows from isomorphisms of functors
\begin{align*}
&(\mathscr{D}/{\rm rad}\mathscr{D})(-,\theta_i(X\oplus Y))\stackrel{{\rm(1)}}{\simeq}({\rm rad}^{i}\mathscr{D}/{\rm rad}^{i+1}\mathscr{D})(-,X\oplus Y)\\
\simeq&({\rm rad}^{i}\mathscr{D}/{\rm rad}^{i+1}\mathscr{D})(-,X)\oplus({\rm rad}^{i}\mathscr{D}/{\rm rad}^{i+1}\mathscr{D})(-,Y)\\
\stackrel{{\rm(1)}}{\simeq}&(\mathscr{D}/{\rm rad}\mathscr{D})(-,\theta_iX)\oplus(\mathscr{D}/{\rm rad}\mathscr{D})(-,\theta_iY)\simeq(\mathscr{D}/{\rm rad}\mathscr{D})(-,\theta_iX\oplus\theta_iY).
\end{align*}

(3) The assertion follows from
\begin{align*}
{\rm length}\mathscr{D}(Y,X)_{{\rm End}_{\mathscr{D}}(Y)}=&\sum_{i\ge0}{\rm length}({\rm rad}^{i}\mathscr{D}/{\rm rad}^{i+1}\mathscr{D})(Y,X)_{{\rm End}_{\mathscr{D}}(Y)}\\
\stackrel{{\rm(1)}}{=}&\sum_{i\ge0}{\rm length}(\mathscr{D}/{\rm rad}\mathscr{D})(Y,\theta_iX)_{{\rm End}_{\mathscr{D}}(Y)}=\sum_{i\ge0}\langle Y,\theta_iX\rangle.\qedhere
\end{align*}
\end{proof}

\begin{example}\label{calculate dimension example}
Using the extriangulated category $(\mathscr{C},\mathbb{F},\mathfrak{t})$ with the following Auslander--Reiten quiver given at the end of Section \ref{section_ARsubcat}, we explain Corollary~\ref{calculate dimension}.
\[
\begin{tikzpicture}[scale=0.3, fl/.style={->,>=latex}]
\foreach \x in {-1,0,1,2} {
   \draw[fl] (3.6*\x+1.8,0.5) -- (3.6*\x+2.8,1.5) ;
   \draw[fl] (3.6*\x,-1.5) -- (3.6*\x+1,-0.5) ;
};
\foreach \x in {0,1,2} {
   \draw[fl] (3.6*\x,1.5) -- (3.6*\x+1,0.5) ;
   \draw[fl] (3.6*\x-1.8,-0.5) -- (3.6*\x-0.8,-1.5) ;
   \draw[fl, dotted] (3.6*\x+2.4,2) -- (3.6*\x+0.4,2) ;
   \draw[fl, dotted] (3.6*\x+0.6,0) -- (3.6*\x-1.4,0) ;
};
\foreach \x in {0,1} {
   \draw[fl, dotted] (3.6*\x-1.2,-2) -- (3.6*\x-3.2,-2) ;
};
\draw (-0.4,2) node {$\bullet$} ;
\draw (3.2,2) node {$\bullet$} ;
\draw (6.8,2) node {$\bullet$} ;
\draw (10.4,2) node {$\bullet$} ;
\draw (-2.2,0) node {$\bullet$} ;
\draw (1.4,0) node {$\bullet$} ;
\draw (5,0) node {$\bullet$} ;
\draw (8.6,0) node {$\bullet$} ;
\draw (-4,-2) node {$\bullet$} ;
\draw (-0.4,-2) node {$\bullet$} ;
\draw (3.2,-2) node {$\bullet$} ;
\draw (6.8,-2) node {$\bullet$} ;
\end{tikzpicture}
\]
Then $\mathrm{AR}_\tau(\underline{\mathscr{C}})$ is given by the following.
\[
\begin{tikzpicture}[scale=0.3, fl/.style={->,>=latex}]
\foreach \x in {0,1,2} {
   \draw[fl] (3.6*\x+1.8,0.5) -- (3.6*\x+2.8,1.5) ;
};
\foreach \x in {0,1} {
   \draw[fl] (3.6*\x,-1.5) -- (3.6*\x+1,-0.5) ;
};
\foreach \x in {1,2} {
   \draw[fl] (3.6*\x,1.5) -- (3.6*\x+1,0.5) ;
   \draw[fl, dotted] (3.6*\x+2.4,2) -- (3.6*\x+0.4,2) ;
   \draw[fl, dotted] (3.6*\x+0.6,0) -- (3.6*\x-1.4,0) ;
};
\foreach \x in {1} {
   \draw[fl] (3.6*\x-1.8,-0.5) -- (3.6*\x-0.8,-1.5) ;
   \draw[fl, dotted] (3.6*\x-1.2,-2) -- (3.6*\x-3.2,-2) ;
};
\draw (3.2,2) node {$\bullet$} ;
\draw (6.8,2) node {$\bullet$} ;
\draw (10.4,2) node {$\bullet$} ;
%
\draw (1.4,0) node {$\bullet$} ;
\draw (5,0) node {$X$} ;
\draw (8.6,0) node {$\bullet$} ;
%
\draw (-0.4,-2) node {$\bullet$} ;
\draw (3.2,-2) node {$\bullet$} ;
\end{tikzpicture}
\]
For $X$ in the previous diagram, $\theta_iX$ is given as follows:
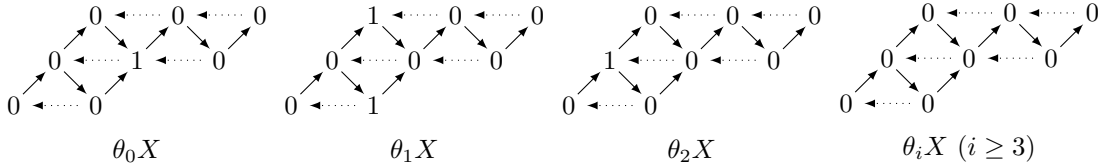
\[
\begin{tikzpicture}[scale=0.3, fl/.style={->,>=latex}]
\foreach \x in {0,1,2} {
   \draw[fl] (3.6*\x+1.8,0.5) -- (3.6*\x+2.8,1.5) ;
};
\foreach \x in {0,1} {
   \draw[fl] (3.6*\x,-1.5) -- (3.6*\x+1,-0.5) ;
};
\foreach \x in {1,2} {
   \draw[fl] (3.6*\x,1.5) -- (3.6*\x+1,0.5) ;
   \draw[fl, dotted] (3.6*\x+2.4,2) -- (3.6*\x+0.4,2) ;
   \draw[fl, dotted] (3.6*\x+0.6,0) -- (3.6*\x-1.4,0) ;
};
\foreach \x in {1} {
   \draw[fl] (3.6*\x-1.8,-0.5) -- (3.6*\x-0.8,-1.5) ;
   \draw[fl, dotted] (3.6*\x-1.2,-2) -- (3.6*\x-3.2,-2) ;
};
\draw (3.2,2) node {$0$} ;
\draw (6.8,2) node {$0$} ;
\draw (10.4,2) node {$0$} ;
%
\draw (1.4,0) node {$0$} ;
\draw (5,0) node {$1$} ;
\draw (8.6,0) node {$0$} ;
%
\draw (-0.4,-2) node {$0$} ;
\draw (3.2,-2) node {$0$} ;
\draw (5,-4) node {$\theta_0X$} ;
\end{tikzpicture}
\begin{tikzpicture}[scale=0.3, fl/.style={->,>=latex}]
\foreach \x in {0,1,2} {
   \draw[fl] (3.6*\x+1.8,0.5) -- (3.6*\x+2.8,1.5) ;
};
\foreach \x in {0,1} {
   \draw[fl] (3.6*\x,-1.5) -- (3.6*\x+1,-0.5) ;
};
\foreach \x in {1,2} {
   \draw[fl] (3.6*\x,1.5) -- (3.6*\x+1,0.5) ;
   \draw[fl, dotted] (3.6*\x+2.4,2) -- (3.6*\x+0.4,2) ;
   \draw[fl, dotted] (3.6*\x+0.6,0) -- (3.6*\x-1.4,0) ;
};
\foreach \x in {1} {
   \draw[fl] (3.6*\x-1.8,-0.5) -- (3.6*\x-0.8,-1.5) ;
   \draw[fl, dotted] (3.6*\x-1.2,-2) -- (3.6*\x-3.2,-2) ;
};
\draw (3.2,2) node {$1$} ;
\draw (6.8,2) node {$0$} ;
\draw (10.4,2) node {$0$} ;
%
\draw (1.4,0) node {$0$} ;
\draw (5,0) node {$0$} ;
\draw (8.6,0) node {$0$} ;
%
\draw (-0.4,-2) node {$0$} ;
\draw (3.2,-2) node {$1$} ;
\draw (5,-4) node {$\theta_1X$} ;
\end{tikzpicture}
\begin{tikzpicture}[scale=0.3, fl/.style={->,>=latex}]
\foreach \x in {0,1,2} {
   \draw[fl] (3.6*\x+1.8,0.5) -- (3.6*\x+2.8,1.5) ;
};
\foreach \x in {0,1} {
   \draw[fl] (3.6*\x,-1.5) -- (3.6*\x+1,-0.5) ;
};
\foreach \x in {1,2} {
   \draw[fl] (3.6*\x,1.5) -- (3.6*\x+1,0.5) ;
   \draw[fl, dotted] (3.6*\x+2.4,2) -- (3.6*\x+0.4,2) ;
   \draw[fl, dotted] (3.6*\x+0.6,0) -- (3.6*\x-1.4,0) ;
};
\foreach \x in {1} {
   \draw[fl] (3.6*\x-1.8,-0.5) -- (3.6*\x-0.8,-1.5) ;
   \draw[fl, dotted] (3.6*\x-1.2,-2) -- (3.6*\x-3.2,-2) ;
};
\draw (3.2,2) node {$0$} ;
\draw (6.8,2) node {$0$} ;
\draw (10.4,2) node {$0$} ;
%
\draw (1.4,0) node {$1$} ;
\draw (5,0) node {$0$} ;
\draw (8.6,0) node {$0$} ;
%
\draw (-0.4,-2) node {$0$} ;
\draw (3.2,-2) node {$0$} ;
\draw (5,-4) node {$\theta_2X$} ;
\end{tikzpicture}
\begin{tikzpicture}[scale=0.3, fl/.style={->,>=latex}]
\foreach \x in {0,1,2} {
   \draw[fl] (3.6*\x+1.8,0.5) -- (3.6*\x+2.8,1.5) ;
};
\foreach \x in {0,1} {
   \draw[fl] (3.6*\x,-1.5) -- (3.6*\x+1,-0.5) ;
};
\foreach \x in {1,2} {
   \draw[fl] (3.6*\x,1.5) -- (3.6*\x+1,0.5) ;
   \draw[fl, dotted] (3.6*\x+2.4,2) -- (3.6*\x+0.4,2) ;
   \draw[fl, dotted] (3.6*\x+0.6,0) -- (3.6*\x-1.4,0) ;
};
\foreach \x in {1} {
   \draw[fl] (3.6*\x-1.8,-0.5) -- (3.6*\x-0.8,-1.5) ;
   \draw[fl, dotted] (3.6*\x-1.2,-2) -- (3.6*\x-3.2,-2) ;
};
\draw (3.2,2) node {$0$} ;
\draw (6.8,2) node {$0$} ;
\draw (10.4,2) node {$0$} ;
%
\draw (1.4,0) node {$0$} ;
\draw (5,0) node {$0$} ;
\draw (8.6,0) node {$0$} ;
%
\draw (-0.4,-2) node {$0$} ;
\draw (3.2,-2) node {$0$} ;
\draw (5,-4) node {$\theta_iX\ (i\ge3)$} ;
\end{tikzpicture}\]
Summing up them, the map $Y\mapsto\mathrm{length}\underline{\mathscr{C}}(Y,X)_{{\rm End}_{\mathscr{C}}(Y)}$ is given by the following.
\[
\begin{tikzpicture}[scale=0.3, fl/.style={->,>=latex}]
\foreach \x in {0,1,2} {
   \draw[fl] (3.6*\x+1.8,0.5) -- (3.6*\x+2.8,1.5) ;
};
\foreach \x in {0,1} {
   \draw[fl] (3.6*\x,-1.5) -- (3.6*\x+1,-0.5) ;
};
\foreach \x in {1,2} {
   \draw[fl] (3.6*\x,1.5) -- (3.6*\x+1,0.5) ;
   \draw[fl, dotted] (3.6*\x+2.4,2) -- (3.6*\x+0.4,2) ;
   \draw[fl, dotted] (3.6*\x+0.6,0) -- (3.6*\x-1.4,0) ;
};
\foreach \x in {1} {
   \draw[fl] (3.6*\x-1.8,-0.5) -- (3.6*\x-0.8,-1.5) ;
   \draw[fl, dotted] (3.6*\x-1.2,-2) -- (3.6*\x-3.2,-2) ;
};
\draw (3.2,2) node {$1$} ;
\draw (6.8,2) node {$0$} ;
\draw (10.4,2) node {$0$} ;
%
\draw (1.4,0) node {$1$} ;
\draw (5,0) node {$1$} ;
\draw (8.6,0) node {$0$} ;
%
\draw (-0.4,-2) node {$0$} ;
\draw (3.2,-2) node {$1$} ;
\end{tikzpicture}
\]
\end{example}


We end this section with an extriangulated version of Reconstruction Theorem \cite[Proposition 5.1]{BG}\cite[Lemma 3.1]{IT2}. For a Krull--Schmidt category $\mathscr{D}$, its \emph{associated completely graded category} ${\rm Gr}\,\mathscr{D}$ has the same objects as $\mathscr{D}$, and the morphisms are given by
\[{\rm Gr}\,\mathscr{D}(X,Y)=\prod_{i\ge0}({\rm rad}^i\mathscr{D}/{\rm rad}^{i+1}\mathscr{D})(X,Y),\]
where the compositions are defined naturally.

\begin{corollary}[Reconstruction Theorem]\label{gr=mesh}\cite[Theorem 9.2]{I1}
Let $\mathscr{C}$ be a Krull--Schmidt extriangulated category which has enough projectives (resp.\ injectives), almost split extensions and source (resp.\ sink) morphisms, and let $\mathscr{D}=\underline{\mathscr{C}}$ (resp.\ $\overline{\mathscr{C}}$).
Then the category ${\rm Gr}\,\mathscr{D}$ is equivalent to the complete mesh category of $\mathrm{AR}_\tau^{\rm sp}(\mathscr{D})$ (see Example \ref{define AR species}, Definition \ref{define mesh category}).
\end{corollary}


\section{Inverse Problem for extriangulated categories}\label{section_inverse}

Let $\mathscr{C}$ be a Krull--Schmidt extriangulated category with almost split extensions. Then the Auslander--Reiten quiver of $\mathscr{C}$ has a structure of a $\tau$-quiver (Proposition \ref{AR quiver is translation}).
This section is devoted to studying the following Inverse Problem.

\begin{problem}[Inverse Problem]\label{inverse problem}
Let $Q$ be a locally finite symmetrizable $\tau$-quiver (see Definition \ref{define valued translation quiver}).
Does there exist a Krull--Schmidt extriangulated category $\mathscr{D}$ with almost split extensions satisfying $\mathrm{AR}_{\rm ET}(\mathscr{D})=Q$?
\end{problem}

Our main result below in this section gives a positive answer to Problem \ref{inverse problem}. Notice that an extriangulated category which is a $\tau$-category has sink morphisms and source morphisms by definition, and hence has almost split extensions by Lemma \ref{sink imply right ASE}.

\begin{theorem}\label{answer for inverse}
Let $Q$ be a locally finite symmetrizable $\tau$-quiver.
\begin{enumerate}[\rm(1)]
\item There exists a Krull--Schmidt extriangulated category $\mathscr{D}$ which is a $\tau$-category satisfying  $\mathrm{AR}_{\rm ET}(\mathscr{D})=Q$.
\item If $Q$ is strict, then there exists a Krull--Schmidt exact category $\mathscr{D}=(\mathscr{D},\mathbb{E},\mathfrak{s})$ which is a strict $\tau$-category satisfying $\mathrm{AR}_{\rm ET}(\mathscr{D})=\mathrm{AR}_\tau(\mathscr{D})=Q$ and the following conditions.
\begin{enumerate}[\rm(a)]
\item The $\mathbb{E}$-projective objects coincide with the $\tau$-projective objects, and the $\mathbb{E}$-injective objects coincide with the $\tau$-injective objects.
\item The almost split sequences are precisely the $\tau$-sequences.
\end{enumerate}
\end{enumerate}
\end{theorem}

Theorem \ref{answer for inverse}(1) follows easily from Theorem \ref{answer for inverse}(2) {by adding more projective-injective objects. This strategy is detailed below. Notice that a similar idea was used in \cite{KS1,KS2}.

The following result is a main step to prove Theorem \ref{answer for inverse}(2).



\begin{theorem}\label{strict case}
Let $\mathscr{D}$ be a strict $\tau$-category.
\begin{enumerate}[\rm(1)]
\item $\mathscr{D}$ has a structure $(\mathscr{D},\mathbb{E},\mathfrak{s})$ of an exact category whose conflations are the sequences $0\to A\xrightarrow{x}B\xrightarrow{y}C\to0$ satisfying the following equivalent conditions.
\begin{enumerate}[\rm(a)]
\item $0\to\mathscr{D}(-,A)\xrightarrow{x\circ-}\mathscr{D}(-,B)\xrightarrow{y\circ-}\mathscr{D}(-,C)\to F\to0$ is an exact sequence in $\Mod\mathscr{D}$ such that $F$ has finite length and $F(X)=0$ hold for each $\tau$-projective object $X\in\mathscr{D}$.
\item $0\to\mathscr{D}(C,-)\xrightarrow{-\circ y}\mathscr{D}(B,-)\xrightarrow{-\circ x}\mathscr{D}(A,-)\to G\to0$ is an exact sequence in $\Mod\mathscr{D}^\mathrm{op}$ such that $G$ has finite length and $G(X)=0$ hold for each $\tau$-injective object $X\in\mathscr{D}$.
\end{enumerate}
\item This exact structure satisfies the following conditions and hence $\mathrm{AR}_{\rm ET}(\mathscr{D})=\mathrm{AR}_\tau(\mathscr{D})$.
\begin{enumerate}[\rm(a)]
\item The $\mathbb{E}$-projective objects coincide with the $\tau$-projective objects, and the $\mathbb{E}$-injective objects coincide with the $\tau$-injective objects.
\item The almost split sequences are precisely the $\tau$-sequences.
\end{enumerate}
\end{enumerate}
\end{theorem}

To prove these results, we need preparations on the functor category.
Let $\mathscr{D}$ be a $\tau$-category. We consider the functors
\[{\rm Hom}_{\mathscr{D}}(-,\mathscr{D})\colon\Mod\mathscr{D}\longleftrightarrow\Mod\mathscr{D}^\mathrm{op}\colon{\rm Hom}_{\mathscr{D}^\mathrm{op}}(-,\mathscr{D}).\]
For $F\in\Mod\mathscr{D}$, ${\rm Hom}_{\mathscr{D}}(F,\mathscr{D})\in\Mod\mathscr{D}^\mathrm{op}$ is given by $({\rm Hom}_{\mathscr{D}}(F,\mathscr{D}))(X)={\rm Hom}_{\mathscr{D}}(F,\mathscr{D}(-,X))$ for each $X\in\mathscr{D}$. For each $i\ge0$, we also consider their derived functors
\[{\rm Ext}^i_{\mathscr{D}}(-,\mathscr{D})\colon\Mod\mathscr{D}\longleftrightarrow\Mod\mathscr{D}^\mathrm{op}\colon{\rm Ext}^i_{\mathscr{D}^\mathrm{op}}(-,\mathscr{D}).\]

Let $\mathscr{S}$ be the full subcategory of $\Mod\mathscr{D}$ consisting of all finite length $\mathscr{D}$-modules $F$ such that $F(X)=0$ for all $\tau$-projective objects $X\in\mathscr{D}$.
Dually, let $\mathscr{T}$ be the full subcategory of $\Mod\mathscr{D}^\mathrm{op}$ consisting of all finite length $\mathscr{D}^\mathrm{op}$-modules $G$ such that $G(X)=0$ for all $\tau$-injective objects $X\in\mathscr{D}$.

The following homological properties of strict $\tau$-categories play a key role.

\begin{proposition}\label{homological algebra}
Let $\mathscr{D}$ be a strict $\tau$-category.
\begin{enumerate}[\rm(1)]
\item $\mathscr{S}$ is a Serre subcategory of $\Mod\mathscr{D}$, and $\mathscr{T}$ is a Serre subcategory of $\Mod\mathscr{D}^\mathrm{op}$.
\item Each object $F\in\mathscr{S}$ has a projective resolution $0\to P_2\xrightarrow{} P_1\xrightarrow{} P_0\to F\to0$ such that $P_i\in\proj\mathscr{D}$ for each $i$ and satisfies ${\rm Ext}^i_{\mathscr{D}}(F,\mathscr{D})=0$ for each $i\neq2$ and ${\rm Ext}^2_{\mathscr{D}}(F,\mathscr{D})\in\mathscr{T}$.
\item Each object $G\in\mathscr{T}$ has a projective resolution $0\to P_2\xrightarrow{} P_1\xrightarrow{} P_0\to G\to0$ such that $P_i\in\proj\mathscr{D}^\mathrm{op}$ for each $i$ and satisfies ${\rm Ext}^i_{\mathscr{D}^\mathrm{op}}(G,\mathscr{D})=0$ for each $i\neq2$ and ${\rm Ext}^2_{\mathscr{D}^\mathrm{op}}(G,\mathscr{D})\in\mathscr{S}$. 
\item We have dualities
\[{\rm Ext}^2_{\mathscr{D}}(-,\mathscr{D})\colon\mathscr{S}\longleftrightarrow\mathscr{T}\colon{\rm Ext}^2_{\mathscr{D}^\mathrm{op}}(-,\mathscr{D}).\]
\end{enumerate}
\end{proposition}

\begin{proof}
(1) is clear.

(2) It suffices to consider the case when $F$ is a simple $\mathscr{D}$-module since each object in $\mathscr{S}$ has a finite filtration by simple $\mathscr{D}$-modules in $\mathscr{S}$.
Thus we can assume $F=S_C:=(\mathscr{D}/{\rm rad}\mathscr{D})(-,C)$ for an indecomposable non-$\tau$-projective object $C$.
Take a $\tau$-sequence $A\xrightarrow{}B\xrightarrow{}C$ in $\mathscr{D}$. Then we have a projective resolution
\[0\to\mathscr{D}(-,A)\to\mathscr{D}(-,B)\to\mathscr{D}(-,C)\to S_C\to0.\]
For $S^A:=(\mathscr{D}/{\rm rad}\mathscr{D})(A,-)$, the definition of $\tau$-sequences implies
\[{\rm Ext}^i_{\mathscr{D}}(S_C,\mathscr{D})=\left\{\begin{array}{ll}0&i\neq2\\ S^A&i=2 \end{array}\right.\ \mbox{ and }\ 
{\rm Ext}^i_{\mathscr{D}^\mathrm{op}}(S^A,\mathscr{D})=\left\{\begin{array}{ll}0&i\neq2\\ S_C&i=2. \end{array}\right.\]
Thus $F=S_C$ satisfies the desired conditions.

(3) is dual to (2), and (4) follows immediately from (2) and (3).
\end{proof}

\begin{proof}[Proof of Theorem \ref{strict case}]
(1) Immediate from Enomoto's construction of exact structure \cite[Theorem 2.7]{E} and Proposition \ref{homological algebra}.

(2) Immediate from the definition of conflations.
\end{proof}

Now we are ready to prove Theorem \ref{answer for inverse}.

\begin{proof}[Proof of Theorem \ref{answer for inverse}]
(2) By Lemma \ref{lift quiver to species}, there exists a $\tau$-species ${\mathcal Q}$ whose underlying $\tau$-quiver is $Q$. By Propositions \ref{property of mesh category} and \ref{strict}(3)$\Rightarrow$(1), $\mathscr{D}:=\mathsf{M}({\mathcal Q})$ is a strict $\tau$-category satisfying 
$\mathrm{AR}_\tau(\mathscr{D})=Q$. By Theorem \ref{strict case}, $\mathscr{D}$ has a structure of an exact category satisfying $\mathrm{AR}_{\rm ET}(\mathscr{D})=\mathrm{AR}_\tau(\mathscr{D})=Q$ and the conditions (a)(b).

(1) Using $Q=(Q_0,d,d',\tau)$, define a new $\tau$-quiver $\widetilde{Q}=(\widetilde{Q}_0,\widetilde{d},\widetilde{d}',\widetilde{\tau})$ as follows:
Let $S:=Q_0\setminus Q_0^p$, and let
\[\widetilde{Q}_0:=Q_0\sqcup S,\ \widetilde{Q}_0^p:=Q_0^p\sqcup S,\ \widetilde{Q}_0^i:=Q_0^i\sqcup S\  \mbox{ and }\ \widetilde{\tau}:=\tau.\]
For $X\in Q_0\setminus Q_0^p$, we denote by $\widetilde{X}\in S$ the corresponding element. Define maps $\widetilde{d},\widetilde{d}':\widetilde{Q}_0\times\widetilde{Q}_0\to\mathbb{Z}_{\ge0}$ by
\begin{enumerate}[$\bullet$]
\item $\widetilde{d}|_{Q_0\times Q_0}:=d$, $\widetilde{d}'|_{Q_0\times Q_0}:=d'$.
\item $\widetilde{d}_{\tau X,\widetilde{X}}=\widetilde{d}'_{\tau X,\widetilde{X}}=\widetilde{d}_{\widetilde{X}X}=\widetilde{d}'_{\widetilde{X}X}:=1$ for each $X\in Q_0\setminus Q_0^p$.
\item $\widetilde{d}_{XY}=\widetilde{d}'_{XY}:=0$ for all other pairs $(X,Y)\in\widetilde{Q}_0\times\widetilde{Q}_0$.
\end{enumerate}
Then any element in $\widetilde{Q}_0\setminus \widetilde{Q}_0^p$ is a target of an arrow starting at an element in $\widetilde{Q}_0^p$. Thus $\widetilde{Q}$ is strict by applying Proposition \ref{strict}(4)$\Rightarrow$(3) to $\widetilde{Q}$ and its mesh category. 
By (2), there exists an exact category $\mathscr{C}$ which is a strict $\tau$-category satisfying $\mathrm{AR}_{\rm ET}(\mathscr{C})=\mathrm{AR}_\tau(\mathscr{C})=\widetilde{Q}$.
Let $\mathscr{B}:={\rm add} S$ and $\mathscr{D}:=\mathscr{C}/\mathscr{B}$. By Proposition \ref{PropQuotARExt}, $\mathscr{D}$ is an extriangulated category satisfying $\mathrm{AR}_{\rm ET}(\mathscr{D})=Q$. The proof is completed.
\end{proof}

\begin{remark}
Notice that the exact categories given in Theorems \ref{answer for inverse} and \ref{strict case} do not necessarily have enough projectives and enough injectives.
\end{remark}

For a symmetrizable $\tau$-quiver which is stable (see Definition \ref{define valued translation quiver}), there is a different approach to Problem \ref{inverse problem}.
Recall that by Riedtmann's Structure Theorem \cite[p.\ 206]{Rie}, each stable $\tau$-quiver $Q$ can be written as $\mathbb{Z} Q'/G$, where $Q'$ is a valued quiver which is a tree and $G$ is a weakly admissible automorphism group of the $\tau$-quiver $\mathbb{Z} Q'$.
We call $Q'$ the \emph{tree type} of $Q$.

\begin{example}\label{stable case}
Let $Q$ be a locally finite stable symmetrizable $\tau$-quiver.
\begin{enumerate}[\rm(1)]
\item There exists a Krull--Schmidt extriangulated category $\mathscr{D}$ such that $\mathrm{AR}_{\rm ET}(\mathscr{D})=Q$.
\item If the tree type of $Q$ is a disjoint union of non-Dynkin quivers, then there exists an exact category $\mathscr{D}$ such that $\mathrm{AR}_{\rm ET}(\mathscr{D})=Q$.
\item If the tree type of $Q$ is a disjoint union of Dynkin quivers, then there exists a triangulated category $\mathscr{D}$ such that $\mathrm{AR}_{\rm ET}(\mathscr{D})=Q$.
\end{enumerate}
\end{example}

\begin{proof}
(2) Since $Q$ is strict by Example \ref{example of theta_i}(1), the assertion follows from Theorem \ref{answer for inverse}.

(3) Write $Q=\mathbb{Z} Q'/G$ for a disjoint union $Q'$ of Dynkin quivers and a weakly admissible automorphism group $G$ of $\mathbb{Z} Q'$. Let $\mathcal{Q}'$ be a species whose underlying valued quiver is $Q'$, and $\Lambda$ the opposite of the tensor algebra of $\mathcal{Q}'$. Then $\mathrm{AR}_{\rm ET}(\operatorname{D^b}(\mod\Lambda))=\mathbb{Z} Q'$ holds. If $G=\{1\}$, then the claim follows. Otherwise, $G$ is generated by the action of an autoequivalence $F:\operatorname{D^b}(\mod\Lambda)\to\operatorname{D^b}(\mod\Lambda)$ on $\mathbb{Z} Q'$. By \cite[Theorem 1]{Ke2}, the orbit category $\mathscr{D}:=\operatorname{D^b}(\mod\Lambda)/F$ has a structure of a triangulated category. Clearly $\mathrm{AR}_{\rm ET}(\mathscr{D})=\mathbb{Z} Q'/G=Q$ holds.

(1) follows immediately from (2) and (3).
\end{proof}


We end this section with the following general question.

\begin{problem}\label{lifting problem}
Let $\mathscr{D}$ be a $\tau$-category. Does there exist an extriangulated structure on $\mathscr{D}$ such that the projective objects coincide with the $\tau$-projective objects, and the injective objects coincide with the $\tau$-injective objects?
\end{problem}

\section{An example from gentle algebras}\label{subsection: example3}

This example is motivated by~\cite[Figure 30]{PPP}.
Let $k$ be a field, let $A$ be the quotient of the path algebra of the $\mathrm{A}_3$ quiver $1\rightarrow 2\rightarrow 3$ with ideal of relations $\operatorname{rad}^2$, and let $A\blossom$ be the algebra given by the quiver: 
\begin{equation}\label{Quiver_Example3}
\xymatrix@-1pc{
& & & f & \\
& & d \ar@{->}[r]_h & 3 \ar@{->}[r]_h \ar@{->}[u]^v  & h \\
& c & & & \\
a \ar@{->}[r]^h & 1 \ar@{->}[rr]^h \ar@{->}[u]_v & & 2 \ar@{->}[r]^h  \ar@{->}[uu]_v & g \\
& b \ar@{->}[u]_v & & e  \ar@{->}[u]^v &
}
\end{equation}
with relations $hv$, $vh$.

We let:
\begin{enumerate}[$\bullet$]
 \item $S=P_c\oplus P_f\oplus P_g\oplus P_h$ be the sum of the simple projective modules over $A\blossom$;
 \item $Q = P_a\oplus P_b\oplus P_d\oplus P_e$ be the sum of the indecomposable projective-injective modules over $A\blossom$;
 \item $\mathscr{E}$ be the full subcategory of $\mod A\blossom$ whose objects are all those modules $M$ such that both $\mathrm{Hom}_{A\blossom}(S,M)$ and $\mathrm{Hom}_{A\blossom}(Q,\tau M)$ vanish.
\end{enumerate}

Then $\mathrm{AR}_{\rm ET}(\mod A\blossom)$ is shown in Figure~\ref{figure: ARquiverModEx3}.
Since, for any modules $M,N$ over $A\blossom$, $\mathrm{Hom}_{A\blossom}(M,\tau N)=0$ if and only if $\mathrm{Ext}^1_{A\blossom}(N,\operatorname{Fac} M)=0$ \cite[Proposition 5.8]{AS}, the subcategory $\mathscr{E}$ is extension-closed in $\mod A\blossom$.

If we denote the full subcategory of projective-injective (resp.\ projective, injective) objects in $\mathscr{E}$ by $\mathscr{B}$ (resp.\ $\mathscr{P}$, $\mathscr{I}$) then by \cite[Proposition 3.30]{NP}, the ideal quotient $\mathscr{E}/\mathscr{B}$ is extriangulated. By Proposition~\ref{PropQuotARExt}, $\mathscr{E}/\mathscr{B}$ has almost split extensions, and $\mathrm{AR}_{\rm ET}(\mathscr{E}/\mathscr{B})$ is depicted in Figure~\ref{figure: ARquiverEx3}.
The Auslander--Reiten translation is represented by dashed arrows.
We note that the extriangulated category $\mathscr{E}/\mathscr{B}$ is not an exact category since the morphism $\bsm2\\3\esm \longrightarrow \bsm d\phantom{3}2\\3\esm$ is (both) an inflation  which is not monic (and a deflation which is not epic).
We also note that the isoclasses of indecomposable objects of $\mathscr{E}/\mathscr{B}$ are in bijection with the isoclasses of indecomposable $A$-modules or shifted projectives.
Moreover, the injective objects in $\mathscr{E}/\mathscr{B}$ are the shifted projectives and the projective objects in $\mathscr{E}/\mathscr{B}$ are the projective $A$-modules.
Let $A\ast A[1]$ be the full subcategory of the homotopy category $\kbproj$ whose objects are the complexes concentrated in degrees 0 and -1.
Via the bijection of \cite[Theorem 4.1]{AIR}, the modules $P_3$, $P_2$, $S_2$, $P_1$, and $S_1$ are respectively sent to the complexes $0\rightarrow P_3$, $0\rightarrow P_2$, $P_3\rightarrow P_2$, $0 \rightarrow P_1$ and $P_2 \rightarrow P_1$ so that there is a bijection between isoclasses of indecomposable objects in $\mathscr{E}/\mathscr{B}$ and isoclasses of indecomposable objects in $A\ast A[1]$.

\begin{center}
\begin{figure}
\begin{tikzpicture}[scale=.5, fl/.style={->,>=latex}]
\draw[fl] (-6,-1) -- (-5,-2) ;
\draw[fl] (-4,-2) -- (-3,-1) ;
\draw[fl] (-2,0) -- (-1,1) ;
\draw[fl] (-2,-1) -- (-1,-2) ;
\begin{scope}[xshift=-1cm, yshift=-1cm, rotate=180, fl/.style={<-,>=latex}] 
\draw[fl] (-6,-1) -- (-5,-2) ;
\draw[fl] (-4,-2) -- (-3,-1) ;
\draw[fl] (-2,0) -- (-1,1) ;
\draw[fl] (-2,-1) -- (-1,-2) ;
\end{scope}
\begin{scope}[yshift=-.5cm]
\draw[fl, dashed] (4.5,0) -- (2.5,0) ;
\draw[fl, dashed] (.5,0) -- (-1.5,0) ;
\draw[fl, dashed] (-3.5,0) -- (-5.5,0) ;
\draw[fl, dashed] (2.5,2) -- (.5,2) ;
\draw[fl, dashed] (-1.5,-2) -- (-3.5,-2) ;
\end{scope}
\begin{scope}[xshift=-.5cm, yshift=-.5cm]
\draw (-6,0) node {$3$} ;
\draw (-4,-2) node {$\bsm2\\3\esm$} ;
\draw (-2,0) node {$2$} ;
\draw (0,-2) node {$3[1]$} ;
\draw (0,2) node {$\bsm1\\2\esm$} ;
\draw (2,0) node {$1$} ;
\draw (4,2) node {$\bsm2\\3\esm[1]$} ;
\draw (6,0) node {$\bsm1\\2\esm[1]$} ;
\end{scope}
\end{tikzpicture}
\caption{$\mathrm{AR}_{\rm ET}(\mathscr{E}/\mathscr{B})$: The vertices are labelled by the corresponding indecomposable $A$-modules or shifted projectives.}\label{figure: ARquiverEx3}
\end{figure}
\end{center}

We give a brief explanation on how $\mathscr{E}/\mathscr{B}$ can be related to $\mod A$. Let $\varepsilon_p\in A\blossom$ be the idempotent element corresponding to each vertex $p$ in $(\ref{Quiver_Example3})$. Put
\[ \varepsilon=1-(\varepsilon_1+\varepsilon_2+\varepsilon_3)=\varepsilon_a+\varepsilon_b+\varepsilon_c+\varepsilon_d+\varepsilon_e+\varepsilon_f+\varepsilon_g+\varepsilon_h, \]and denote the two-sided ideal $A\blossom\varepsilon A\blossom\subseteq A\blossom$ by $\mathfrak{a}$.
The functor tensoring with $A=A\blossom/\mathfrak{a}$
\[ G\colon\mod A\blossom\to\mod A\ ;\ M\mapsto M/\mathfrak{a} M \]
satisfies $G(\mathscr{I})=0$, and thus induces an additive functor $\overline{G}\colon\mathscr{E}/\mathscr{I}\to\mod A$ (recall that $\mathscr{I}$ is the full subcategory of those objects that are injective in $\mathscr{E}$).

For any indecomposable object $N\in\mathscr{E}$ which is not injective in $\mathscr{E}$, we can observe on $\mathrm{AR}_{\rm ET}(\mod A\blossom)$ that there are exact sequences in $\mod A\blossom$
\[ 0\to\mathfrak{a} N\to N\to N/\mathfrak{a} N\to0,\quad 0\to K\to I\to \mathfrak{a} N\to0 \]
satisfying $I\in\mathscr{I}$ and $K,\mathfrak{a} N\in\operatorname{Fac} Q$.
More explicitly, the only non-injective objects $N$ in $\mathscr{E}$ for which $\mathfrak{a}N$ is non-zero are $\bsm d\phantom{3}2\\3\esm$ and $\bsm 1 \phantom{2} e\\ 2\esm$.
Moreover, we have short exact sequences:
\[
0\to \bsm d\\3\esm \to \bsm d\phantom{3}2\\3\esm \to 2\to 0,\;\;\;
0\to \bsm e\\2\esm \to \bsm 1\phantom{2}e\\2\esm \to 1\to 0,\;\;\;
0\to \bsm d\\3\esm \to \bsm \phantom{d3}e\\d\phantom{3}2\\3\esm \to \bsm e\\2\esm \to 0,
\]
where $\bsm d\\3\esm$ and $\bsm \phantom{d3}e\\d\phantom{3}2\\3\esm$ are injective in $\mathscr{E}$,
$\bsm d\\3\esm\in\operatorname{Fac}\bsm d\\3\\h\esm$ and
$\bsm e\\2\esm\in\operatorname{Fac}\bsm e\\2\\3\\f\esm$.

For any $M\in\mathscr{E}$, by the Auslander--Reiten duality, we have $\operatorname{Ext}^1_{A\blossom}(M,\mathfrak{a}N)\cong \Dd\overline{\operatorname{Hom}}_{A\blossom}(\mathfrak{a}N,\tau M) =0$
and $\operatorname{Ext}^1_{A\blossom}(M,K)\cong \Dd\overline{\operatorname{Hom}}_{A\blossom}(K,\tau M) =0$ because $K,\mathfrak{a}N\in\operatorname{Fac}Q$ and $\operatorname{Hom}(Q,\tau M)=0$.
This shows that the short exact sequences above induce surjections:

\begin{eqnarray*}
&\mathrm{Hom}_{A\blossom}(M,N)\to\mathrm{Hom}_{A\blossom}(M,N/\mathfrak{a} N)\simeq\mathrm{Hom}_A(G(M),G(N)),&\\
&\mathrm{Hom}_{A\blossom}(M,I)\to\mathrm{Hom}_{A\blossom}(M,\mathfrak{a} N)\simeq\mathrm{Ker}\big(\mathrm{Hom}_{A\blossom}(M,N)\overset{G}{\longrightarrow}\mathrm{Hom}_A(G(M),G(N))\big),&
\end{eqnarray*}
which show that
\[ (\mathscr{E}/\mathscr{I})(M,N)\overset{\overline{G}}{\longrightarrow}\mathrm{Hom}_A(\overline{G}(M),\overline{G}(N)) \]
is bijective. Thus $\overline{G}$ is fully faithful.

Moreover, since non-injective indecomposable objects in $\mathscr{E}$ satisfy
\[ 3\simeq G(3),\ \ \bsm2\\3\esm\simeq G(\bsm2\\3\esm),\ \ 2\simeq G(\bsm d\phantom{3}2\\3\esm),\ \ \bsm1\\2\esm\simeq G(\bsm1\\2\esm),\ \ 1\simeq G(\bsm 1\phantom{2}e\\2\esm), \]
we can observe that $\overline{G}$ is essentially surjective, and thus $\overline{G}\colon\mathscr{E}/\mathscr{I}\to\mod A$ is an equivalence of categories.

Explicit computations detailed below show the following:

\medskip\noindent {\bf Claim}: \emph{The bijection in Figure~\ref{figure: ARquiverEx3} induces a bijection between basic support $\tau$-tilting modules (equivalently: basic $\tau$-tilting pairs) over $A$ and basic maximal $\mathbb{E}$-rigid objects in the extriangulated category $\mathscr{E}/\mathscr{B}$. Moreover, $\tau$-tilting mutation can be computed in $\mathscr{E}/\mathscr{B}$ by means of approximation extriangles.}

\begin{proof}
Up to isomorphism, the extriangles in $\mathscr{E}/\mathscr{B}$ are the trivial ones, the ten extriangles listed below and all their direct sums. One can then check that $\tau$-compatibility of $\tau$-rigid pairs in $\mod A$ corresponds to $\mathbb{E}$-rigidity in $\mathscr{E}/\mathscr{B}$. As shown in Figure~\ref{figure: mutation}, maximal $\mathbb{E}$-rigid objects in $\mathscr{E}/\mathscr{B}$ have a well-behaved theory of mutation and mutation of $\tau$-tilting pairs in $\mod A$ correspond to mutation of maximal $\mathbb{E}$-rigid objects in $\mathscr{E}/\mathscr{B}$.

\[
\xymatrix@R=3pt{
(1) & 3 \ar@{->}[r] & ^2_3 \ar@{->}[r] & 2 \ar@{-->}[r] &&&
(2) & ^2_3 \ar@{->}[r] & 2 \ar@{->}[r] & 3[1] \ar@{-->}[r] & \\
(3) & 2 \ar@{->}[r] & ^1_2\oplus 3[1] \ar@{->}[r] & 1 \ar@{-->}[r] &&&
(4) & 2 \ar@{->}[r] & 3[1] \ar@{->}[r] & ^2_3[1] \ar@{-->}[r] & \\
(5) & ^1_2 \ar@{->}[r] & 1 \ar@{->}[r] & ^2_3[1] \ar@{-->}[r] &&&
(6) & ^1_2 \ar@{->}[r] & 0 \ar@{->}[r] & ^1_2[1] \ar@{-->}[r] & \\
(7) & 3 \ar@{->}[r] & 0 \ar@{->}[r] & 3[1] \ar@{-->}[r] &&&
(8) & ^2_3 \ar@{->}[r] & ^1_2 \ar@{->}[r] & 1 \ar@{-->}[r] & \\
(9) & ^2_3 \ar@{->}[r] & 0 \ar@{->}[r] & ^2_3[1] \ar@{-->}[r] &&&
(10) & 1 \ar@{->}[r] & ^2_3[1] \ar@{->}[r] & ^1_2[1] \ar@{-->}[r] &
}
\]
\end{proof}

\begin{figure}
\begin{tikzpicture}
\coordinate (0) at (0,-1) ;
\coordinate (1) at (0,2) ;
\coordinate (2) at (-5,3) ;
\coordinate (3) at (4.5,3) ;
\coordinate (4) at (-2,4) ;
\coordinate (5) at (2,4) ;
\coordinate (6) at (4.5,5.5) ;
\coordinate (7) at (0,6) ;
\coordinate (8) at (6.5,6) ;
\coordinate (9) at (6.5,8.5) ;
\coordinate (10) at (-5,7.5) ;
\coordinate (11) at (0,10) ;
\draw (0) -- (2) ;
\draw (0) -- (1) ;
\draw (0) -- (3) ;
\draw (2) -- (10) ;
\draw (2) -- (4) ;
\draw (1) -- (4) ;
\draw (1) -- (5) ;
\draw (3) -- (6) ;
\draw (3) -- (8) ;
\draw (4) -- (7) ;
\draw (5) -- (7) ;
\draw (5) -- (6) ;
\draw (6) -- (9) ;
\draw (7) -- (11) ;
\draw (8) -- (9) ;
\draw (9) -- (11) ;
\draw (10) -- (11) ;
\draw[dashed] (10) -- (8) ;
\foreach \x in {0,1,...,11} {
\draw[white, fill=white] (\x) circle (.6cm) ;
} ;
\draw (0) node {$\bsm1\\2\esm[1]\oplus\bsm2\\3\esm[1]\oplus 3[1]$} ;
\draw (1) node {$\bsm2\\3\esm[1]\oplus 1\oplus 3[1]$} ;
\draw (2) node {$\bsm1\\2\esm[1]\oplus\bsm2\\3\esm[1]\oplus 3$} ;
\draw (3) node {$\bsm1\\2\esm[1]\oplus 3[1]\oplus 2$} ;
\draw (4) node {$\bsm2\\3\esm[1]\oplus 1\oplus 3$} ;
\draw (5) node {$\bsm1\\2\esm\oplus 1\oplus 3[1]$} ;
\draw (6) node {$\bsm1\\2\esm\oplus 3[1]\oplus 2$} ;
\draw (7) node {$\bsm1\\2\esm\oplus 1\oplus 3$} ;
\draw (8) node {$\bsm1\\2\esm[1]\oplus\bsm2\\3\esm\oplus 2$} ;
\draw (9) node {$\bsm1\\2\esm\oplus\bsm2\\3\esm\oplus 2$} ;
\draw (10) node {$\bsm1\\2\esm[1]\oplus\bsm2\\3\esm\oplus 3$} ;
\draw (11) node {$\bsm1\\2\esm\oplus\bsm2\\3\esm\oplus 3$} ;
\draw[white, fill=white] (0,0.5) circle (.3cm) ;
\draw[blue] (0,0.5) node {(10)} ;
\draw[white, fill=white] (-5/2,1) circle (.3cm) ;
\draw[blue] (-5/2,1) node {(1)} ;
\draw[white, fill=white] (2.25,1) circle (.3cm) ;
\draw[blue] (2.25,1) node {(4)} ;
\draw[white, fill=white] (-5,5.25) circle (.3cm) ;
\draw[blue] (-5,5.25) node {(9)} ;
\draw[white, fill=white] (-3.5,3.5) circle (.3cm) ;
\draw[blue] (-3.5,3.5) node {(10)} ;
\draw[white, fill=white] (-1,3) circle (.3cm) ;
\draw[blue] (-1,3) node {(7)} ;
\draw[white, fill=white] (1,3) circle (.3cm) ;
\draw[blue] (1,3) node {(5)} ;
\draw[white, fill=white] (4.5,4.25) circle (.3cm) ;
\draw[blue] (4.5,4.25) node {(6)} ;
\draw[white, fill=white] (5.5,4.5) circle (.3cm) ;
\draw[blue] (5.5,4.5) node {(2)} ;
\draw[white, fill=white] (-1,5) circle (.3cm) ;
\draw[blue] (-1,5) node {(5)} ;
\draw[white, fill=white] (1,5) circle (.3cm) ;
\draw[blue] (1,5) node {(7)} ;
\draw[white, fill=white] (3.25,4.75) circle (.3cm) ;
\draw[blue] (3.25,4.75) node {(3)} ;
\draw[white, fill=white] (5.5,7) circle (.3cm) ;
\draw[blue] (5.5,7) node {(2)} ;
\draw[white, fill=white] (0,8) circle (.3cm) ;
\draw[blue] (0,8) node {(8)} ;
\draw[white, fill=white] (6.5,7.25) circle (.3cm) ;
\draw[blue] (6.5,7.25) node {(6)} ;
\draw[white, fill=white] (3.25,9.25) circle (.3cm) ;
\draw[blue] (3.25,9.25) node {(1)} ;
\draw[white, fill=white] (-2.5,8.75) circle (.3cm) ;
\draw[blue] (-2.5,8.75) node {(6)} ;
\draw[white, fill=white] (1.25,6.7) circle (.3cm) ;
\draw[blue] (1.25,6.7) node {(1)} ;
\end{tikzpicture}
\caption{The poset of basic maximal $\mathbb{E}$-rigid objects in the extriangulated category $\mathscr{E}/\mathscr{B}$. Vertices are representatives for the isoclasses of basic maximal $\mathbb{E}$-rigid objects in $\mathscr{E}/\mathscr{B}$. Two vertices are linked by an edge if and only if the corresponding $\mathbb{E}$-rigid objects differ by one indecomposable summand. Edges are labelled with the corresponding approximation extriangle.}\label{figure: mutation}
\end{figure}

\begin{landscape}
\begin{figure}
\begin{tikzpicture}[scale=0.65, fl/.style={->,>=latex}]
\foreach \x in {0,1,...,4} { 
  \draw[fl] (2*\x-14,-2*\x+2) -- (2*\x-13,-2*\x+1) ;
  \draw[fl] (2*\x-8,-2*\x+4) -- (2*\x-7,-2*\x+3) ;
};
\foreach \x in {0,1,2,3} {
  \draw[fl] (2*\x-10,-2*\x+2) -- (2*\x-9,-2*\x+1) ;
  \draw[fl] (2*\x-10,2*\x-5) -- (2*\x-9,2*\x-4) ;
};
\foreach \x in {0,1,2} {
  \draw[fl] (2*\x-14,2*\x-1) -- (2*\x-13,2*\x) ;
  \draw[fl] (2*\x-6,2*\x-5) -- (2*\x-5,2*\x-4) ;
  \draw[fl] (2*\x-4,2*\x-7) -- (2*\x-3,2*\x-6) ;
};
\foreach \x in {0,1} {
  \draw[fl] (2*\x-10,2*\x-1) -- (2*\x-9,2*\x) ;
  \draw[fl] (2*\x-2,-2*\x+2) -- (2*\x-1,-2*\x+1) ;
};
\begin{scope}[yshift=1cm, xshift=-1cm, rotate=180, fl/.style={<-,>=latex}]
 \foreach \x in {0,1,...,4} { 
  \draw[fl] (2*\x-14,-2*\x+2) -- (2*\x-13,-2*\x+1) ;
  \draw[fl] (2*\x-8,-2*\x+4) -- (2*\x-7,-2*\x+3) ;
};
\foreach \x in {0,1,2,3} {
  \draw[fl] (2*\x-10,-2*\x+2) -- (2*\x-9,-2*\x+1) ;
  \draw[fl] (2*\x-10,2*\x-5) -- (2*\x-9,2*\x-4) ;
};
\foreach \x in {0,1,2} {
  \draw[fl] (2*\x-14,2*\x-1) -- (2*\x-13,2*\x) ;
  \draw[fl] (2*\x-6,2*\x-5) -- (2*\x-5,2*\x-4) ;
  \draw[fl] (2*\x-4,2*\x-7) -- (2*\x-3,2*\x-6) ;
};
\foreach \x in {0,1} {
  \draw[fl] (2*\x-10,2*\x-1) -- (2*\x-9,2*\x) ;
};
\end{scope}
\begin{scope}[xshift=-0.5cm, yshift=0.5cm]
\draw (0,0) node {$2$} ;
\draw (-14,-2) node {$h$} ;
\draw (14,2) node {$b$} ;
\draw (-14,2) node {$f$} ;
\draw (14,-2) node {$a$} ;
\draw (-12,0) node {$\bsm3\\f\;h\esm$} ;
\draw (12,0) node {$\bsm a\;b\\1\esm$} ;
\draw[blue] (12,0) circle (.5cm) ;
\draw (-10,-6) node {$g$} ;
\draw (10,6) node {$e$} ;
\draw (-10,-2) node {$\bsm3\\f\esm$} ;
\draw (10,2) node {$\bsm a\\1\esm$} ;
\draw (-10,2) node {$\bsm3\\h\esm$} ;
\draw (10,-2) node {$\bsm b\\1\esm$} ;
\draw[blue] (10,-2) ellipse (.4cm and .6cm) ;
\draw (-8,-4) node {$\bsm 2\\3\;g\\f\;\phantom{\;g}\esm$} ;
\draw (8,4) node {$\bsm a\;\phantom{\;e}\\1\;e\\2\esm$} ;
\draw[blue] (8,4.05) ellipse (.55cm and .7cm) ;
\draw (-8,0) node {$3$} ;
\draw[blue] (-8,0) circle (.5cm) ;
\draw (8,0) node {$1$} ;
\draw (-8,4) node {$\bsm d\\3\\h\esm$} ;
\draw (8,-4) node {$\bsm b\\1\\c\esm$} ;
\draw (-6,-6) node {$\bsm2\\3\\f\esm$} ;
\draw (6,6) node {$\bsm a\\1\\2\esm$} ;
\draw[blue] (6,6) ellipse (.3cm and .7cm) ;
\draw (-6,-2) node {$\bsm2\\3\;g\esm$} ;
\draw (6,2) node {$\bsm1\;e\\2\esm$} ;
\draw[blue] (6,2) circle (.55cm) ;
\draw (-6,2) node {$\bsm d\\3\esm$} ;
\draw[blue] (-6,2) ellipse (.4cm and .6cm) ;
\draw (6,-2) node {$\bsm1\\c\esm$} ;
\draw (-4,-8) node {$\bsm e\\2\\3\\f\esm$} ;
\draw (4,8) node {$\bsm a\\1\\2\\g\esm$} ;
\draw (-4,-4) node {$\bsm 2\\3\esm$} ;
\draw[blue] (-4,-4) ellipse (.4cm and .6cm) ;
\draw (4,4) node {$\bsm 1\\2\esm$} ;
\draw[blue] (4,4) ellipse (.4cm and .6cm) ;
\draw (-4,0) node {$\bsm d\phantom{3}2\\\phantom{d\;}\,3\phantom{2}g\esm$} ;
\draw (4,0) node {$\bsm \phantom{c\;}1\phantom{2}e\\c\phantom{1}2\esm$} ;
\draw (-2,-6) node {$\bsm e\\2\\3\esm$} ;
\draw[blue] (-2,-6) ellipse (.3cm and .7cm) ;
\draw (2,6) node {$\bsm 1\\2\\g\esm$} ;
\draw (-2,-2) node {$\bsm d\phantom{3}2\\3\esm$} ;
\draw[blue] (-2,-2) circle (.55cm) ;
\draw (2,2) node {$\bsm1\\c\;\,2\esm$} ;
\draw (-2,2) node {$\bsm 2\\g\esm$} ;
\draw (2,-2) node {$\bsm e\\2\esm$} ;
\draw (-2,6) node {$c$} ;
\draw (2,-6) node {$d$} ;
\draw (0,-4) node {$\bsm \phantom{d3}e\\d\phantom{3}2\\3\esm$} ;
\draw[blue] (0,-4) ellipse (.55cm and .7cm) ;
\draw (0,4) node {$\bsm 1\\c\phantom{3}2\\\phantom{c3}g\esm$} ;
\draw[red] (14.3,2.5) -- (14.3,1.5) ;
\draw[red] (14.3,-1.5) -- (14.3,-2.5) ;
\draw[red] (12.7,0.5) -- (12.7,-0.5) ;
\draw[red] (8.8,4.5) -- (8.8,3.5) ;
\draw[red] (10.3,6.5) -- (10.3,5.5) ;
\draw[red] (8.3,-3.4) -- (8.3,-4.6) ;
\draw[red] (7.7,-3.4) -- (7.7,-4.6) ;
\draw[red] (3.7,8.7) -- (3.7,7.3) ;
\draw[red] (4.3,8.7) -- (4.3,7.3) ;
\draw[red] (2.3,-5.5) -- (2.3,-6.5) ;
\draw[red] (-0.4,4.7) -- (-0.4,3.3) ;
\begin{scope}[rotate=180]
\draw[red] (14.3,2.5) -- (14.3,1.5) ;
\draw[red] (14.3,-1.5) -- (14.3,-2.5) ;
\draw[red] (12.37,0.6) -- (12.37,-0.6) ;
\draw[red] (8.4,4.7) -- (8.4,3.3) ;
\draw[red] (10.3,6.5) -- (10.3,5.5) ;
\draw[red] (8.3,-3.4) -- (8.3,-4.6) ;
\draw[red] (7.7,-3.4) -- (7.7,-4.6) ;
\draw[red] (3.7,8.7) -- (3.7,7.3) ;
\draw[red] (4.3,8.7) -- (4.3,7.3) ;
\draw[red] (2.3,-5.5) -- (2.3,-6.5) ;
\draw[red] (-0.7,4.5) -- (-0.7,3.5) ;
\end{scope}
\draw[dashed, thick, blue] (3.4,4.45) -- (3.4,3.45) ;
\draw[dashed, thick, blue] (5.6,6.6) -- (5.6,5.3) ;
\draw[dashed, thick, blue] (6.4,6.6) -- (6.4,5.3) ;
\draw[dashed, thick, blue] (9.45,-1.6) -- (9.45,-2.7) ;
\draw[dashed, thick, blue] (9,4.5) -- (9,3.4) ;
\draw[dashed, thick, blue] (10.55,-1.6) -- (10.55,-2.7) ;
\draw[dashed, thick, blue] (12.9,0.5) -- (12.9,-0.6) ;
\begin{scope}[xshift=4cm, rotate=180]
\draw[dashed, thick, blue] (3.1,4.5) -- (3.1,3.4) ;
\draw[dashed, thick, blue] (5.6,6.6) -- (5.6,5.3) ;
\draw[dashed, thick, blue] (6.4,6.6) -- (6.4,5.3) ;
\draw[dashed, thick, blue] (9.45,-1.6) -- (9.45,-2.7) ;
\draw[dashed, thick, blue] (8.6,4.5) -- (8.6,3.4) ;
\draw[dashed, thick, blue] (10.55,-1.6) -- (10.55,-2.7) ;
\draw[dashed, thick, blue] (12.7,0.5) -- (12.7,-0.5) ;
\end{scope}
\end{scope}
\end{tikzpicture}
\caption{$\mathrm{AR}_{\rm ET}(\mod A\blossom)$: The indecomposable objects in $\mathscr{E}$ are circled in blue; the projective (resp.\ injective) modules are highlighted by a red vertical line to their left (resp.\ to their right) and similarly for the relative projectives (resp.\ injectives) in $\mathscr{E}$ by using dashed blue vertical lines.}\label{figure: ARquiverModEx3}
\end{figure}
\restoregeometry
\end{landscape}

\end{document}